\documentclass[12pt]{article}%

\usepackage{amsmath}
\usepackage{amssymb,amsfonts,amsthm}
\usepackage{fancyhdr}
\usepackage{tabularx}
\usepackage{dsfont}
\usepackage{mathtools}
\usepackage{authblk}
\usepackage{xr}
\usepackage[colorlinks = True]{hyperref}
\hypersetup{
    colorlinks = true,
    linkcolor = blue,
    anchorcolor = blue,
    citecolor = blue,
    filecolor = blue,
    urlcolor = blue
    }
\usepackage[a4paper, top=2.5cm, bottom=2.5cm, left=1.5cm, right=1.5cm]%
{geometry}
\usepackage{times}
\usepackage[table,xcdraw]{xcolor}
\usepackage[capitalize, nameinlink]{cleveref}
\usepackage{changepage}
\usepackage{enumitem}
\usepackage{mathrsfs}
\usepackage{tcolorbox}
\usepackage{stackrel,amssymb}
\usepackage{graphicx}
\usepackage{multicol}

\usepackage{appendix}

\usepackage[backend=biber,
style=alphabetic,
minalphanames=3,
maxalphanames=3,
maxnames=50
]{biblatex}
\def\hmath$#1${\texorpdfstring{{\rmfamily\textit{#1}}}{#1}}

\newtheorem{theorem}{Theorem}
\newtheorem{corollary}[theorem]{Corollary}
\newtheorem{observation}[theorem]{Observation}
\newtheorem{conjecture}[theorem]{Conjecture}
\newtheorem{definition}[theorem]{Definition}
\newtheorem{lemma}[theorem]{Lemma}
\newtheorem{proposition}[theorem]{Proposition}
\newtheorem{problem}[theorem]{Problem}
\newtheorem{claim}[theorem]{Claim}
\theoremstyle{remark}
\newtheorem{remark}[theorem]{Remark}
\newtheorem{example}[theorem]{Example}

\counterwithin{theorem}{section}

\newcommand{\indicator}{\mathds{1}}
\newcommand{\hypercubegraph}{\mathsf{RAG}(n, \{\pm 1\}^d, p,\sigma)}

\newcommand{\sphere}{\mathbb{S}}
\newcommand{\dsphere}{\mathbb{S}^{d-1}}
\newcommand{\dtorus}{\mathbb{T}^{d}}
\newcommand{\onetorus}{\mathbb{T}^{1}}

\newcommand{\poly}{\mathsf{poly}}

\newcommand{\RAG}{\mathsf{RAG}}
\newcommand{\bfG}{\mathbf{G}}
\newcommand{\bfH}{\mathbf{H}}
\newcommand{\bfK}{\mathbf{K}}
\newcommand{\bfA}{\mathbf{A}}
\newcommand{\RGG}{\mathsf{RGG}}

\newcommand{\ergraph}{\mathsf{G}(n,p)}

\newcommand{\hypercube}{\{\pm 1\}^d}
\newcommand{\signedcount}{\mathsf{SC}}

\newcommand{\signedweight}{\mathsf{SW}}
\newcommand{\unsignedweight}{\mathsf{W}}

\newcommand{\advantage}{\mathsf{ADV}}
\newcommand{\unitcircle}{\mathbb{S}^1}
\newcommand{\ergraphhalf}{\mathsf{G}(n, 1/2)}
\newcommand{\unif}{\mathsf{Unif}}

\newcommand{\error}{\mathsf{Err}}
\newcommand{\numconnectedcomponets}{\mathsf{numc}}

\newcommand{\nonzero}{\mathsf{NZ}}
\newcommand{\eqclasssize}{\mathsf{ES}}
\newcommand{\sfp}{\mathsf{p}}

\newcommand{\bfx}{\mathbf{x}}
\newcommand{\bfnull}{\mathbf{0}}
\newcommand{\bfc}{\mathbf{c}}
\newcommand{\bfv}{\mathbf{v}}
\newcommand{\bfy}{\mathbf{y}}
\newcommand{\bfz}{\mathbf{z}}
\newcommand{\bfg}{\mathbf{g}}
\newcommand{\bfgg}{\mathbf{gg}}
\newcommand{\bfh}{\mathbf{h}}
\newcommand{\bfu}{\mathbf{u}}

\newcommand{\bfP}{\mathbf{P}}
\newcommand{\bfQ}{\mathbf{Q}}

\newcommand{\expect}{{\mathbf{E}}}
\newcommand{\distribution}{{\mathcal{D}}}

\newcommand{\infl}{{\mathbf{Inf}}}
\newcommand{\inflvector}{\overrightarrow{\mathbf{{Inf}}}}
\newcommand{\maxinfl}{{\mathbf{MaxInf}}}

\newcommand{\Var}{{\mathbf{Var}}}
\newcommand{\Cov}{{\mathbf{Cov}}}

\newcommand{\TV}{{\mathsf{TV}}}
\newcommand{\KL}{{\mathsf{KL}}}
\newcommand{\prob}{{\mathbf{P}}}

\newcommand{\ER}{Erd{\H o}s-R\'enyi }

\newcommand{\B}{\Big}

\newcommand{\Group}{\mathcal{G}}

\newcommand{\quadand}{\quad \text{and} \quad}

\newcommand{\sign}{\mathrm{sign}}

\newcommand{\iidsim}{\stackrel{\mathrm{i.i.d.}}\sim}

\usepackage[titles]{tocloft}
\title{Detection of $L_\infty$ Geometry in Random Geometric Graphs:\\
Suboptimality of Triangles and Cluster Expansion}

\author{Kiril Bangachev\thanks{Dept. of EECS, MIT. \texttt{kirilb@mit.edu} } \quad Guy Bresler\thanks{Dept. of EECS, MIT. \texttt{guy@mit.edu}. Supported by NSF Career Award CCF-1940205.}}

\begin{document}

\maketitle
\pagenumbering{roman}
\begin{abstract}
In this paper we study the random geometric graph  $\RGG(n,\dtorus,\unif,\sigma^q_p,p)$ with $L_q$ distance where each vertex is sampled uniformly from the $d$-dimensional torus and where the connection radius is chosen so that the marginal edge probability is $p$.  
%
In addition to results addressing other questions, we make progress on determining when it is possible to distinguish $\RGG(n,\dtorus,\unif,\sigma^q_p,p)$ from the Erd\H os-R\'enyi graph $\ergraph$.



Our strongest result is in the extreme setting $q = \infty$, in which case $\RGG(n,\dtorus,\unif,\sigma^{\infty}_p,p)$ is the \textsf{AND} of $d$ 1-dimensional random geometric graphs. We derive a formula similar to the \emph{cluster-expansion} from statistical physics, capturing the compatibility of subgraphs from each of the $d$ 1-dimensional copies, and use it to bound the signed expectations of small subgraphs. We show that counting signed 4-cycles is optimal among all low-degree tests, succeeding with high probability if and only if $d = \tilde{o}(np).$ In contrast, the signed triangle test is suboptimal and only succeeds when $d = \tilde{o}((np)^{3/4}).$ Our result stands in 
sharp contrast to the existing literature on random geometric graphs (mostly focused on $L_2$ geometry) where the signed triangle statistic is optimal. 

\end{abstract}
\newpage

\setcounter{tocdepth}{2}
{
  \hypersetup{linkcolor=black}
  \tableofcontents
}
\newpage
\pagenumbering{arabic}
\section{Introduction}
Networks arising in the sciences are often modeled as latent space graphs. Each node in a network has a latent feature vector and the probability of connection between two nodes is a function of the two feature vectors. One instance is the case of (random) geometric graphs in which each feature vector is a (random) element of a metric space and the connection function is determined by the distance between the two vectors. Applications of random geometric graphs include protein-protein interactions and viral spread in the biological sciences \cite{Higham03,Preciado2009SpectralAO}, 
wireless networks and motion planning in engineering \cite{Haenggi2009,Solovey18},
consensus dynamics and citation networks in the social sciences \cite{XIE2016167,ESTRADA201620}. 

Formally, a random geometric graph is defined as follows.

\begin{definition}[Random Geometric Graph]
  Given are a metric space $(\Omega, \mu),$ a distribution $\distribution$ over $\Omega,$ and 
  {connection} function $\sigma:\Omega\times \Omega\longrightarrow [0,1]$ such that $\sigma(\bfx,\bfy)$ only depends on $\mu(\bfx,\bfy).$ 
  Let $\expect[\sigma(\bfx,\bfy)] = p.$
  Then, $\RGG(n,\Omega, \distribution,\sigma,p)$ is the following distribution over $n$-vertex graphs.
$$\displaystyle
\prob[\bfG = A] =
\expect_{\bfx^1, \bfx^2, \ldots, \bfx^n}
\Bigg[
\prod_{1\le i < j\le n} \sigma(\bfx^i,\bfx^j)^{A_{i,j}}
(1-\sigma(\bfx^i,\bfx^j))^{1-A_{i,j}}\Bigg].
$$
When $\sigma$ is monotone in $\mu,$ we say that $\RGG(n,\Omega, \distribution,\sigma,p)$ is a monotone random geometric graph.
\end{definition}
In words, each node $i$ has an associated independent latent vector $\bfx^i$ in $\Omega$ distributed according to $\mathcal{D}.$ Conditioned on $\bfx^1, \bfx^2,\ldots, \bfx^n,$ each pair of nodes $i$ and $j$ independently forms an edge with probability $\sigma(\bfx^i,\bfx^j).$ Now on, we will focus on the monotone case which has the natural interpretation that closer nodes are more (less) likely to be adjacent.\footnote{Non-monotone settings sometimes also have very natural interpretations, see for example \cite{bangachev2023random}.} In practice, one sometimes observes the network with partial data on the underlying feature vectors. In this work, we assume that the vectors are fully hidden.

Associated to random geometric graphs with latent vectors are a wide range of statistical and computational tasks such as: 1) \emph{Clustering and Embedding} of the nodes in a way that captures the distances between latent vectors \cite{li2023spectral,OConner20,Ma2020UniversalLS}; 2) \emph{Estimating the dimension} of the underlying space $\Omega$ in the case when dimension is naturally defined such as $\Omega \in \{\dsphere,\dtorus,\hypercube\}$ \cite{Bubeck14RGG,friedrich2023dimension}; 3) \emph{Testing} whether the network has a geometric structure against a ``pure noise'' (i.e., \ER)\footnote{In the \ER distribution $\ergraph,$ each of the $\binom{n}{2}$ edges appears independently with probability $p.$ As there is no underlying dependence structure, this is a natural null model.} null hypothesis \cite{Devroye11,Bubeck14RGG,Brennan19PhaseTransition,Liu2021APV,Liu2022STOC,Brennan22AnisotropicRGG,bangachev2023random} and others.

The current work is mostly focused on the hypothesis-testing question which can be formalised as follows (e.g. \cite{bangachev2023random}): Given $G$, decide between
\begin{equation}
    \label{eq:hypoithesistesting}
    H_0: G\sim \ergraph\quad  \text{versus}\quad  H_1:G\sim \RGG(n,\Omega, \distribution,\sigma,p). 
\end{equation}
Associated to these hypotheses are (at least) two different questions:
\begin{enumerate}
    \item \textit{Statistical:} 
    First, when is there a consistent test?
    To this end, we aim to characterize the parameter regimes in which the total variation between the two distributions tends to zero or instead to one. 
    \item \textit{Computational:} Second, we can ask for a computationally efficient test. In particular, when does there exist a polynomial-time test solving \cref{eq:hypoithesistesting} with high probability?
\end{enumerate}
This hypothesis testing question has received significant attention in recent years in the case when $(\Omega, \mu)$ captures an $L_2$ geometry. Namely, $\mu$ is the induced $L_2$ distance from $\mathbb{R}^d$ and $(\Omega,\mathcal{D})$ is either the unit sphere $\dsphere$ with its uniform (Haar) measure
\cite{Devroye11,Bubeck14RGG,Brennan19PhaseTransition,Liu2022STOC}
or Euclidean space $\mathbb{R}^d$ with a Gaussian measure
\cite{Liu2021APV,Liu21PhaseTransition,Brennan22AnisotropicRGG}. 
In all of the above \textit{monotone} models, the conjectured information-theoretically optimal statistic is the signed triangle statistic (see \cref{def:signedcountsperformance}), which is also computable in polynomial time. For a more extensive summary of results in models with $L_2$ geometry, we refer the reader to \cite{duchemin22,bangachev2023random}. Here, we only discuss the case (most relevant to our work) when $\Omega = \dsphere, \distribution = \unif,$ and 
$\sigma(\bfx,\bfy) = \indicator[\langle \bfx, \bfy\rangle\ge \rho^d_p],$ where $\rho^d_p$ is chosen so that the expected density is $p.$ The state of the art results are as follows. When $d = \tilde{O}(n^3p^3),$ by counting signed triangles (see \cref{def:signedcountsperformance}) one can distinguish between the $\RGG$ model and $\ergraph$ with high probability \cite{Bubeck14RGG,Liu2022STOC}. There is a matching information-theoretic lower bound when $p =\Theta(n^{-1})$ \cite{Liu2022STOC} and  
when $p = \Tilde{\Theta}(1)$ \cite{Bubeck14RGG}. The case $n^{-1}\ll p \ll \Tilde{\Theta}(1)$ remains open and the best known lower bound due to \cite{Liu2022STOC} is $d = \tilde{\Omega}(n^3p^2).$ Namely, when $d = \tilde{\Omega}(n^3p^2),$ one has 
$$
\TV\Big(\RGG(n,\dsphere,\unif, \sigma,p),\ergraph\Big) = o(1).
$$

In \cite{Bubeck14RGG}, the authors also show that the signed triangle statistic is optimal for \emph{exact} recovery of the dimension in the model $\Omega = \dsphere,\mathcal{D} = \unif$ and $\sigma(\bfx^i, \bfx^j) =\indicator[\langle \bfx^i,\bfx^j\rangle\ge 0].$ Using the (signed) triangle statistic for detecting geometry in monotone models is intuitive as it captures the axiomatic triangle inequality: If $x$ and $y $ are close and $y$ and $z$ are close, then so are $x$ and $z$ \cite{Bubeck14RGG}.

These results and intuition have led to the conventional wisdom that (signed) triangles are most informative in monotone random geometric graphs.\footnote{\cite{bangachev2023random} provides several \emph{geometric} examples in which signed triangles are not the optimal statistical test for \cref{eq:hypoithesistesting}. However, neither of them is a \emph{monotone random geometric graph}. In these examples, either the connection functions are not monotone in the respective distance or the connections functions do not correspond to true ``distances'' (but, for example, to a non-PSD inner product \cite[Theorem 6.17]{bangachev2023random}) } Subsequent works in very different geometries have also used triangle-based statistics, for example to 
estimate the hidden dimension
\cite{Almagro22,
friedrich2023dimension}.  

\medskip 

In this paper, we go against this conventional wisdom and demonstrate that the (signed) triangle statistic can be \emph{suboptimal}. More concretely, we study the hypothesis 
testing problem under
$L_q$ geometry for $q\in [1,\infty)\cup \{\infty\}$ and show that different values of $q$ yield both quantitatively and qualitatively different behaviours (see \cref{Fig:linftydiag,Fig:lqdiag}). In particular, when $q = \infty,$ triangle-based tests are always suboptimal. The suboptimality of triangle-based statistics extends to the task of dimension estimation as well. We use the (unweighted version of the) model of \cite{Friedrich23cliques,friedrich2023dimension} with $L_q$ geometry over $\dtorus,$ defined as follows.

\begin{definition}[$L_q$-Hard Thresholds Model on $\dtorus$] 
\label{def:Lqtorus}
Consider the torus $\dtorus\cong (2\sphere^1)^{\times d},$ which is a product of $d$ circles of circumference $2.$\footnote{We choose the circumference to be equal to 2 simply for convenience.}\footnote{One can equivalently define $\dtorus = \mathbb{R}^d/\sim,$ where $\bfx\sim \bfy$ if and only if $\bfx - \bfy\in 2\mathbb{Z}^d.$} Let $\unif$ be the uniform (Haar) measure over $\dtorus.$ For $x_1, y_1\in 2\sphere^1,$ denote by $|x_1 - y_1|_C\in [0,1]$ the circular distance, i.e. the length of the shorter arc connecting $x_1$ and $y_1.$ For $1\le q<+\infty,$ introduce the $L_q$ distance on $\dtorus$ given by 
$$
\|\bfx - \bfy\|_{q}\coloneqq 
\Big(
\sum_{i = 1}^d 
|x_i - y_i|_C^q
\Big)^{1/q}.
$$
Also, denote $\|\bfx-\bfy\|_\infty \coloneqq \lim_{q\longrightarrow + \infty}\|\bfx - \bfy\|_{q} = \max_i |x_i - y_i|_C.$
Let $1\ge p \ge 0,\tau^q_p\ge 0$ be such that 
$\expect_{\bfx,\bfy\iidsim \unif(\dtorus)}\Big[\indicator[\|\bfx-\bfy\|_q\le \tau^q_p]\Big] = p$
and $\sigma^q_p(\bfx,\bfy)\coloneqq \indicator[\|\bfx-\bfy\|_q\le \tau^q_p].$ Then, $\RGG(n, \dtorus,\unif,\sigma^q_p,p)$ is the random geometric graph over $\dtorus$ in which vertices are adjacent if and only if the $L_q$ distance between the corresponding latent vectors is at most $\tau_p^q,$ leading to expected density $p.$ 
\end{definition}

\begin{figure}[!htb]
   \begin{minipage}{0.47\textwidth}
     \centering
     \includegraphics[width=0.8\linewidth]{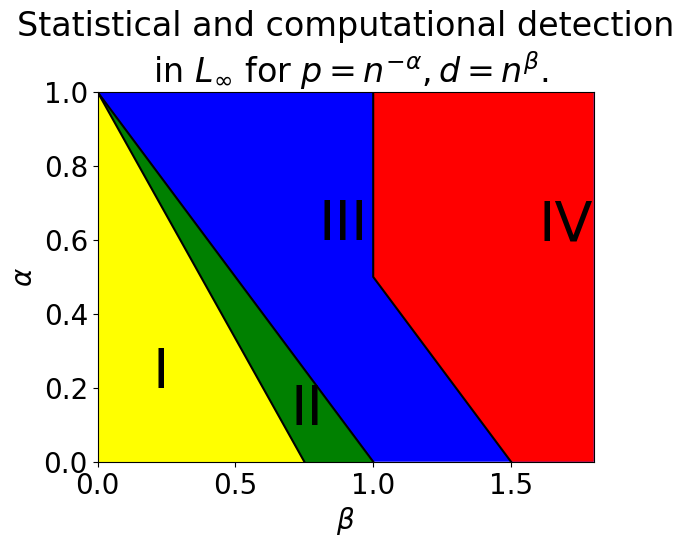}
     \caption{\footnotesize{Visualizing \cref{thm:linfinitydetection,thm:linftylowdegreeindist,thm:linftyinformationtheory}. In region $I,$ the signed triangle test solves \cref{eq:hypoithesistesting} for $\RGG(n,\dtorus, \unif, \sigma^\infty_p,p)$ with high probability. In region $\mathrm{I} + \mathrm{II},$ the signed 4-cycle test  succeeds with high probability. In region $\mathrm{III} + \mathrm{IV},$ no low-degree polynomial test succeeds. In $\mathrm{IV},$ it is information theoretically impossible to solve \cref{eq:hypoithesistesting} with high probability. The last region is potentially suboptimal.}}\label{Fig:linftydiag}
   \end{minipage}\hfill
   \begin{minipage}{0.47\textwidth}
     \centering
     \includegraphics[width=0.7\linewidth]{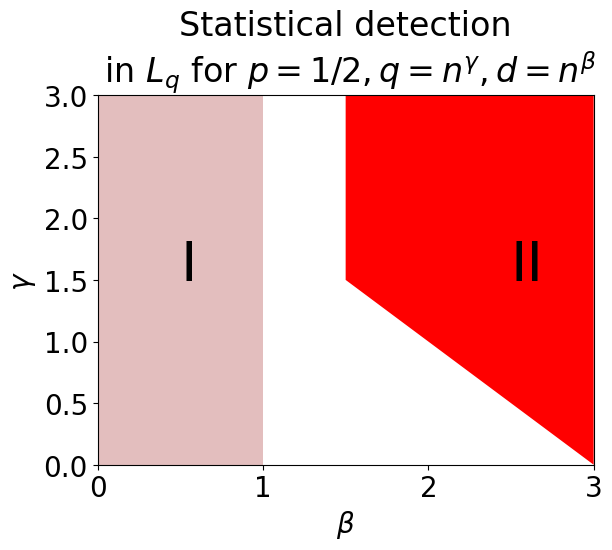}
     \caption{\footnotesize{Visualizing \cref{thm:torussmallq,thm:entropybound}. 
     In region $\mathrm{I},$ the entropy of 
     $\RGG(n,\dtorus,\unif,\sigma^q_{1/2}, 1/2)$ is much lower than that of $\ergraphhalf.$ Yet, we do not know any efficient test that distinguishes the two graph models in this region (even though, we believe the signed 4-cycle count does in a strictly larger region \cref{conj:sigendfourcycleinlqp}). In region $\mathrm{II},$  it is information theoretically impossible to solve \cref{eq:hypoithesistesting} with high probability. Both regions are potentially suboptimal.
     }}\label{Fig:lqdiag}
   \end{minipage}\hfill
\end{figure}

To the best of our knowledge, the work of \cite{Friedrich23cliques} is the first to explore \cref{eq:hypoithesistesting} for random geometric graphs in non-$L_2$ geometries. 
They showed that in the $L_q$ model of \cref{def:Lqtorus}\footnote{Their result is actually slightly more general as it applies to inhomogeneous random geometric graphs.} for fixed $p,n,$
$$\lim_{d\to \infty}\TV\Big(\RGG(n,\dtorus, \unif, \sigma^q_p,p),\ergraph\Big) = 0.$$
Their approach, based on a multidimensional Berry-Esseen theorem and mimicking \cite{Devroye11}, however, only yields $\TV$ distance of order $o(1)$ when $d = \exp(\Omega(n^2)).$ The authors pose the direction of improving this bound as an open problem, which is also one of the main motivations of the current work.

A different direction of study taken in \cite{Friedrich23cliques} is estimating the probability with which a given set of edges
appears
in $\RGG(n,\dtorus, \unif, \sigma^\infty_p,p).$  In the homogeneous case when also $d = \omega(\log(n)^2)$ their results can be restated as follows: if $|\mathcal{A}|$ is constant, then
the probability that all edges of $\mathcal{A}$ appear in $\RGG(n,\dtorus, \unif, \sigma^\infty_p,p)$ is $p^{|\mathcal{A}|}(1 + o(1)).$
This also allows the authors to bound the clique number of $\RGG(n,\dtorus, \unif, \sigma^\infty_p,p)$ (and its inhomogeneous generalization). In a subsequent paper, the authors use these quantities for estimating the dimension of a random geometric graph \cite{friedrich2023dimension}.

\subsection{Main Results for \texorpdfstring{$L_\infty$}{Linfty} Geometry}
The $L_\infty$ case is special because of the following factorization property over coordinates: \emph{$\|\bfx - \bfy\|_\infty\le \tau$ holds if and only if $|x_i - y_i|\le \tau $ holds for each $i \in [d].$} This means that each edge $(u,v)$ is the $\mathsf{AND}$ of $d$ \emph{independent} edges in the 1-dimensional random geometric graphs over the different coordinates. In comparison, 
previously studied $L_2$ models have a (weighted) $\mathsf{MAJORITY}$ combinatorics. For instance, in the spherical case $\langle \bfx,\bfy\rangle\ge \rho$ if and only if 
$\sum_{i = 1}^d |x_i|\times |y_i|\times \sign(x_iy_i)\ge \rho.$ Here, each $\sign(x_iy_i)$ is an independent 1-dimensional edge and the values $|x_i|\times |y_i|$ are the corresponding weights. 

Factorization over the induced independent 1-dimensional random geometric graphs makes the computation of expected (signed) subgraph counts tractable as computations in one dimension are naturally much simpler (see \cref{claim:1dtoruscounts}). Signed subgraph counts are fundamental in studying random graph distributions as they are the Fourier coefficients of the probability density.
The factorization property, also utilized in \cite{Friedrich23cliques}, is the first main ingredient in our results in the $L_\infty$ case. 

The second ingredient is combining the induced 1-dimensional structures via the $\mathsf{AND}$ function. While in certain special cases this step is nearly trivial (e.g., in \cref{thm:linftyinformationtheory} we only need to do it for $K_{2,t}$ subgraphs and in \cref{thm:linfinitydetection} for triangles and 4-cycles), in full generality it requires a careful analysis of the compatibility of induced 1-dimensional structures. We carry out such an analysis in \cref{sec:linftysignedcounts} by viewing each 1-dimensional structure as a polymer and expanding the product over the $d$ coordinates. A rearrangement of terms yields a tremendous amount of cancellations that leaves us with an expression for the expected signed subgraph counts similar to the celebrated cluster expansion formula (e.g., \cite{mayer40cluster,kotecky86,friedli_velenik_2017}) from statistical physics (which has found many other applications in combinatorics, e.g. \cite{scott2005cluster}). In our case, the compatibility criterion is given by the size of the overlap of different 1-dimensional structures. What makes a cluster-expansion-like formula appealing is a rapid decay of terms which means that terms corresponding to small clusters determine its asymptotics (as in the Koteck{\`y}-Preiss theorem \cite{kotecky86}). The derivation and analysis of this formula is our technical and conceptual highlight in the $L_\infty$ case.

\medskip

Throughout we will frequently refer to signed subgraph count tests and low-degree polynomial tests. As these are by now standard in the literature on latent space graphs, we defer the definitions to \cref{sec:preliminaries}.
Throughout the rest of the paper, we make the following assumption:\footnote{Most results can be extended to the setting $\min(p, 1-p) = \Omega(n^{-1}), \exp((\log n)^\delta)<d,$ but this comes at a significant cost in the exposition.} 
\begin{equation}\tag{\textbf{A}}
\label{eq:assumption}
    \text{There exist some absolute constants $\delta, \epsilon>0$ such that  $n^{-1 + \epsilon}\le p \le 1/2, n^\delta \le d.$}
\end{equation}

\subsubsection{Detecting $L_\infty$ Geometry}

Our first result shows information-theoretic indistinguishability from \ER graphs for dimension above a certain value. An argument due to Liu and Racz \cite{Liu2021APV} (see \cref{eq:Kltensorization}) reduces this question to bounding signed counts of $K_{2,t}$ subgraphs, which facilitates the following result. 

\begin{theorem}[Information-Theoretic Lower Bound for $L_\infty$ Model] 
\label{thm:linftyinformationtheory}
If $d = \tilde{\omega}(\max(n^{3/2}p,n)),$ then
$$
\TV\Big(
\RGG(n, \dtorus,\unif,\sigma^\infty_{p},p),
\ergraph
\Big) = o(1).
$$
\end{theorem}

\cref{thm:linftyinformationtheory} already highlights a quantitative difference between $L_\infty$ random geometric graphs over $\dtorus$ and $L_2$ models over $\dsphere$(see the aforementioned results of \cite{Liu2022STOC}). The former converge to \ER at a polynomially smaller dimension. Much more interesting, however, is the following qualitative difference. Signed triangles are suboptimal for detecting $L_\infty$ geometry and signed four-cycles are strictly stronger at any density $p \ge n^{-1+\epsilon}.$

\begin{theorem}
    \label{thm:linfinitydetection}
Under Assumption \eqref{eq:assumption}, consider the hypothesis testing \cref{eq:hypoithesistesting} with 
$H_1:\RGG(n,\dtorus,\unif,\sigma^\infty_p,p).$ 
\begin{enumerate}
\item The signed 4-cycle test distinguishes the two graph models w.h.p. if and only if $d = \tilde{o}(np).$ 
\item
The signed triangle test distinguishes the two graph models w.h.p. if and only if $d = \tilde{o}((np)^{3/4}).$ 
\end{enumerate}
\end{theorem}

We provide some intuition behind the suboptimality of signed triangles and further consequences in  \cref{sec:linfty4vs3}. 
Before that, however, we address the large gap left between the 4-cycle statistic upper bounds in \cref{thm:linfinitydetection} and information-theoretic lower bound for convergence to \ER in \cref{thm:linftyinformationtheory}.
We show that the signed 4-cycle statistic is optimal (up to lower order terms) among low-degree tests.

\begin{theorem}[Computational Lower Bound for $L_\infty$ Model]
\label{thm:linftylowdegreeindist}
Under Assumption \eqref{eq:assumption},
there exists some function $f_\epsilon(t) = o_t(1)$ with the following property. No polynomial test of degree $(\log n)^{5/4}/(\log \log n)$ can distinguish 
$\ergraph$ and $\RGG(n, \dtorus,\unif, \sigma_p^\infty,p)$ with high probability when $d \ge (np)^{1 + f_\epsilon(np)}.$ \end{theorem}

A popular conjecture is that ``sufficiently noisy'' statistical problems in high-dimension can be solved in polynomial time only if there
is an 
$O(\log n)$-degree polynomial tests that solves them
\cite{hopkins18}. In this light, our result suggests that at least one of the following is true: 1) There is a statistical-computational gap for detecting $L_\infty$ geometry;
2) Or, \cref{thm:linftyinformationtheory} is suboptimal. Whether there is a statistical-computational gap for testing between $\RGG(n, \dtorus,\unif, \sigma_p^\infty,p)$ and $\ergraph$ is an exciting question for future research. Closely related models in the literature provide examples of both positive and negative answers to this question. Spherical random geometric graphs do not exhibit a statistical-computational gap in the dense case $p= 1/2$ \cite{Bubeck14RGG}. On the other hand, in \cite[Definition 2.18]{kothari2023planted}, the authors construct an instance of the stochastic block model - which, in particular, can be realized as a random algebraic graph over a (discrete) torus - with an information-computation gap (at least within the low-degree polynomial tests framework).

The main step in proving \cref{thm:linftylowdegreeindist} is utilizing the aforementioned cluster-expansion-like approach
which gives the following bound on signed subgraph weights in $\RGG(n,\dtorus,\unif,\sigma^\infty_p,p).$ We give an overview of the cluster-expansion approach in \cref{sec:clusterhighlevel}. For a set of edges \linebreak 
$H = \{(i_1,j_1),(i_2,j_2),\ldots, (i_k,j_k)\},$ denote 
$
\signedweight_H(G) \coloneqq  
\prod_{(ij)\in E(H)}(G_{ij} - p).
$ 

\begin{proposition}
\label{thm:linftysignedcounts}
Suppose that $H\subseteq K_n$ is a graph on $|E(H)| \le (\log d)^{5/4}/(\log \log d)$ edges. Under Assumption \eqref{eq:assumption}, 
$$
\Big|\expect_{\bfG\sim \RGG(n,\dtorus,\unif,\sigma_p^\infty,p)}\Big[\signedweight_H(\bfG)\Big]\Big|  = 
{O}\Bigg(p^{|E(H)|}\bigg(\frac{{(\log d)^{C}}}{{d}}\bigg)^{|V(H)|/2}\Bigg
),
$$
where $C$ is an absolute constant.
\end{proposition}

The quantity $p^{|E(H)|}$ appears naturally as each of the $|E(H)|$ edges has marginal expectation $p.$ An exponentially small quantity in the number of vertices i.e.  $\big({{(\log d)^{C}}}/{{d}}\big)^{|V(H)|/2}$---appears frequently in the computation of Fourier coefficients of probabilistic latent space graphs as 
it corresponds to events determined by the $|V(H)|$ latent vectors
(e.g., \cite{hopkins18} for planted clique and \cite{kothari2023planted,rush2022} for certain instances of the stochastic block model). While we do not currently have an intuitive explanation of why $|V(H)|/2$ is the correct quantitative dependence in our case, it is crucial to the proof of \cref{thm:linftylowdegreeindist} and a weaker exponent of the form $|V(H)|/(2 +\xi),$ where $\xi>0$ is constant, would not suffice.

A simplification of our methods in \cref{thm:linftysignedcounts} yields improved estimates in its unsigned analogue, i.e. subgraph counts in $\RGG(n,\dtorus,\unif, \sigma^\infty_p,p)$, which were studied in \cite{Friedrich23cliques}. See \cref{sec:linftyunsignedweights}.

\subsubsection{Triangles and 4-Cycles in $L_\infty$ Geometry}
\label{sec:linfty4vs3}

We end our discussion of the $L_\infty$ model with a further comparison between signed triangle counts and signed four-cycle counts. We begin with an example illustrating why signed triangles are less informative than signed four-cycles.

\begin{example}
\label{rmk:why4not3}
Consider, for simplicity, the density $p = 1/2$ case, in which one can compute that $\tau^\infty_{1/2} =  1 - \lambda,$ where $\lambda  = \Theta(1/d)$ (see \cref{def:Lqtorus}).

First, we interpret the signed expectation of triangle $\{1,2,3\},$ i.e. 
$\expect[(2\bfG_{12}-1)(2\bfG_{23} - 1)(2\bfG_{31} - 1)].$ This expectation measures the correlation between the events \textit{2 is a neighbour of $1$} (captured by the term $(2\bfG_{12} - 1)$) and \textit{2 is a two-step neighbour of $1$ via 3} (the term $(2\bfG_{13}-1)(2\bfG_{32} - 1)$). For the case of random geometric graphs over the unit sphere, e.g. \cite{Bubeck14RGG}, these two notions are well correlated as both are monotone in the distance between the latent vectors $\bfx^1, \bfx^2.$ The closer $\bfx^1, \bfx^2$ are, the larger the probability that $\bfx^3$ is a common neighbor or a neighbor of neither. 

This, however, is not the case in the $L_\infty$ model. Consider, for example, $\bfx^1 = (0,0,\ldots, 0), \bfx^{2_a} = (1,0,0,\ldots, 0),$ and $\bfx^{2_b} = (\frac{1}{2}, \frac{1}{2},\ldots, \frac{1}{2}).$ Clearly, $\|\bfx^1 - \bfx^{2_a}\|_\infty = 1,$ so vertices $1$ and $2_a$ are not adjacent. Still, the set of latent vectors adjacent to $\bfx^1, \bfx^{2_a}$ has measure $(1-2\lambda)\times (1-\lambda)^{d-1} = \frac{1}{2}(1 + o(1))$ since a point $\bfx^3$ is adjacent to $\bfx^1$ and $\bfx^{2_a}$ if and only if $(\bfx^3)_1 \not \in (-\lambda,\lambda)\cup (1-\lambda, 1 + \lambda),$ and 
$(\bfx^3)_i \not \in (1-\lambda, 1 + \lambda)$ for $i \in \{2,3,\ldots, d\}.$ In contrast, $\bfx^1$
 and $\bfx^{2_b}$ are adjacent and only at distance $1/2$, but the set of latent vectors adjacent to $\bfx^1, \bfx^{2_b}$ has the much smaller measure $(1-2\lambda)^d = \frac{1}{4}(1 + o(1)).$ A point $\bfx^3$ is adjacent to $\bfx^1$ and $\bfx^{2_b}$ if and only if $(\bfx^3)_i \not \in (\frac{1}{2}-\lambda,\frac{1}{2}+\lambda)\cup (1-\lambda, 1 + \lambda)\; \forall i.$ This lack of correlation causes (signed) triangle counts to be suboptimal. 

The 4-cycle statistic on cycle $\{1,3,2,4\}$ measures the correlation between two-step paths 1--3--2 and 1--4--2 from 1 to 2. This statistic does not suffer from the same issue as signed triangle counts because it measures the correlation between two objects of the same type. 

In \cref{sec:linftyunsignedweights}, we see yet another reason why bipartite subgraph tests are more informative in the $L_\infty$ model. It has to do with the fact that all short cycles in the complements of the induced 1-dimensional random geometric graphs are of even length.\end{example}

Finally, we show that the advantage of counting signed four cycles over counting signed triangles in $\RGG(n,\dtorus,\unif,\sigma^\infty_p ,p)$ extends beyond the task of distinguishing from \ER. The existing literature on dimension estimation is fully focused on triangle-based statistics \cite{Bubeck14RGG,Almagro22,friedrich2023dimension}. Not much is known about the optimality of these statistics beyond the case of $L_2$ geometry. We show that indeed, the simple signed 4-cycle counting test is stronger than the signed triangle test also for the problem of estimating the dimension in $\RGG(n,\dtorus,\unif,\sigma^\infty_p ,p).$ Specifically, we consider the following problem.

\begin{problem}
\label{problem:dimensionestimation}
On input $n,p$ and $\bfG,$ where $\bfG\sim \RGG(n,\dtorus, \unif,\sigma^\infty_p,p),$ find the unknown dimension $d$ exactly with high probability.
\end{problem}
Of course, one can also consider variants of this problem, such as when the expected density $p$ is unknown or when one allows for a small error in estimating $d.$ We focus on this simplest version as our goal is to demonstrate the advantage of counting signed four-cycles over counting signed three-cycles. The precise statement is given in \cref{sec:applicationstodimensionestimation}. For now, we say that our formal notion of success of exact recovery of dimension via polynomial tests (given in \cref{def:estimationfrompoly}) exactly captures prior work on the problem \cite{Bubeck14RGG,friedrich2023dimension} and it mimics the more common framework of hypothesis testing using low-degree polynomial tests (for example, \cite{hopkins18}). 
 
\begin{proposition}[Informal, Simple Tests for Dimension Estimation] 
\label{prop:linftydimensionestmation} Consider \cref{problem:dimensionestimation} under assumption \cref{eq:assumption} with known value of $\delta$ such that $d\ge n^\delta.$
\begin{enumerate}
    \item The signed 4-cycle statistic recovers the dimension $d$ correctly w.h.p. if and only if $d = \tilde{o}((np)^{2/3}).$
    \item The signed triangle statistic recovers the dimension $d$ correctly w.h.p. if and only if $d = \tilde{o}((np)^{1/2}).$
\end{enumerate}

\end{proposition}
It is important to note that \cref{prop:linftydimensionestmation} holds under the assumption \cref{eq:assumption} requiring $np$ and $d$ to be polynomial in $n.$ The setting of \cite{friedrich2023dimension} in which the authors use a (weighted) signed triangle count is in the regime $np = \Theta(1), d = o(\log n).$


\subsection{Additional Results}

\subsubsection{\texorpdfstring{$L_q$}{Lq} Geometry for \texorpdfstring{$q<\infty$}{q}}
So far, we have shown that random geometric graphs with $L_\infty$ geometry behave qualitatively and quantitatively differently from $L_2$ models with respect to \cref{eq:hypoithesistesting}. This motivates the question of understanding \cref{eq:hypoithesistesting} under other geometries as well, in particular $L_q.$ 

The choice of latent space $(\dtorus, \unif)$ for comparison of random geometric graphs with $L_q$ geometries is natural. A large class of natural symmetries of $\dtorus$ such as coordinate permutations and translations ($\bfx\longrightarrow \bfc + \bfx$ for a fixed $\bfc$) are isometries for any $L_q$ metric: there exists a transitive group of isometries for any $L_q$ metric over $\dtorus$ that is also measure-preserving. This leads to the following desirable \emph{homogeneity} property: for any fixed $x\in \dtorus$ and $\bfy\sim\unif(\dtorus),$ the distance $\|x-\bfy\|_q$ (hence, $\sigma(x,\bfy)$) has the same distribution.

The analysis of $L_q$ models, however, turns out to be much more difficult when $q<\infty$ as the factorization over 1-dimensional random geometric graphs does not hold any longer. In particular, this makes the computation of signed subgraph counts much more difficult and we have not succeeded to perform such a computation even for triangles.

One special case in which we manage to bound the signed subgraph count is the case of bipartite graphs $K_{2,t},$ which is enough to prove an analogue of \cref{thm:linftyinformationtheory}. What makes this calculation simpler is that the signed expectation of $K_{2,t}$ has a very natural interpretation as the $t$-th centered moment of the self-convolution of $\sigma^q_{1/2}.$ Using the Bernstein-McDiarmid inequality (see \cref{claim:bernsteinmcdiarmidrag}), we bound the centered moments of $\sigma$ by revealing the $d$ coordinates one at a time. The technical highlight of this argument is proving that each coordinate (say $x_d$) is marginally nearly uniform on $\onetorus$ even conditioned on the value of $\sigma^q_{1/2}(\bfx,\bfy)$ when $q\ll d.$ The reason for this phenomenon is that the the contribution of the remaining $d-1$ coordinates, i.e. $\sum_{i = 1}^{d-1} 
|x_i - y_i|_C^q,$ is sufficiently anticoncentrated and, thus, there are no spikes in its distribution that would bias $x_d$ strongly when conditioning on $\sigma^q_{1/2}(\bfx,\bfy).$ The formal statement is given in \cref{claim:lqanticoncentration} and we prove it by adapting an anticoncentration inequality of Bobkov and Chistyakov \cite{Bobkov14} to random variables with unbounded density (see \cref{appendix:anticoncentration}).

In our analogue of \cref{thm:linftyinformationtheory},
we fix $p = 1/2$ and vary $q$ so that we obtain a meaningful comparison of the convergence to \ER for different geometries. 

\begin{theorem}
\label{thm:torussmallq}
Suppose that $q\ge 1.$  
\begin{enumerate}
    \item If $q = o(d/\log d),$ then 
    $
\TV\Big(
\RAG(n, \dtorus,\sigma^q_{1/2},1/2),
\ergraphhalf
\Big) = o(1)
$ whenever $dq = \tilde{\omega}(n^3).$
\item If $q = \Omega(d/\log d),$ then 
$
\TV\Big(
\RAG(n, \dtorus,\sigma^q_{1/2},1/2),
\ergraphhalf
\Big) = o(1)
$ whenever $d^2 = \tilde{\omega}(n^3).$
\end{enumerate}
\end{theorem}
This statement interpolates between known results for $L_2$ models where convergence to \ER occurs when $d = \tilde{\omega}(n^3)$ (for example, in the spherical case \cite{Bubeck14RGG}) and $L_q$ models when convergence occurs for $d^2 = \tilde{\omega}(n^{3})$ (see \cref{thm:linftyinformationtheory}).

As already mentioned, we did not manage to prove algorithmic upper bounds for distinguishing general $L_q$ geometry. We present some minimal progress and conjectures in \cref{appendix:signedcountsinLq}, based on a Fourier-analytic interpretation of signed subgraph counts similar to \cite[Observation 2.1]{bangachev2023random}.
We can only rigorously show the following entropy-based upper bound which, however, does not obviously lead to a computationally efficient test. 

\begin{theorem} 
\label{thm:entropybound}
Take any $q\in [1,+\infty]$ and any $p$ such that $1/2\ge p \ge 1/n.$ If $d  = o(np/\log n),$ then 
$$
\TV\Big(\RGG(n,\dtorus,\unif, \sigma^\infty_p,p), \ergraph\Big) = 1 - o(1).
$$
\end{theorem}
Interestingly, this gives the same bound as the signed 4-cycle test in the case of $L_\infty$ geometry (in \cref{thm:linfinitydetection}). 

\subsubsection{Random Algebraic Graphs}
\label{sec:rag}

What makes the Bernstein-McDiarmid analysis feasible in the case of \cref{thm:torussmallq} is that the coordinates of $\dtorus$ are independent. It turns out that the method can be extended to other cases of a product structure.

\begin{definition}[Random Algebraic Graph over an Abelian Group \cite{bangachev2023random}]
\label{def:rag}
Suppose that $\Group$ is a finite Abelian group or a finite-dimensional torus $\dtorus.$ Let $\unif$ be the uniform measure \footnote{That is, the Haar measure in the case of $\dtorus.$} over $\Group$ and let $\sigma: \Group\longrightarrow [0,1]$ be a measurable function such that $\sigma(\bfg) = \sigma(-\bfg)$ holds a.s. and $\expect_{\bfg\sim\unif(\Group)}[\sigma(\bfg)] = p.$ Then, the random algebraic graph $\RAG(n, \Group,\sigma,p)$ is a random graph over vertex set $[n]$ with distribution of its adjacency matrix $\bfA$ given by 
\begin{equation}
    \prob[\bfA = G] = \expect_{\bfx^1, \bfx^2, \ldots, \bfx^n\iidsim\unif(\Group)}
    \Bigg[
    \prod_{i<j} \sigma(\bfx^i - \bfx^j)^{G_{ij}}(1 - \sigma(\bfx^i - \bfx^j))^{1-G_{ij}}\Bigg].
\end{equation}
\end{definition}

For any choice of $n,d,q,p,$ the random geometric graph $\RGG(n,\dtorus,\unif, \sigma^q_p,p)$ is also a random algebraic graph under the choice $\Group = \dtorus$ and $\sigma(\bfg) = \indicator[\|\bfg\|_q\le \tau^q_p].$ Overloading notation, we will also use $\sigma^q_p$ as one function-argument, that is $\sigma^q_p(\bfx, \bfy) = \sigma^q_p(\bfx - \bfy).$

In \cite{bangachev2023random}, the authors study random algebraic graphs over $\hypercube$ with general connections $\sigma.$ They derive a general criterion based on the sizes of Fourier coefficients on each level that guarantee\linebreak $\TV\Big(\RAG(n,\hypercube, \sigma, p), \ergraph\Big) = o(1)$ in \cite[Theorem 3.1.]{bangachev2023random}. Using a much simpler argument, based on the combination of \cref{eq:Kltensorization} and Bernstein's inequality, we also recover such a criterion.

\begin{theorem} 
\label{thm:hypercube}
Suppose that $\sigma:\hypercube\longrightarrow [0,1]$ is a connection with expectation $p.$ Then, 
$$
\TV\Big(\RAG(n,\hypercube,\sigma,p),\ergraph\Big)^2 = O\Bigg(\frac{n^3\sum_{i = 1}^d\infl_i[\sigma]^2}{p^2(1-p)^2}\Bigg).
$$ 
\end{theorem}

A detailed comparison between \cref{thm:hypercube} and \cite[Theorem 3.1.]{bangachev2023random} is provided in \cref{appendix:twohypercuberesults}. For now, we simply show two very quick applications of the theorem.

\begin{corollary} 
\label{cor:applicationshypercube}
$
\TV\Big(\RAG(n,\hypercube,\sigma,p),\ergraph\Big) = o(1) 
$ in the following cases:
\begin{enumerate}
    \item If $\sigma$ is $\frac{1}{r\sqrt{d}}$-Lipschitz and $d = o\big(\frac{n^3}{p^2r^4}\big) .$
    \item If $\sigma(\bfg) = \indicator\Big[\sum_{i = 1}^dg_i\ge \tau^{\hypercube}_p\Big],$ where $\tau^{\hypercube}_p$ is defined so that $\expect[\sigma] = p,$ and 
    $d = \tilde{\omega}(n^3p^2).$
\end{enumerate}
\end{corollary}
\begin{proof}
For the first statement, observe that whenever $\sigma$ is $(r\sqrt{d})^{-1}$-Lipschitz, by the definition of influence,
$$
\infl_i[\sigma] = \expect_{\bfx\sim\unif(\hypercube)}\Bigg[\Big(\frac{\sigma(\bfx) - \sigma(\bfx^{\oplus i})}{2}\Big)^2\Bigg]\le 
\frac{1}{r^2d},
$$
where $\bfx^{\oplus i}$ denotes the vector $\bfx$ with the $i$-th coordinate flipped. We used
$|\sigma(\bfx) - \sigma(\bfx^{\oplus i})| \le \frac{2}{r\sqrt{d}}$ which follows directly from the Lipschitzness assumption.
The conclusion follows from \cref{thm:hypercube}.

For the second statement, again consider  
$\displaystyle
\infl_i[\sigma] 
$.
The expression $\sigma(\bfx) - \sigma(\bfx^{\oplus i})$ is non-zero only if $\bfx$ has $\frac{d + \tau_p}{2}$ or $\frac{d + \tau_p}{2} - 1$ ones. A simple calculation (carried out, for example, in \cite[Proof of Proposition 4.7]{bangachev2023random}) 
shows that the probability of this happening is $ O(p\sqrt{\log \frac{1}{p}}/{\sqrt{d}}).$ This means that each influence is of order $\tilde{O}(p^2/d)$ and the conclusion follows.
\end{proof}

\section{Preliminaries and Notation}
\label{sec:preliminaries}
\paragraph{Graph Notation.} Denote by $K_n$ the clique on $n$ vertices, by $K_{a,b}$ the complete bipartite graph with parts of sizes $a$ and $b,$ 
and by $C_m$ the cycle on $m$ vertices. 
For a set of edges 
$H = \{(i_1, j_1), \ldots, (i_k,j_k)\}\in [n]\times [n],$ denote by $H$ the subgraph of $K_n$ with vertex set $\{i_1, j_1, i_2, j_2, \ldots, i_k, j_k\}$ and edge set $\{(i_1, j_1), \ldots, (i_k,j_k)\}.$

A graph is $2$-connected if it is connected and for any $v\in V(H),$ the induced subgraph of $H$ on vertex set $V(H)\backslash\{v\}$
is connected.

\subsection{Statistical Detection of Latent Space Structure}
\label{prelim:statsindist}
\paragraph{Information Theory.}
We use the standard notions for Total Variation and KL-distance (for example, \cite{Polianskiy22+}). Specifically, for two distributions $\bfP,\bfQ$ over the same measurable spaces $(\Omega,\mathcal{F})$, such that $\bfP$ is absolutely continuous with respect to $\bfQ,$
\begin{equation}
\begin{split}
    &\TV(\bfP,\bfQ) = 
    \sup_{A\in \mathcal{F}}|\bfP(A) - \bfQ(A)| = 
    \frac{1}{2}\int_\Omega \Big|\frac{d\bfP(\omega)}{d\bfQ(\omega)} - 1\Big|d\bfQ(\omega),\\
    &\KL(\bfP\|\bfQ) = 
    \int_\Omega 
    \frac{d\bfP(\omega)}{d\bfQ(\omega)}\log \frac{d\bfP(\omega)}{d\bfQ(\omega)}
    d\bfQ(\omega).
    \end{split}
\end{equation}
Total variation appears naturally in hypothesis testing settings as $1 - \TV(\bfP, \bfQ)$ is the minimal sum of Type I and Type II errors when testing between $\bfP$ and $\bfQ$ with a single sample (e.g. \cite{Polianskiy22+}). In practice, it is usually more convenient to work and compute with $\KL.$ Importantly, this is enough for proving convergence in total variation due to the celebrated inequality of Pinsker stating that
$\TV(\bfP,\bfQ)^2 \le \frac{1}{2}\KL(\bfP,\bfQ).$ 

\paragraph{A Bound on the KL divergence due to Liu and Racz.}
In \cite{Liu2021APV}, the authors give the following convenient bound on the $\KL$ divergence between $\ergraph$ and a probabilistic latent space graph. Specialized to random algebraic graphs (which encompass graphs $\RGG(n,\dtorus, \unif, \sigma^q_p, p)$), their bound reads as follows:
\begin{equation}
\label{eq:Kltensorization}
\begin{split}
& \KL\Big(\RAG(n, \Group, \sigma, p) \|\ergraph\Big)\le  
\sum_{k = 0}^{n-1}
\log \Big(\expect_{\bfx\sim \unif{(\Group)}}\bigg[\Big(1 + \frac{\gamma(\bfx)}{p(1-p)}\Big)^k\bigg]\Big),\\
& \text{where }
\gamma(\bfx) \coloneqq  \expect_{\bfz\sim \unif{\Group}}\bigg[(\sigma(\bfx  -\bfz)-p)(\sigma(\bfz)-p)\bigg] = 
\expect_{\bfz\sim \unif{\Group}}\bigg[\sigma(\bfx  -\bfz)\sigma(\bfz)\bigg] - p^2.
\end{split}
\end{equation}

Over random algebraic graphs, $\gamma(\bfx) = \sigma*\sigma(\bfx) - p^2,$ where $\sigma*\sigma(\bfx): = \expect_{\bfz\sim \unif{\Group}}\bigg[\sigma(\bfx  -\bfz)\sigma(\bfz)\bigg]$
is the self-convolution. Thus, one can expand the left hand-side of \cref{eq:Kltensorization} either in terms of the moments of $\sigma*\sigma$ or in terms of the moments of $\sigma*\sigma - p^2.$ 
It turns out that in the case of $\RGG(n,\dtorus, \unif, \sigma^\infty_p,p),$ one can easily compute (up to lower-order terms) the moments of $\sigma*\sigma$ and this is enough to prove
\cref{thm:linftyinformationtheory}.

\begin{remark}
    \normalfont
In \cref{appendix:onthekltensobound}, we discuss two combinatorial interpretations of \cref{eq:Kltensorization} which connect the bound of Liu and Racz to different notions of pseudorandomness appearing in the literature. 
One is related to the recent break-through work of Kelly and Meka on 3-term arithmetic progressions \cite{kelley2023strong} and the other to the classic work of Chung-Graham-Wilson on quasi-random graphs \cite{chung87}. To the best of our knowledge, these interpretations were not known to Liu and Racz. 
\end{remark}

\paragraph{The Bernstein-McDiarmid Approach.} In the case of $L_q$ geometry for $q<\infty,$ calculating the moments of $\sigma*\sigma$ seems out of reach. Our proof of \cref{thm:torussmallq} instead exploits the product structure of $\dtorus$ to bound the moments of $\gamma$ via the Bernstein-McDiarmid inequality.

\begin{claim}[{\cite[Corollary 5.6 and Problem 5.2]{RvH}}]
\label{claim:bernsteinmcdiarmidrag}
    Let $\bfg_1, \bfg_2, \ldots, \bfg_d$ be independent random variables and $\gamma$ a function of $(\bfg_1, \bfg_2, \ldots, \bfg_d).$
    Denote $\bfg_{-i} \coloneqq  (\bfg_1, \bfg_2, \ldots, \bfg_{i-1},\bfg_{i+1},\ldots, \bfg_d)$ and 
    \begin{equation*}
    \begin{split}
    &D_i\gamma(\bfg_{-i}) \coloneqq  
    \sup_{\bfg^+_i}\gamma((\bfg_1, \bfg_2, \ldots, \bfg_{i-1},\bfg_i^+,\bfg_{i+1},\ldots, \bfg_d)) - 
    \inf_{\bfg^{-}_i}
    \gamma((\bfg_1, \bfg_2, \ldots, \bfg_{i-1},\bfg_i^-,\bfg_{i+1},\ldots, \bfg_d)), \\
    &\Var_i[\gamma(\bfg_{-i})] \coloneqq  
    \Var_{\bfg_i}[\gamma((\bfg_1, \bfg_2, \ldots, \bfg_{i-1},\bfg_i,\bfg_{i+1},\ldots, \bfg_d))|\bfg_{-i}].
    \end{split}
    \end{equation*}
    Then, for any positive $t,$
    $$
    \prob\bigg[\gamma(\bfg) \ge t + \expect[\gamma(\bfg)]\bigg]\le 
    \exp\Bigg(- \min\bigg(\frac{t^2}{4\sum_{j = 1}^d\|\Var_i[\gamma]\|_\infty} ,\frac{t}{2\max_i\|D_i\gamma\|_\infty} \bigg) \Bigg).
    $$    
    Furthermore, for some absolute constant $C,$
    $$
    \|\gamma\|_k\le C\Bigg( \sqrt{k}\sqrt{\sum_{i = 1}^d\|\Var_i[\gamma]\|_\infty} + 
    k \max_i \|D_if\|_\infty
    \Bigg).
    $$
\end{claim}
We bound $\Var_i[\sigma], D_i[\gamma]$ for $\gamma$ defined as in \cref{eq:Kltensorization} via a careful combination of Fourier-theoretic and anticoncentration arguments to obtain \cref{thm:torussmallq}. We also derive \cref{thm:hypercube} as a combination of \cref{eq:Kltensorization} and \cref{claim:bernsteinmcdiarmidrag}.

\subsection{Computational Detection of Latent Space Structure}
\label{sec:prelimcompindist}
To solve \cref{eq:hypoithesistesting}, one observes a certain $n$-vertex graph $G$ and needs to compute a function $f(G)$ based on which to decide between $H_0$ and $H_1.$ The graph $G$ is simply a sequence of $\binom{n}{2}$ bits. It is well-known that any function of $0/1$ vectors is simply a polynomial \cite{ODonellBoolean}. For computationally efficient tests, one needs to be able to compute $f$ in time polynomial in $n.$

\paragraph{Signed Subgraph Counts.}
Most important to the current paper are polynomials corresponding to signed-subgraph counts. Namely, suppose that we want to test between two graph distributions over $n$ vertices in which each edge appears with a marginal probability $p.$ Let $H = \{(i_1, j_1), (i_2,j_2), \ldots, (i_k,j_k)\}$ be any subgraph of $K_n.$ Then, we define the signed weight of $H$ as the polynomial
\begin{equation}
\label{eq:signedweight}
\signedweight_H(G) \coloneqq  
\sum_{(ij)\in E(H)}(G_{ij} - p).
\end{equation}
 For brevity and uniformity with the $\signedweight$ notation, for a set of edges $H = \{(i_1,j_1),(i_2,j_2),\ldots, (i_k,j_k)\},$ denote 
$\unsignedweight_H(G) = \prod_{(ij)\in H}G_{ij} = \indicator[G_{ij} = 1\; \forall (ij)\in H].$
Respectively, the signed count of $H$ in $\bfG$ is 
\begin{equation}
\label{eq:signedcount}
\signedcount_H(G) = \sum_{H_1\subseteq E(K_n)\; : H_1\sim H}
\signedweight_{H_1}(G),
\end{equation}
where the sum is over all subgraphs of $K_n$ isomorphic to $H.$ Note that whenever $H$ has a constant number of edges, the polynomial $\signedcount_H(G)$ is certainly efficiently computable. 

Clearly $\expect_{\bfG\sim \ergraph}\signedcount_H(\bfG) = 0,$ which leads to the following approach to \cref{eq:hypoithesistesting} appearing in \cite{Bubeck14RGG}. Upon observing $G,$ compute $\signedcount_H(G)$ and, if sufficiently close to 0, report $H_0.$ Else  report $H_1.$ Using Chebyshev's inequality, this can be formalized as follows.

\begin{definition}[Success  of the Signed Subgraph Count ] 
\label{def:signedcountsperformance}
We say that signed $H$-count statistical test $\signedcount_H(G)$ succeeds in distinguishing between $\ergraph$ and $\RGG$ if 
\begin{equation}
\label{eq:signedtestdefinition}
    \big|
    \expect_{\bfG\sim \RGG}\big[
    \signedcount_H(\bfG)\big]
    \big|  = \omega \big( 
    \sqrt{
    \Var_{\bfK\sim \ergraph}\big[
    \signedcount_H(\bfK)\big] + 
    \Var_{\bfG\sim \RGG}\big[
    \signedcount_H(\bfG)\big]
    }\big).
\end{equation}
Indeed, if this is the case, one can solve \cref{eq:hypoithesistesting} with Type I and Type II errors both of order $o(1)$ by comparing $\signedcount_H(G)$ to $\frac{1}{2}\expect_{\bfG\sim \RGG}\Big[
    \signedcount_H(\bfG)\Big].$

If, on the other hand, 
\begin{equation}
\label{eq:signedtestdefinitionfail}
    \big|
    \expect_{\bfG\sim \RGG}\big[
    \signedcount_H(\bfG)\big]
    \big|  = o\big( 
    \sqrt{
    \Var_{\bfK\sim \ergraph}\big[
    \signedcount_H(\bfK)\big] + 
    \Var_{\bfG\sim \RGG}\big[
    \signedcount_H(\bfG)\big]
    }\big),
\end{equation}
we say that the signed $H$-count statistical test fails with high probability.
\end{definition}

In this work, we are mostly interested in the case of triangles, $H = C_3,$ and 4-cycles, $H = C_4.$

\paragraph{Low-Degree Tests.} In \cref{def:signedcountsperformance}, one can replace $\signedweight_H(\cdot)$ with any polynomial $f(\cdot)$ and compare 
$$
\big|
    \expect_{\bfG\sim \RGG}\big[
    f(\bfG)\big] - 
\expect_{\bfG\sim \ergraph}\big[
    f(\bfG)\big]
    \big| \text{ and }
\sqrt{
    \Var_{\bfK\sim \ergraph}\big[
    f(\bfK)\big] + 
    \Var_{\bfG\sim \RGG}\big[
    f(\bfG)\big]
    }.  
$$
High-probability success and failure are similarly defined. 

A popular conjecture \cite{hopkins18} states that all polynomial time algorithms for solving (sufficiently noisy) hypothesis testing questions in high-dimension are captured by polynomials of degree $O(\log n).$ Indeed, there is growing evidence in support of this conjecture. Clearly, low degree polynomial tests capture (signed) counts of small subgraphs (note that one can even capture the first $O((\log n)/k)$ moments of the (signed) counts 
of a graph $H$ with $k$ edges), which have proven powerful in detecting random geometric graphs \cite{Bubeck14RGG}, planted cliques and colorings \cite{kothari2023planted}, the number of communities in a stochastic block model \cite{rush2022} and others. 
Low-degree polynomials further capture spectral methods \cite{kunisky2019notes},
constant round approximate message passing algorithms \cite{montanari2022equivalence}, and statistical query algorithms \cite{brennan2021statisticalquery}. Thus, a lot of recent work in high-dimensional statistics has focused on ruling out low-degree polynomial algorithms for statistical problems. This constitutes strong evidence that the respective statistical problems cannot be solved in polynomial time.  

Formally, in the case of \cref{eq:hypoithesistesting} one needs to show that there exists some function $D(n) = \omega(\log n)$ such that  for all degree $D = D(n)$ polynomials $f,$ it is the case that 
$$
\big|
    \expect_{\bfG\sim \RGG}\big[
    f(\bfG)\big] - 
\expect_{\bfG\sim \ergraph}\big[
    f(\bfG)\big]
    \big|  = o\big( 
    \sqrt{
    \Var_{\bfK\sim \ergraph}\big[
    f(\bfK)\big] + 
    \Var_{\bfG\sim \RGG}\big[
    f(\bfG)\big]
    }\big).
$$
One way to prove such an inequality is by bounding the following quantity \cite{hopkins18}:
\begin{equation}
    \advantage_{\le D}: = \max_{f\; : \; deg(f)\le D} 
    \frac{
    \expect_{\bfG\sim \RGG}\big[
    f(\bfG)\big]}{\sqrt{\expect_{\bfK\sim \ergraph}\big[
    f(\bfK)^2\big]}}\,.
\end{equation}
In particular, if $\advantage_{\le D} = 1 + o(1),$ then statistical test $f(\cdot)$ fails with large probability (e.g. \cite{rush2022}).

It turns out that the product structure of $\ergraph$ yields a convenient formula for $\advantage_{\le D}.$
The set of polynomials $\{\signedweight_H\times (p(1-p))^{-|E(H)|/2}\}_{H\subseteq E(K_n)\; 0 \le |E(H)|\le D}$ forms an orthonormal basis of the polynomials of degree up to $D$ with respect to $\ergraph.$ A standard application of the Cauchy-Schwartz inequality (e.g. \cite{hopkins18}) shows that
$$
\advantage^2_{\le D}-1 = 
\sum_{H\subseteq E(K_n)\; 1 \le |E(H)|\le D}
\expect_{\bfG\sim \RGG}\big[\signedweight_H\times (p(1-p))^{-|E(H)|/2}\big]^2.
$$
We summarize in the following proposition.

\begin{proposition} If there exists some $D = \omega(\log n)$ such that 
$$
\sum_{H\subseteq E(K_n)\; 1 \le |E(H)|\le D}
\expect_{\bfG\sim \RGG}\big[\signedweight_H\times (p(1-p))^{-|E(H)|/2}\big]^2  = o(1),
$$
then the Type I plus Type II error of any degree $D$ polynomial in solving \cref{eq:hypoithesistesting} is of order $1 - o(1).$ 
\end{proposition}
We use the bounds from \cref{thm:linftysignedcounts} and this proposition to prove \cref{thm:linftylowdegreeindist}. We note that low-degree polynomials are similarly used in the literature for estimation and refutation tasks (e.g. \cite{Schramm_2022,rush2022}). We discuss this in more detail in \cref{sec:applicationstodimensionestimation} in the context of estimating the dimension of a graph sampled from $\RGG(n,\dtorus,\unif, \sigma^\infty_p,p).$

\section{Detection of Geometry via Subgraph Counts in the \texorpdfstring{$L_\infty$}{Sup-Norm} Model}
\label{sec:complinfty}
The goal of this section is to prove \cref{thm:linfinitydetection,thm:linftylowdegreeindist,prop:linftydimensionestmation} which show lower and upper bounds on distinguishing $\RGG(n,\dtorus, \unif, \sigma^\infty_p,p)$ from Erd\H os-R\'enyi using low-degree polynomials. 
 As discussed in \cref{sec:prelimcompindist}, low-degree polynomials over graphs correspond to (signed) subgraph counts. In \cref{sec:clusterhighlevel}, we describe our ``cluster-expansion'' strategy for bounding the signed (and unsigned) weights of small subgraphs in $\RGG(n,\dtorus, \unif, \sigma^\infty_p,p)$ (given in \cref{thm:linftysignedcounts}). We complete this approach in \cref{appendix:proofof1dtorus,sec:linftyunsignedweights,sec:linftysignedcounts}.
In \cref{sec:linftyperformance} we use these estimates to prove \cref{thm:linfinitydetection,thm:linftylowdegreeindist,prop:linftydimensionestmation}.

\subsection{The Cluster-Expansion Approach to Bounding Expected (Signed) Weights}
\label{sec:clusterhighlevel}
Recall Assumption \eqref{eq:assumption}. 
Note that $\tau^\infty_p$ satisfies 
$(\tau^\infty_p)^d = p.$ Indeed, this is the case since $p = \prob[\|\bfx\|_\infty\le \tau^\infty_p] = \prob[|x_1|_C\le \tau^\infty_p]^d.$ This immediately implies that $\tau^\infty_p = 1 -\lambda_p^\infty,$ where $\lambda_p^\infty = \frac{\log (1/p)}{d}(1 + o(1)).$ We will write $\sigma,\lambda,\tau$ instead of $\sigma_p^\infty,\lambda_p^\infty,\tau_p^\infty $ for brevity.

Fix some subgraph $H\subseteq K_n$ defined by edges $e_1, e_2, \ldots, e_k.$ We want to understand \linebreak 
$\expect_{\bfG\sim \RGG(n,\dtorus, \unif, \sigma^\infty_{p}, p)}[\signedweight_H(\bfG)]$ and $\expect_{\bfG\sim \RGG(n,\dtorus, \unif, \sigma^\infty_{p}, p)}[\unsignedweight_H(\bfG)].$ We will describe how to utilize the \textsf{AND} structure of $L_\infty$ random geometric graphs, described in the introduction, towards this goal. This is done in several steps, which can be similarly applied in other instances of \textsf{AND} structure (see \cref{sec:discussion}). 

\paragraph{Step 1: Factorizing Expected Weights over Independent Coordinates.}
The main reason that the analysis over the $L_\infty$ model is simple is that the different coordinates factorize. Namely, $e_\ell = (i_\ell,j_\ell)$ is an edge if and only if $|x_u^{i_\ell} - x_u^{j_\ell}|_C\le 1-\lambda$ for each coordinate $u\in [d].$ Using the  independence of coordinates under the distribution $\unif(\dtorus),$
\begin{equation}
\label{eq:1dimtoddim}
\expect_{\bfG\sim \RGG(n,\dtorus, \unif, \sigma^\infty_{p}, p)}[\unsignedweight_H(\bfG)] = 
\expect_{\bfG\sim \RGG(n,\mathbb{T}^1, \unif, \sigma^\infty_{1-\lambda}, 1-\lambda)}[\unsignedweight_H(\bfG)]^d.
\end{equation}

\paragraph{Step 2: Computations Over a Single Coordinate via Inclusion-Exclusion.}
It turns out that computing the one-dimensional quantities over the graph complement $\overline{\bfG}$ is simpler than computing them over $\bfG.$ The intuitive reason is that in the complement each edge appears only with very low probability $\lambda = \tilde{\Theta}(1/d).$ In other words, the appearance of an edge is a very restrictive event that largely determines the configuration of latent vectors. Concretely, 
for a set of edges $A,$ denote by $\chi(A)$ the probability that no edge of $A$ appears in $\bfG\sim \RGG (n,\mathbb{T}^1, \unif, \sigma^\infty_{1-\lambda}, 1-\lambda)$, i.e.,
\begin{equation}
\label{eq:chidef}
\chi(A) :=\prob_{\bfG\sim \RGG (n,\mathbb{T}^1, \unif, \sigma^\infty_{1-\lambda}, 1-\lambda)}[G_{ij}=0\text{ for all }ij\in A]\,.
\end{equation}
Equivalently, $\chi(A)$ is the probability that each edge in set $A$ appears in the random geometric graph over $\onetorus$ with connection $\sigma(x,y)_\lambda^{1,>} = \indicator[|x-y|_C\ge 1-\lambda]$ and expected density $\lambda$: 
$$\chi(A) = \expect_{\bfH\sim \RGG(n,\onetorus,\unif, \sigma(x,y)_\lambda^{1,>}, \lambda )}[\unsignedweight_A(\bfH)]\,.$$
The reason this is feasible to compute is that the event $\big\{\sigma(x,y)_\lambda^{1,>} = 1\big\}$ significantly constrains the relative locations of $x, y$ on $\onetorus$: They are at distance $1 - \tilde{O}(d^{-1}),$ so they are nearly diametrically opposite.

Now, one can simply use the principle of inclusion-exclusion to convert the computations in the complement to computations over the original graph:
\begin{equation}
\begin{split}
\label{eq:pie}
        \expect_{\bfG\sim \RGG(n,\mathbb{T}^1, \unif, \sigma^\infty_{1-\lambda}, 1-\lambda)}[\unsignedweight_H(\bfG)] = 
        \sum_{A\subseteq E(H)}
        (-1)^{|E(A)|}
       \chi(A).
\end{split}
\end{equation}

\paragraph{Step 3: Measuring Perturbations From \ER.} Here, we take an approach inspired by statistical-physics of measuring perturbations from the ``ground state'' \ER graph.\footnote{While no familiarity with statistical physics is needed to follow the argument, we will borrow some terminology with the purpose of explaining our approach in familiar language.} Measuring perturbations from \ER is a very natural approach as that is the null model against which we are testing.

We first measure perturbations from \ER at the level of single subgraphs appearing in the 1-dimensional complements, as in \cref{eq:pie}. Namely, define \begin{equation}
\label{eq:piedefinition}
\psi(A)\coloneqq  \chi(A) - \lambda^{|E(A)|}.
\end{equation}
This is the deviation from the probability of all edges in $A$ appearing in $\mathsf{G}(n,\lambda).$ Recalling \cref{eq:pie}, we immediately get a perturbative expression for $\expect_{\bfG\sim \RGG(n,\mathbb{T}^1, \unif, \sigma^\infty_{1-\lambda}, 1-\lambda)}[\unsignedweight_H(\bfG)]$:
\begin{equation}
\label{eq:errordefinition1d}
    \begin{split}
& \expect_{\bfG\sim \RGG(n,\mathbb{T}^1, \unif, \sigma^\infty_{1-\lambda}, 1-\lambda)}[\unsignedweight_H(\bfG)] = \\
& = 
\sum_{A\subseteq H}
(-1)^{|E(A)|}\chi(A) = 
\sum_{A\subseteq H}
(-1)^{|E(A)|}(\psi(A) + \lambda^{|E(A)|})\\
& = 
(1-\lambda)^{|E(H)|} +
\error(H,\lambda),
\qquad \text{ where }\quad 
\error(H, \lambda)\coloneqq \sum_{A\subseteq H}(-1)^{|E(A)|}\psi(A).
    \end{split}
\end{equation}
We interpret each subgraph $A$ of $H$ as a \emph{polymer} and the quantity $(-1)^{|E(A)|}\psi(A)$ as the \emph{weight} of the polymer. In that view, the expression $\error(H,\lambda)$ is the sum of the weights of polymers which captures ``the first order''  deviation from the ground state $(1-\lambda)^{|E(H)|}.$ The quantity $(1-\lambda)^{|E(H)|}$ is a natural ground state for the expected weight of $H$ in one dimension as it corresponds to the expected weight when edges are independent. 

Now, \cref{eq:1dimtoddim} allows us to obtain a similar expression in the $d$-dimensional case:
\begin{equation}
\label{eq:errordefinition}
    \begin{split}
& \expect_{\bfG\sim \RGG(n,\dtorus, \unif, \sigma^\infty_{p}, p)}[\unsignedweight_H(\bfG)]= 
\Big((1-\lambda)^{|E(H)|} +
\error(H,\lambda)\Big)^d = 
\sum_{i = 0}^d \binom{d}{i}(1-\lambda)^{(d-i)|E(H)|}\error(H,\lambda)^i.
    \end{split}
\end{equation}
Again, the term $(1-\lambda)^{d|E(H)|} = p^{|E(H)|}$ corresponding to $i = 0$ is the ``ground state'' weight of $H$ in $\ergraph.$ Each term of the form $\error(H,\lambda)^i$ is composed of products of $i$-tuples of polymer weights, and, thus, can be interpreted as ``the $i$-th order'' perturbation from the ground state. 

\paragraph{Step 4: Bounds on Polymer Weights.} To derive a bound from \cref{eq:errordefinition}, one needs to bound the polymer weights and, subsequently, the $\error(H,\lambda)$ term. Those are relatively straightforward computations as they are all over a 1-dimensional random geometric graph (recall \cref{eq:chidef}). In \cref{appendix:proofof1dtorus}, we prove the following claim, which is used extensively. It shows that perturbations $\psi(A)$ are indeed small.

\begin{claim}
\label{claim:1dtoruscounts}
For every set of edges $A$ such that $V(A)\le {1}/{8\lambda},$ the following hold:
\begin{enumerate} 
    \item If $A$ can be decomposed as $A_1\cup A_2,$ where $|V(A_1)\cap V(A_2)|\le 1,$ then
    $\chi(A) = \chi(A_1)\chi(A_2).$
    \item If $A$ is a forest, then $\chi(A) = \lambda^{|E(A)|}$ and $\psi(A) = 0.$
    \item $\chi(A)\le \lambda^{|V(A)|-1}$ whenever $A$ is connected.
    \item If $A$ is not bipartite, $\chi(A) = 0.$ In particular, $\psi(C_{2m+1}) = - \lambda^{2m+1}.$ 
    \item $|\psi(A)|\le 2\cdot \lambda^{\max\{|V(A)|/2 + 1, |V(A)| - \numconnectedcomponets(A)\}},$ where 
    $\numconnectedcomponets(A)$ denotes the number of connected components of $A.$
    \item If $m\le 1/{8\lambda}$, then
    $\chi(C_m) = \lambda^{m-1}\phi(m-1),$ where 
    $\phi(m)\coloneqq \prob[U_1+ U_2+\cdots + U_{m-1}\in [-1,1]]$ for $U_1, U_2, \ldots, U_{m-1}\iidsim[-1,1]$.\footnote{One can easily check that $\phi(1) = 1,\phi(2) = 3/4,\phi(3) = 2/3,\phi(m-1) = \Theta(m^{-1/2}).$} 
\end{enumerate}
\end{claim}

Using \cref{claim:1dtoruscounts} and triangle inequality, one can easily derive (tight) upper bounds on $\error(H,\lambda).$ In combination with \cref{eq:errordefinition}, this is enough to provide tight bounds on the expected weights of subgraphs. We delay this to \cref{sec:linftyunsignedweights} and  now proceed to the much more subtle case of signed weights.

\paragraph{Step 5: From Unsigned Weights to Signed Weights - Again PIE.} Signed subgraph weights do not immediately factorize over the independent coordinates. That is, while in the unsigned case we have $\indicator[\|x^i - x^j\|_\infty\le 1-\lambda] = \prod_{u = 1}^d\indicator[\|x^i_u - x^j_u\|_\infty\le 1-\lambda],$ no such expression holds for $(\indicator[\|x^i - x^j\|_\infty\le 1-\lambda] - p).$\footnote{One cannot expect $(\indicator[\|x^i - x^j\|_\infty\le 1-\lambda] - p)$ to always be a $d$-th power, for example because $\indicator[\|x^i - x^j\|_\infty\le 1-\lambda] - p$ might be negative while a $d$'th power is always positive when $d$ is even.}
Instead, we reduce to what we know about unsigned weights:
\begin{equation}    \label{eq:signedcountsasunsignedcounts}
    \begin{split}
        & \expect_{\bfG\sim \RGG(n,\dtorus,\unif, \sigma^\infty_p, p)}[\signedweight_H(\bfG)] = 
        \expect\Big[\prod_{i = 1}^k(\bfG_{e_i} - p)\Big] \\
        & =
        \sum_{A\subseteq E(H)}(-p)^{|E(H)|-|E(A)|}
        \expect_{\bfG\sim \RGG(n,\dtorus,\unif, \sigma^\infty_p, p)}[\unsignedweight_A(\bfG)]\\
        & =  
        \sum_{A\subseteq E(H)}(-1)^{|E(H)|-|E(A)|}(1-\lambda)^{d(|E(H)|-|E(A)|))}
        \expect_{\bfG\sim \RGG(n,\dtorus,\unif, \sigma^\infty_p, p)}[\unsignedweight_A(\bfG)].
    \end{split}
\end{equation}
Using \cref{eq:errordefinition} for any $A\subseteq H,$ we obtain 
\begin{equation}
    \begin{split}
    \label{eq:signedcountsexpansion}
        & \expect_{\bfG\sim \RGG(n,\dtorus,\unif, \sigma^\infty_p, p)}[\signedweight_H(\bfG)] =\\
        & =  
        \sum_{A\subseteq H}(-1)^{|E(H)|-|E(A)|}(1-\lambda)^{d(|E(H)|-|E(A)|)}
        \sum_{i = 0}^d 
        \binom{d}{i} 
        (1-\lambda)^{(d-i)|E(A)|}\error(A,\lambda)^i\\
        & = 
        \sum_{i = 0}^d
        \binom{d}{i}
        (1-\lambda)^{(d-i)|E(H)|}
        \sum_{A\subseteq H}
        (-1)^{|E(H)| - |E(A)|}
        (1-\lambda)^{i(|E(H)| - |E(A)|)}
        \error(A,\lambda)^i.
    \end{split}
\end{equation}

\paragraph{Step 6: The Cluster Expansion Perspective on Signed Subgraph Counts.} Again, $\error(A,\lambda)^i$ is the sum of products of $i$-tuples of weights of polymers or, equivalently, ``the $i$-th order'' deviation from the ground state. When we sum over $A\subseteq H,$ each $i$-tuple will appear with some coefficient which captures the compatibility of the respective $i$-tuple. Specifically, in \cref{eq:errorexpansionshort}, we expand (using a similar approach to the formal derivation of the cluster expansion formula, e.g. in \cite[Chapter 5]{friedli_velenik_2017}) each $\error(A,\lambda)^i$ as a sum of $i$-tuples of polymer weights:
    \begin{align}
    \label{eq:polymerexpansion1}
            &\sum_{A\subseteq E(H)}
        (-1)^{|E(H)| - |E(A)|}
        (1-\lambda)^{i(|E(H)| - |E(A)|)}
        \error(A,\lambda)^i
        \\
        \label{eq:polymerexpansion}
        &=  
        \sum_{K_1,K_2,\ldots, K_i\subseteq E(H)}
        (1 - (1-\lambda)^i)^{|E(H)| - |E(K_1\cup \cdots \cup K_i)|}
        (-1)^{|E(K_1)| + \cdots + |E(K_i)|}
        \psi(K_1)\psi(K_2)\cdots \psi(K_i).
    \end{align}

\cref{eq:polymerexpansion} is the $i$-th order of the  ``cluster expansion'' for signed subgraph weights. 
Note that the ground state captured by the terms appearing when $i = 0$ vanishes as 
\begin{equation*}
     \label{eq:iequal0term}
      \sum_{A\subseteq H}(-1)^{|E(H)| - |E(A)|} = 
 \sum_{j = 0}^{|E(H)|}\binom{|E(H)|}{j}(-1)^{|E(H)| - j} = 0.
 \end{equation*}
The fact that the ground state is 0 is intuitive, because in the case of independent edges each signed subgraph weight has expectation 0. It remains to interpret the ``soft'' compatibility criterion captured by 
    the coefficient
    $(1 - (1-\lambda)^i)^{|E(H)| - |E(K_1\cup K_2\cdots \cup K_i)|}.$ Whenever $|E(K_1\cup K_2\cdots \cup K_i)|$ is small, this coefficient is very small as $1 - (1-\lambda)^i = \tilde{O}(d^{-1}).$ This means that polymers $K_1, K_2, \ldots, K_i$ are more compatible when 
    $|E(K_1\cup K_2\cdots \cup K_i)|$ is smaller. 
    Such a compatibility criterion should not be surprising---it says that the subgraphs $K_j$ corresponding to different coordinates are more compatible when they are more similar (so that their union does not blow up).

Our final goal will be to bound the $i$-th order deviation from the ground state for each $i.$ We will prove the following proposition which immediately gives the desired bound on signed subgraph counts \cref{thm:linftysignedcounts}. 

\begin{proposition}
\label{prop:innersumbound}
Recall the definition of $\error(A,\lambda)$ in \eqref{eq:errordefinition}.
For $1\le i \le d,$ the following inequality holds: 
\begin{equation}
    \Bigg|\sum_{A\subseteq H}
        (-1)^{|E(H)| - |E(A)|}
        (1-\lambda)^{i(|E(H)| - |E(A)|)}
        \error(A,\lambda)^i\Bigg| \le 
        \frac{1}{(4d)^i}\times\Big(\frac{(\log d)^{C}}{d}\Big)^{|V(H)|/2}.
\end{equation}
\end{proposition}

In proving \cref{prop:innersumbound}, there are two conceptually distinct regimes for $i,$ as is common in the asymptotic analysis of sums (in particular, in the cluster-expansion formula). 
\begin{enumerate}
    \item \textbf{Small values of $i$} (\cref{sec:smallibounds}).
    We use \cref{eq:polymerexpansion}.
    By \cref{claim:1dtoruscounts}, $|\psi(K_j)|\le (2\lambda)^{|V(K_j)|/2 + 1}.$ Thus, whenever $\sum_{j = 1}^i |V(K_j)|$ is large, the total weight $|\psi(K_1)\psi(K_2)\cdots \psi(K_i)|$ of the $i$-tuple is low. An energy-entropy trade-off phenomenon occurs---and there are very few $i$-tuples for which $\sum_{j = 1}^i |V(K_j)|$is small:
     \begin{claim}[Rephrasing of \cref{claim:lowprobabilityofsmallsum}]
    \label{claim:lownumberofsmallsum}
    Let $i \ge 2, 0 \le b \le i$ be integers and $a>0$ be a real number. Then, the number of $i$-tuples $K_1, K_2, \ldots K_i$ of $H$ such that $\sum_{j = 1}^i |V(K_j)|\le ab$ is at most 
    $$
    \exp\Big(b(\log i) + a^2i\log |E(H)| + |E(H)|b\Big).
    $$
    \end{claim}
    To handle the few potentially ``high-energy'' terms - for which $\sum_{j = 1}^i |V(K_j)|$ is small - we 
    use a comparison inequality. Namely,  $|\psi(K_j)|\le (2\lambda)^{|V(K_j)|-\numconnectedcomponets(K_j)}$ from \cref{claim:1dtoruscounts} for all $j$ and the fact that the quantity $|V(K)|-\numconnectedcomponets(K)$ is subadditive under edge unions (see \cref{claim:subadditivity}) allows us to 
    bound $|\psi(K_1)\psi(K_2)\cdots \psi(K_i)|$ by $|\psi(K_1\cup K_2\cdots \cup K_i)|$. This is useful because it makes all quantities in \eqref{eq:polymerexpansion} functions of $K_1\cup K_2\cdots \cup K_i$ (up to signs).

    \item \textbf{Large values of $i$ }(\cref{sec:largeibound}). ``High degree'' terms are asymptotically irrelevant due to a rapid enough decay of $ \error(A,\lambda)^i$ in \cref{eq:polymerexpansion1}. Specifically, one can prove that for all $A\subset E(H),$
    $|\error(A,\lambda)|\le d^{-3 + o_d(1)}$ (see \cref{claim:errorbound}) by applying triangle inequality over all subgraphs $K$ of $A$ (recall the definition of $\error(A,\lambda)$ in \eqref{eq:errordefinition}) and using  that $|\psi(K)|\le (2\lambda)^{\max(3, |V(K)|/2 + 1)}$ from \cref{claim:1dtoruscounts}. 
\end{enumerate}
We now fill in the details for deriving the bounds on expected weights and expected signed weights.

\subsection{Proof of \texorpdfstring{\cref{claim:1dtoruscounts}}{1 dimensional torus}}
\label{appendix:proofof1dtorus}
Let $\bfH\sim \RGG(n,\onetorus,\unif, \sigma(x,y)_\lambda^{1,>}, \lambda ).$
 \paragraph{Item 1.} Observe that if $A_1$ and $A_2$ do not share a vertex, then, clearly $\unsignedweight_{A_1}(\bfH)$ and $\unsignedweight_{A_2}(\bfH)$ are independent as they depend on disjoint sets of latent vectors. If $|V(A_1)\cap V(A_2)| = 1,$ we use the fact that $\sigma(x,y)_\lambda^{1,>}$ only depends on $x-y$ as follows. Let $V(A_1) = \{u_0, u_1, \ldots, u_k\}, V(A_2) = \{v_0, v_1, \ldots, v_r\},$ where $u_0 = v_0.$ Note that $A_1, A_2$ have no common edges. Then
\begin{equation}
    \begin{split}
    & \chi(A) = 
        \expect_{\bfH\sim \RGG(n,\onetorus,\unif, \sigma(x,y)_\lambda^{1,>}, \lambda )}
        \Big[
        \prod_{(u_s, u_t)\in E(A_1)}
        H_{(u_s,u_t)}
        \prod_{(v_k,v_\ell)\in E(A_2)}
        H_{(v_k,v_\ell)}
        \Big]\\
        & = 
        \expect_{\bfx^{u_0}, \bfx^{u_1}, \ldots, \bfx^{u_k}, 
        \bfx^{v_1}, \ldots, \bfx^{v_r}\iidsim\unif(\onetorus)
        }\Big[
        \prod_{(u_s, u_t)\in E(A_1)}
        \sigma(\bfx^{u_s},\bfx^{u_t})_\lambda^{1,>}
        \prod_{(v_k,v_\ell)\in E(A_2)}
        \sigma(\bfx^{v_k},\bfx^{v_\ell})_\lambda^{1,>}
        \Big]\\
        & = \expect_{\bfz,\bfx^{u_0}, \bfx^{u_1}, \ldots, \bfx^{u_k}, 
        \bfx^{v_1}, \ldots, \bfx^{v_r}\iidsim\unif(\onetorus)
        }\Big[
        \prod_{(u_s, u_t)\in E(A_1)}
        \sigma(\bfx^{u_s},\bfx^{u_t})_\lambda^{1,>}
        \prod_{(v_k,v_\ell)\in E(A_2)}
        \sigma(\bfx^{v_k} + \bfz,\bfx^{v_\ell} + \bfz)_\lambda^{1,>}
        \Big]\\
        & = \expect\Big[
        \prod_{(u_s, u_t)\in E(A_1)}        \sigma(\bfx^{u_s},\bfx^{u_t})_\lambda^{1,>}\Big]\times 
        \expect\Big[
        \prod_{(v_k,v_\ell)\in E(A_2)}
        \sigma(\bfx^{v_k} + \bfz,\bfx^{v_\ell} + \bfz)_\lambda^{1,>}
        \Big] = 
        \chi(A_1)\chi(A_2),
    \end{split}
\end{equation}
where we used the fact that the set of vectors 
$\bfx^{u_0}, \bfx^{u_1}, \ldots, \bfx^{u_k}, 
 \bfx^{u_0} + \bfz, \bfx^{v_1}+\bfz, \ldots, \bfx^{v_r}+\bfz$ are independent. 
 \paragraph{Item 2.} Follows from an inductive application of item 1 and the fact that each edge appears marginally with probability $\lambda$ in $\bfH.$  \paragraph{Item 3.} Let $T$ be a spanning tree of $A$ with $V(A) - 1$ edges. The simple fact $\unsignedweight_T(\bfH)\ge \unsignedweight_A(\bfH)$ and item 2 give the desired inequality.  \paragraph{Item 4.} Suppose that $A$ is not bipartite. Then it has an odd cycle formed by vertices $i_1,i_2, \cdots i_{2k+1},i_{2k+2} = i_1$ of length $2k+1\le \frac{1}{8\lambda}.$ We will show that for any latent vectors $x_{i_1},x_{i_2} , \ldots, x_{i_{2k+1}},$ it is the case that there exists some $t \in [2k+1]$ for which $\sigma(x_{i_t},x_{i_{t+1}})_\lambda^{1,>} = 0.$ Indeed, otherwise $|x_{i_t} - x_{i_{t+1}}|_C\ge 1-\lambda$ and $|x_{i_{t+1}} - x_{i_{t+2}}|_C\ge 1-\lambda $
imply that $|x_{i_t} - x_{i_{t+2}}|_C\le 2\lambda$ holds for each $t.$ However, this means that 
$|x_1 - x_{2k+1}|\le k\cdot  2\lambda <1-\lambda,$ which means that $\sigma(x_{i_1},x_{i_{2k+1}})_\lambda^{1,>} = 0.$ 

 \paragraph{Item 5.} Observe that $A$ has a spanning forest $T$ on $V(A) - \numconnectedcomponets(A)$ edges. This gives the bound
 $$|\psi(A)|\le |\chi(A)| + \lambda^{|E(A)|}\le 
 |\chi(T)| + \lambda^{|E(T)|} = 2\lambda^{|E(T)|} = 2\lambda^{|V(A)| - \numconnectedcomponets(A)}.$$ The only remaining case is when 
 $|V(A)|- \numconnectedcomponets(A)< |V(A)|/2 + 1$ or, equivalently, $\numconnectedcomponets(A) > |V(A)|/2 - 1.$ Note, however, that since $A$ is defined by a set of edges, there are no isolated vertices and, so, 
 $\numconnectedcomponets(A) \le |V(A)|/2.$ Thus, we have two cases. First, $\numconnectedcomponets(A) = |V(A)|/2,$ in which case $A$ must be the union of $|V(A)|/2$ disjoint edges, but then $\psi(A) = 0$ by item 2. Or, $\numconnectedcomponets(A) = |V(A)|/2 - 1/2,$ so 
 $A$ must be the union of a triangle and $(|V(A)|-3)/2$ disjoint edges. In that case, using items 1, 2, and 4, $\chi(A) = 0,$ so $\psi(A) = -\lambda^{|E(A)|} = -\lambda^{|V(A)|/2 + 3/2}$.
 \paragraph{Item 6.} Let $C_m$ be the cycle on $m$ vertices $1,2,\ldots, m.$ Note that whenever $(ij)$ is an edge in $\bfH,$ then $x_i = x_j + 1 + \lambda_{ij},$ where $\lambda_{ij}\in [-\lambda, \lambda].$ Thus, using that the path $1,2,\ldots, m$ is a tree and item 2,
 
\begin{align*}
 &\chi(C_m) = 
\expect[\unsignedweight_{C_m}(\bfH)] = 
\prob[H_{1,m} = 1,H_{1,2} = 1, \ldots, H_{m-1, m}]\\
&
=
\prob\big[H_{1,m} = 1\big| H_{1,2} = 1, \ldots, H_{m-1, m}\big]\prob\big[H_{1,2} = 1, \ldots, H_{m-1, m}\big]\\
&  
= \prob\big[|x_1 - x_m|_C\ge 1 - \lambda\big| x_{i+1} = x_i + 1 + \lambda_{i,i+1}, |\lambda_{i,i+1}|\le \lambda\forall i\big]\cdot\lambda^{m-1}\\ 
&
= \prob\Big[|x_1 - x_m|_C\ge 1 - \lambda\big| x_m = 1 + x_1 + \sum_{i = 1}^{m-1}\lambda_{i,i+1}, |\lambda_{i,i+1}|\le \lambda\forall i\Big]\cdot \lambda^{m-1}\\
& = 
\prob\Big[\sum_{i = 1}^{m-1}\lambda_{i,i+1} \in [-\lambda, \lambda]\big|
|\lambda_{i,i+1}|\le \lambda\forall i
\Big]\cdot \lambda^{m-1} = \lambda^{m-1}\phi(m-1).
\end{align*}
This completes the proof. \hfill \qed

\begin{remark}
\label{rmk:onragfactorization} Parts 1, 2, 3, and the bound $|\psi(A)|\le 2\cdot \lambda^{|V(A)|-\numconnectedcomponets(A)}$ hold for \textit{any} random algebraic graph of density $\lambda,$ without any condition on the size of $\lambda.$ The proof is the same. In particular, for any forest $F$ and any random geometric graph $\RGG(n,\dtorus, \unif, \sigma^q_p, p),$ we have 
$
\expect_{\bfG\sim \RGG(n,\dtorus, \unif, \sigma^q_p, p)}
\big[
\unsignedweight_F(\bfG)
\big] = p^{|E(F)|}.
$
\end{remark}

\subsection{Unsigned Weights of Small Subgraphs}
\label{sec:linftyunsignedweights}
Here, we compute the unsigned weight of a cycle. The argument for graphs beyond cycles is similar and is done in full detail in \cref{appendix:linftygraphcounts}, but we sketch here the necessary modifications.

\begin{proposition}
\label{prop:linftycyclecounts} Suppose that $1\le m \le {1}/{8\lambda}.$ Then, 
\begin{equation}
    \expect_{\bfG\sim \RGG(n,\dtorus,\unif,\sigma_p^\infty,p)}\Big[\unsignedweight_{C_m}(\bfG)\Big]=\\
    \begin{cases}
        p^m\Big(1 + \frac{d\lambda^m}{(1-\lambda)^m} + 
        {O}\big({d^2\lambda^{2m}}\big)\Big)
        \quad \text{when $m$ is odd,}\\[10pt] 
         p^m\Big(1 + \frac{d(\lambda^{m-1}\phi(m-1) - \lambda^m)}{(1-\lambda)^m} + {O}\big({d^2\lambda^{2(m-1)}}\big)\Big)
        \quad \text{when $m$ is even}.\\
    \end{cases}
\end{equation}
\end{proposition}
\begin{proof} We use \cref{eq:errordefinition}. In the odd case $C_m$ is not bipartite and item 4 of \cref{claim:1dtoruscounts} applies, yielding
\begin{align*}
    & \expect_{\bfG\sim \RGG(n,\dtorus,\unif,\sigma_p^\infty,p)}\Big[\unsignedweight_{C_m}(\bfG)\Big]=
    \big((1-\lambda)^m + \lambda^{m}\big)^d = 
    (1-\lambda)^{md}(1 + \lambda^m/(1-\lambda)^m)^d\\
    &= 
    p^m\Big(1 + d\frac{\lambda^m}{(1-\lambda)^m} + \sum_{k = 2}^d\binom{d}{k}\frac{\lambda^{mk}}{(1-\lambda)^{mk}}\Big).
\end{align*}   
The statement follows as the sum can be bounded by $\sum_{j = 2}^{\infty}(d\lambda^m)^k/(1-\lambda)^{mk}.$ 
Now, clearly, there is exponential decay in the sum as $d\lambda^m/(1-\lambda)^m\le d\lambda^3/(1-\lambda)^3 = o(1).$ Finally, note that $(1-\lambda)^m\ge (1-\lambda)^{8/\lambda} = \Omega(1).$ The even case is the same, except that we use item 6 of \cref{claim:1dtoruscounts}, which gives
\begin{equation*}
\expect_{\bfG\sim \RGG(n,\dtorus,\unif,\sigma_p^\infty,p)}\big[\unsignedweight_{C_m}(\bfG)\big]=
    \big((1-\lambda)^m +
    \lambda^{m-1}\phi(m-1)
    - \lambda^{m}\big)^d.
    \qedhere
\end{equation*}
\end{proof}

\begin{remark}
    We get arbitrarily better precision in \cref{prop:linftycyclecounts} by keeping $k\ge 2$ terms in the expansion of $(1 + \lambda^m/(1-\lambda)^m)^d$.
\end{remark}

\begin{corollary}
\label{cor:linftysignedcycles} 
Suppose that $1\le m \le {1}/{8\lambda}.$ Then, 
\begin{equation}
    \expect_{\bfG\sim \RGG(n,\dtorus,\unif,\sigma_p^\infty,p)}\Big[\signedweight_{C_m}(\bfG)\Big]=\\
    \begin{cases}
        p^m\Big(\frac{d\lambda^m}{(1-\lambda)^m} + 
        {O}\big({d^2\lambda^{2m}}\big)\Big)
        \quad \text{when $m$ is odd,}\\[10pt]
         p^m\Big(\frac{d(\lambda^{m-1}\phi(m-1) - \lambda^m)}{(1-\lambda)^m} + {O}\big({d^2\lambda^{2(m-1)}}\big)\Big)
        \quad \text{when $m$ is even}.\\
    \end{cases}
\end{equation}
\end{corollary}
\begin{proof} Using the definition of $\signedweight_{C_m}$ in \cref{eq:signedweight} and  \cref{rmk:onragfactorization},
\begin{align*}
& \expect_{\bfG\sim \RGG}\Big[\signedweight_{C_m}(\bfG)\Big] = \expect_{\bfG\sim \RGG}\Big[\unsignedweight_{C_m}(\bfG)\Big] + 
\sum_{F\subsetneq E(C_m)}
(-p)^{m-|E(F)|}\expect_{\bfG\sim \RGG}\Big[\unsignedweight_{F}(\bfG)\Big]\\
& = 
\expect_{\bfG\sim \RGG}\Big[\unsignedweight_{C_m}(\bfG)\Big] + \sum_{j = 1}^m\binom{m}{j}(-p)^{m-j}p^j = 
\expect_{\bfG\sim \RGG}\Big[\unsignedweight_{C_m}(\bfG)\Big] - p^m,
\end{align*}
which is enough by \cref{prop:linftycyclecounts}.
\end{proof}

Finally, to derive the bounds on the weights of arbitrary subgraphs, we use the truncated inclusion-exclusion inequality in place of \cref{eq:pie}. Namely, for any odd number $t,$ 
\begin{equation}
\begin{split}
\label{eq:pietruncated}
        \sum_{A\subseteq E(H)\; : \;  |A|\le t}
        (-1)^{|E(A)|}
       \chi(A)\le 
        \expect_{\bfG\sim \RGG(n,\mathbb{T}^1, \unif, \sigma^\infty_{1-\lambda}, 1-\lambda)}[\unsignedweight_H(\bfG)] \le  
        \sum_{A\subseteq E(H)\; : \;  |A|\le t+1}
        (-1)^{|E(A)|}
       \chi(A).
\end{split}
\end{equation}
This yields the following proposition, proven in \cref{appendix:linftygraphcounts}.

\begin{proposition}
\label{prop:linftygraphcounts} 
Let $H = \{(i_1, j_1), (i_2, j_2),\ldots (i_k,j_k)\}$ be a set of edges and let $m$ be the length of the shortest cycle formed by these edges. Let $N(u)$ be the number of cycles of length $u$ in $H.$
Suppose further that Assumption \eqref{eq:assumption} holds and $k^{m+2} = o(1/\lambda) = o(d/\log(1/p))$ and let $\phi(u) = \Theta(u^{-1/2})$ be defined as in \cref{claim:1dtoruscounts}.
Then, for $\bfG\sim \RGG(n,\dtorus,\unif,\sigma_p^\infty,p),$ 
\begin{equation}
\Big|\expect\Big[\unsignedweight_H(\bfG)\Big]\Big|=\\
    \begin{cases}
        p^k\Big(1 + d\big(N(m) + \phi(m+1)N(m+1)\big)\lambda^m(1 + o(1))\Big)
        \quad \text{when $m$ is odd,}\\[10pt]
         p^k\Big(1 + d\phi(m)N(m)\lambda^{m-1}(1 + o(1))\Big)
        \quad \text{when $m$ is even}.\\
    \end{cases}
\end{equation}
\end{proposition}

We remark that the only restrictive condition in this theorem is $k^{m+2} = o(1/\lambda).$ Note, however, that it still covers a wide range of cases. Indeed, suppose that $d = \poly(n).$ As $k\le m,$ it can be applied whenever $k = |E(H)| = o(\log d/\log \log d).$ If, furthermore, $m$ is a constant (say $m \in \{3,4\}$), it can be applied to very large graphs with polynomial number of edges, i.e. $|E(H)| = d^{1/(m+2) - o(1)}.$

\subsection{Signed Weights of Small Subgraphs}
\label{sec:linftysignedcounts}
We now work towards proving \cref{thm:linftysignedcounts}.
Fix $H$ with at most $(\log d)^{5/4}/(\log \log d)$ edges. We also assume $H$ is 2-connected. Indeed, otherwise $H$ can be decomposed into two graphs $H_1, H_2$ which share at most one vertex. Using the same argument as in the proof of part 1 of \cref{claim:1dtoruscounts},
$$
\expect_{\bfG\sim \RGG(n,\dtorus,\unif, \sigma^\infty_p, p)}\Big[\signedweight_H(\bfG)\Big] = 
\expect_{\bfG\sim \RGG}\Big[\signedweight_{H_1}(\bfG)\Big]\times
\expect_{\bfG\sim \RGG}\Big[\signedweight_{H_2}(\bfG)\Big]
$$
and we can induct as $|V(H_1)| + |V(H_2)|\ge |V(H)|,|E(H_1)| + |E(H_2)|\ge |E(H)| .$ In particular, the 2-connectivity assumption means that $|V(H)|\le |E(H)|.$ We also assume that $H$ has at least $4$ edges as the other cases are covered in \cref{claim:1dtoruscounts,cor:linftysignedcycles} (for triangles, we get $\expect_{\bfG\sim \RGG}\big[\signedweight_{C_3}(\bfG)\big] = p^3(\log(1/p)/d)^{2}$ and for acyclic graphs, 0).

\subsubsection{Proof of \cref{thm:linftysignedcounts} 
Assuming \cref{prop:innersumbound}}
We first show how \cref{prop:innersumbound} implies 
\cref{thm:linftysignedcounts}.

\begin{proof}
 Using \cref{eq:signedcountsexpansion} and \cref{eq:iequal0term}, we compute: 
\begin{equation}
    \begin{split}
        \Bigg|\expect_{\bfG\sim \RGG}[\signedweight_H(\bfG)]\Bigg|
        & \le 
        \Bigg|
        \sum_{i = 0}^d
        \binom{d}{i}
        (1-\lambda)^{(d-i)|E(H)|}
        \sum_{A\subseteq H}
        (-1)^{|E(H)| - |E(A)|}
        (1-\lambda)^{i(|E(H)| - |E(A)|)}
        \error(A,\lambda)^i\Bigg|\\
        & \le 
        \sum_{i = 1}^d
        \binom{d}{i}
        (1-\lambda)^{(d-i)|E(H)|}
        \Bigg|
        \sum_{A\subseteq H}
        (-1)^{|E(H)| - |E(A)|}
        (1-\lambda)^{i(|E(H)| - |E(A)|)}
        \error(A,\lambda)^i\Bigg|\\
        & \le 
        \sum_{i = 1}^d 
        d^i (1-\lambda)^{d|E(H)|}(1-\lambda)^{-i|E(H)|}
        \frac{1}{(4d)^i}\times 
        \Big(\frac{(\log d)^{C}}{d}\Big)^{|V(H)|/2}\\
        & = 
        (1-\lambda)^{d|E(H)|}
        \Big(\frac{(\log d)^{C}}{d}\Big)^{|V(H)|/2}
        \sum_{i = 1}^d 
        \Big(\frac{1}{4(1-\lambda)^{|E(H)|}}\Big)^i\\
        & = 
        p^{|E(H)|}
        \Big(\frac{(\log d)^{C}}{d}\Big)^{|V(H)|/2}
        \sum_{i = 1}^d 
        \Big(\frac{1}{4(1-\lambda)^{|E(H)|}}\Big)^i.
    \end{split}
\end{equation}
Now, observe that $(1-\lambda)^{|E(H)|}\ge 
1 - \lambda|E(H)|\ge 1 - O((\log d)^{9/4}/d)\ge 1/2 
$
for all large enough $d.$ Thus, 
$4(1-\lambda)^{|E(H)|}>2$ and so 
$
\sum_{i = 1}^d 
        \Big(\frac{1}{4(1-\lambda)^{|E(H)|}}\Big)^i\le 1,
$
which completes the proof.\end{proof}

What remains is to prove \cref{prop:innersumbound}. As described in Step 6 of \cref{sec:clusterhighlevel}, there are two conceptually different regimes.

\subsubsection{Proof of \texorpdfstring{\cref{prop:innersumbound}}{ithordererror} for Small Values of \texorpdfstring{$i.$}{i}} 
\label{sec:smallibounds}
Suppose that $i< 11|V(H)|/41.$\footnote{In principle, any constant in the interval $(1/4, 1/2)$ would work for the proof, but constants less than $3/10$ reduce the amount of case work, hence the peculiar choice of $11/41.$}
The first step towards proving \cref{prop:innersumbound} is expanding \cref{eq:signedcountsexpansion}. 

\begin{equation}
\label{eq:errorexpansion}
    \begin{split}
        &\sum_{A\subseteq H}
        (-1)^{|E(H)| - |E(A)|}
        (1-\lambda)^{i(|E(H)| - |E(A)|)}
        \error(A,\lambda)^i\\
        & = 
        \sum_{A\subseteq H}
        (-1)^{|E(H)| - |E(A)|}
        (1-\lambda)^{i(|E(H)| - |E(A)|)}
        \Big(
        \sum_{K\subseteq A}
        (-1)^{|E(K)|}
        \psi(K)
        \Big)^i\\
        & = 
        \Bigg(
        \sum_{K_1,K_2,\ldots, K_i\subseteq H}
        (-1)^{|E(K_1)| + |E(K_2)| + \cdots + |E(K_i)|}
        \psi(K_1)\psi(K_2)\cdots \psi(K_i)\times\\
        & \quad \quad\quad\quad \times
        \sum_{A \subseteq H \; : \; K_j \subseteq A \; \forall j}
        (-1)^{|E(H)|-|E(A)|}
        (1-\lambda)^{i(|E(H)| - |E(A)|)}\Bigg).
    \end{split}
\end{equation}
Let $\mathcal{K} = K_1\cup K_2\cdots K_i.$ Then, in the 
last sum, we perform a summation over all $A$ such that 
$\mathcal{K}\subseteq A \subseteq H.$ In particular, we obtain

\begin{equation}
\label{eq:pieterminerrorexpansion}
\begin{split}
& \sum_{A \subseteq H \; : \; K_j \subseteq A \; \forall j}(-1)^{|E(H)|-|E(A)|}(1-\lambda)^{i(|E(H)| - |E(A)|)}\\
& = 
\sum_{\mathcal{K}^c
\subseteq 
H\backslash \mathcal{K}
}
(-1)^{|E(\mathcal{K}^c)|}
(1-\lambda)^{i |E(\mathcal{K}^c)|}\\
& = 
\sum_{t = 0}^{|E(H)| - |E(\mathcal{K})|}
\binom{|E(H)| - |E(\mathcal{K})|}{t}
(-1)^t
(1-\lambda)^{it}\\
& = 
\big(1 - (1-\lambda)^i\big)^{|E(H)| - |E(\mathcal{K})|}\le 
(\lambda i)^{|E(H)| - |E(\mathcal{K})|}, 
\end{split}
\end{equation}
where in the last line we used Bernoulli's inequality $(1-\lambda)^  i\ge 1 - \lambda i.$  Now, using \cref{eq:pieterminerrorexpansion}, we can rewrite the RHS of
\cref{eq:errorexpansion} as
\begin{equation}
\label{eq:errorexpansionshort}
    \begin{split}
        \sum_{K_1,K_2,\ldots, K_i\subseteq H}
        (1 - (1-\lambda)^i)^{|E(H)| - |E(\mathcal{K})|}
        (-1)^{|E(K_1)| + |E(K_2)| + \cdots + |E(K_i)|}
        \psi(K_1)\psi(K_2)\cdots \psi(K_i)
    \end{split}
\end{equation}
Using the triangle-inequality, we bound \cref{eq:errorexpansionshort} by 
\begin{equation}
\label{eq:errorexpansiontriangleineq}
    \begin{split}
        \sum_{K_1,K_2,\ldots, K_i\subseteq H}
        (1 - (1-\lambda)^i)^{|E(H)| - |E(\mathcal{K})|}
        \Big|
        \psi(K_1)\psi(K_2)\cdots \psi(K_i)\Big|.
    \end{split}
\end{equation}

Now, we will bound the quantity $|\psi(K_1)\psi(K_2)\ldots \psi(K_i)|$ in two different ways.

\begin{observation}
\label{obs:productofdeviationsbound}
The value of $
    \big|
        \psi(K_1)\psi(K_2)\cdots \psi(K_i)
    \big|
    $ is less than each of 
    \begin{enumerate}
        \item $\displaystyle \prod_{j =1}^ i 
    (2\lambda)^{|V(K_j)|/2 +  1},$ and
        \item $(2\lambda)^{|V(\mathcal{K})|- \numconnectedcomponets(\mathcal{K})}.$
    \end{enumerate}
\end{observation}
To prove Observation~\ref{obs:productofdeviationsbound}, we will need the following claim, whose proof is deferred to \cref{appendix:fromsignedsubgraphcounts}.

\begin{claim} 
\label{claim:subadditivity}
Suppose that $G$ is a graph and $G_1$ and $G_2$ are two (not necessarily induced) subgraphs such that $E(G_1)\cup E(G_2) = E(G).$ Then, 
$
|V({G})| - \numconnectedcomponets(G)\le 
|V(G_1)| - \numconnectedcomponets(G_1) + 
|V(G_2)| - \numconnectedcomponets(G_2).
$
\end{claim}

\begin{proof}[Proof of Observation~\ref{obs:productofdeviationsbound}]
Using part 5 of \cref{claim:1dtoruscounts} on each $\psi(K_i)$ yields the first bound.
For the second bound, we again apply part 5 of \cref{claim:1dtoruscounts} on each $\psi(K_i)$ and then we repeatedly apply Claim~\ref{claim:subadditivity}:
\begin{equation*}
    \Big|
        \psi(K_1)\psi(K_2)\cdots \psi(K_i)
    \Big|
    \le 
    \prod_{j =1}^ i 
    (2\lambda)^{|V(K_j)| -\numconnectedcomponets(K_j)}\le 
    (2\lambda)^{|V(\mathcal{K})|- \numconnectedcomponets(\mathcal{K})}.\qedhere
\end{equation*}
\end{proof}

We now proceed to bound \cref{eq:errorexpansiontriangleineq} for a fixed fixed $i$-tuple $K_1, K_2,\ldots, K_i.$

\begin{observation}
\label{obs:productofdeviationsboundextrafactor}
    The value of 
    $
    (1 - (1-\lambda)^i)^{|E(H)| - |E(\mathcal{K})|}
    \Big|
    \psi(K_1)\psi(K_2)\cdots \psi(K_i)
    \Big|$
    is less than each of 
        \begin{enumerate}
            \item $(2\lambda)^{i + \sum_{j = 1}^i |V(K_i)|/2}$, and
            \item $(2i\lambda)^{|V(H)|-1}.$
        \end{enumerate}
\end{observation}
We will need the following combinatorial inequality, proved in \cref{appendix:fromsignedsubgraphcounts}.
\begin{claim}
\label{claim:numcintwoconnectedsubgraph}
For any 2-connected graph $H$ and any (not necessarily induced) subgraph $\mathcal{K}$ of $H,$
    $$|E(H)| - |E(\mathcal{K})|\ge \numconnectedcomponets(\mathcal{K}) + |V(H)| - |V(\mathcal{K})|-1.$$ 
\end{claim}

\begin{proof}
The first bound follows directly from part 1 in \cref{obs:productofdeviationsbound} and the fact that $|1 - (1-\lambda)^i|\le 1.$
For the second bound, we use \cref{claim:numcintwoconnectedsubgraph} to obtain
\begin{align*}
       &  
    (1 - (1-\lambda)^i)^{|E(H)| - |E(\mathcal{K})|}
    \Big|
    \psi(K_1)\psi(K_2)\cdots \psi(K_i)
    \Big|\\
        & \le
        (2\lambda)^{|V(\mathcal{K})| - \numconnectedcomponets(\mathcal{K})}
        (i\lambda )^{|V(H)| - |V(\mathcal{K})|+\numconnectedcomponets(\mathcal{K})-1}\\
        & \le (2i\lambda)^{|V(H)|-1}. \qedhere
\end{align*}
\end{proof}

\bigskip

We proceed to bounding the expression in \cref{eq:errorexpansiontriangleineq} in several steps. First, note that $H$ has $2^{|E(H)|}$ subgraphs. Thus, 
\begin{equation}
    \begin{split}
        & \sum_{K_1,K_2,\ldots, K_i\subseteq H}
        (1 - (1-\lambda)^i)^{|E(H)| - |E(\mathcal{K})|}
        \Big|
        \psi(K_1)\psi(K_2)\cdots \psi(K_i)\Big|\\
        & = 
        2^{|E(H)|i}\expect\Big[
        (1 - (1-\lambda)^i)^{|E(H)| - |E(\mathbf{K})|}
        |\psi(\bfK_1)\psi(\bfK_2)\cdots \psi(\bfK_i)|
        \Big],
    \end{split}
\end{equation}
where each $\bfK_j$ is sampled independently of the others by independently including each edge of $H$ with probability $1/2$ and $\mathbf{K} = \bfK_1\cup\mathbf{K}_2\ldots\cup\mathbf{K}_i.$ 


\paragraph{Case 1.1)} First, suppose that $i = 1.$ We will use the first bound in \cref{obs:productofdeviationsboundextrafactor}. We have to show that 
$$
2^{|E(H)|}(2\lambda)^{|V(H)|-1}\le \frac{1}{4d}\Big(\frac{(\log d)^C}{d}\Big)^{|V(H)|/2}
$$
for some absolute constant $C.$ 
Taking a logarithm on both sides, it is enough to show that 
$$
|E(H)| + (\log d- C\log \log d)(|V(H)|/2 +1)\le (\log d - C'''\log \log d)(|V(H)|-1)
$$
holds, where $C'''$ is the hidden constant in $\log (1/\lambda) = \log d - O(\log \log d).$ If $|E(H)| = o(\log \log d),$ choosing a large enough $C,$ we need to show that 
$$
(\log d- (C+2)\log \log d)(|V(H)|/2 +1)\le (\log d - C'''\log \log d)(|V(H)|-1).
$$
This clearly holds for large enough $C$ as 
$|V(H)|\ge 4,$ so $|V(H)|-1\ge |V(H)|/2 + 1.$

On the other hand, if $|E(H)| = \Omega(\log \log d),$ this means that $|V(H)| = \Omega((\log \log d)^{1/2}).$ Thus, for large enough $d,$ the inequality becomes equivalent to 
$$
|E(H)|\le \frac{1}{2}(\log d)|V(H)|(1-o_d(1)).
$$
This clearly holds since $\sqrt{|E(H)|}\le 2|V(H)|$ for any graph $H$ and 
$\sqrt{|E(H)|}< (\log d)^{5/8}.$
\paragraph{Case 1.2)} Now, suppose that
$\displaystyle
2< i \le \frac{3(\log d)|V(H)|}{10(|E(H)| + \log d)}.
$
In particular, this case is non-trivial if and only if 
$\displaystyle
2\le \frac{3(\log d)|V(H)|}{10(|E(H)| + \log d)},
$
which implies that
$|V(H)|\ge 6$ for large enough values of $d.$ Thus, we assume that $|V(H)|\ge 6.$ This, combined with $\displaystyle i \le \frac{3(\log d)|V(H)|}{10(|E(H)| + \log d)}$ implies
$(2i\lambda )^{|V(H)| - 1}2^{i|E(H)|}\le d^{- i - |V(H)|/2}$ for large enough values of $d$ (see \cref{prop:linftysignedboundcl1}).
One concludes from the second bound in \cref{obs:productofdeviationsboundextrafactor} that 
\begin{equation}
\begin{split}
&2^{|E(H)|i}\expect\Big[
        (1 - (1-\lambda)^i)^{|E(H)| - |E(\mathbf{K})|}
        |\psi(\bfK_1)\psi(\bfK_2)\cdots \psi(\bfK_i)|
        \Big]\\
&\le 2^{|E(H)|i}(2i\lambda)^{|V(H)|-1}\\
& \le 
d^{-i - |V(H)|/2}.
\end{split}
\end{equation}

\paragraph{Case 1.3)} $\frac{3(\log d)|V(H)|}{10(|E(H)| + \log d)}\le i \le \frac{11|V(H)|}{41}.$ In particular, such $i$ exist if and only if $|E(H)|\ge 13\log d/110 = \Omega(\log d).$ We assume that $|E(H)| = \Omega(\log d)$ in the rest of this case. We will use the following claim.

\begin{claim}
\label{claim:lowprobabilityofsmallsum}
Let $i \ge 2, 0 \le b \le i$ be integers and $a>0$ be a real number. Then, 
$$
\prob\Big[\sum_{j = 1}^i|V(\mathbf{K}_j)|\le ab\Big]\le 
\frac{1}{2^{i\times |E(H)|}}
\exp\Big(b(\log i) + a^2i\log |E(H)| + |E(H)|b\Big).
$$
\end{claim}
\begin{proof}
Note that 
\begin{equation}
    \begin{split}
        \prob[|V(\bfK_1)|\le a]\le 
        \prob\Big[|E(\bfK_1)|\le a^2\Big]\le 
        \frac{1}{2^{|E(H)|}}
        \sum_{j = 0}^{a^2} \binom{|E(H)|}{j}\le 
        \frac{|E(H)|^{a^2}}{2^{|E(H)|}}.
    \end{split}
\end{equation}
It follows that 
\begin{equation}
    \begin{split}
& \prob\Big[\sum_{j = 1}^i|V(\mathbf{K}_j)|\le ab\Big]\\
& \le\prob [
\exists j_1< j_2< \ldots< j_{i - b}
\;\text{s.t.}\;
|V(\bfK_{j_u})|\le a\;  \forall \; u \in [i-b]
] \\
&\le  \sum_{1\le j_1, j_2, \ldots, j_{i - b}\le i}
\prod_{u = 1}^{i - b}
\prob[|V(\bfK_{j_u})|\le a]\\
& \le 
\binom{i}{i - b}
\Bigg(\frac{|E(H)|^{a^2}}{2^{|E(H)|}}\Bigg)^{i - b}\\
&  = 
\frac{1}{2^{i|E(H)|}}
\binom{i}{b}
{|E(H)|^{a^2(i - b)}2^{|E(H)|b}}\\
& \le 
\frac{1}{2^{i|E(H)|}}
i^{b}
|E(H)|^{ia^2}e^{|E(H)|b}\\
& \le 
\frac{1}{2^{i|E(H)|}}
\exp\Big(
b(\log i) + 
ia^2\log |E(H)| + 
b|E(H)|
\Big).\qedhere
    \end{split}
\end{equation}
\end{proof}
We will apply the claim with the choices
 $$a = \frac{|V(H)|^{1/2}(\log d)^{1/2}}{i^{1/2}}(\log \log d)^{-1}\quad \text{and} \quad
b =\bigg\lfloor \frac{|V(H)|\log d}{|E(H)|}(\log \log d)^{-1}\bigg\rfloor.
$$
The condition $b\le i$ holds for large enough $d$ since $|E(H)| = \Omega(\log d)$ and $\frac{3(\log d)|V(H)|}{10(|E(H)| + \log d)}\le i.$ 

Now, we can write 
\begin{equation}
    \begin{split}
        &2^{|E(H)|i}\expect\Big[
        (1 - (1-\lambda)^i)^{|E(H)| - |E(\mathbf{K})|}
        |\psi(\bfK_1)\psi(\bfK_2)\cdots \psi(\bfK_i)|
        \Big]\\
        & \le 
        2^{|E(H)|i}\expect\Big[
        (1 - (1-\lambda)^i)^{|E(H)| - |E(\mathbf{K})|}
        |\psi(\bfK_1)\psi(\bfK_2)\cdots \psi(\bfK_i)| \; \Big|\sum_{j = 1}^i|V(\mathbf{K}_j)|\le ab
        \Big]\prob\Big[\sum_{j = 1}^i|V(\mathbf{K}_j)|\le ab\Big] + \\
        & +  
        2^{|E(H)|i}\expect\Big[
        (1 - (1-\lambda)^i)^{|E(H)| - |E(\mathbf{K})|}
        |\psi(\bfK_1)\psi(\bfK_2)\cdots \psi(\bfK_i)| \; \Big|\sum_{j = 1}^i|V(\mathbf{K}_j)|> ab
        \Big]\prob\Big[\sum_{j = 1}^i|V(\mathbf{K}_j)|>ab\Big]\\
        &\le 2^{|E(H)|i}
        \frac{1}{2^{i|E(H)|}}
        \exp\Big(
        b(\log i) + 
        ia^2\log |E(H)| + 
        b|E(H)|
        \Big)
        (2i\lambda)^{|V(H)| - 1}\\
        & + 
        2^{|E(H)|i}
        (2\lambda )^{i + ab/2},\\
    \end{split}
\end{equation}
where we used the second bound from \cref{obs:productofdeviationsboundextrafactor} in the case $\sum_{j = 1}^i|V(\mathbf{K}_j)|\le ab$ and the first bound in the case $\sum_{j = 1}^i|V(\mathbf{K}_j)|>ab.$ We now analyze the two terms separately.

\paragraph{Case 1.3.1)} We show that 
$$
\exp\Big(
        b(\log i) + 
        ia^2\log |E(H)| + 
        b|E(H)|
        \Big)
        (2i\lambda)^{|V(H)| - 1}
        \le 
        d^{-i - |V(H)|/2}.
$$
This is equivalent to 
$$
b (\log i) + ia^2\log |E(H)| + b|E(H)| + (i + |V(H)|/2)\log d \le 
(|V(H)|-1)(\log d - O(\log \log d)).
$$
We compare as follows:
\begin{enumerate}
    \item $b(\log i) + b|E(H)|\le 2b |E(H)| = 2|V(H)|\log d/(\log \log d)\le (|V(H)|-1)(\log d - O(\log \log d))/41$ for all large enough $d.$ We used the fact that $i \le |V(H)|\le |E(H)|$ and $|V(H)|\ge |E(H)|^{1/2} = \omega_d(1).$
    \item $ia^2\log |E(H)| = |V(H)|(\log d)(\log |E(H)|)/(\log \log d)^2\le (|V(H)|-1)(\log d - O(\log \log d))/41 $ for all large enough $d,$ where we used the fact that $|E(H)|\le (\log d)^{5/4},$ so $\log |E(H)| = O(\log \log d).$
    \item $(i + |V(H)|/2)\log d \le 33(|V(H)|-1)(\log d - O(\log \log d))/41$ for all large enough $d$ as $i \le 11V(H)/41$ and $|V(H)|\ge \sqrt{|E(H)|} = \Omega(\sqrt{\log d}) = \omega_d(1).$
\end{enumerate}
Altogether, this implies that
$$
b (\log i) + ia^2\log |E(H)| + b|E(H)| + (i + |V(H)|/2)\log d \le \frac{35}{41}
(|V(H)|-1)(\log d - O(\log \log d)),
$$
which is enough.

\paragraph{Case 1.3.2)} We show that 
$$
2^{|E(H)|i}
(2\lambda)^{i + ab/2}
\le 
d^{-i - |V(H)|/2}.
$$
Bounding $2^{|E(H)|i}\le e^{|E(H)|i}$ and 
using $ab = \frac{|V(H)|^{3/2}(\log d)^{3/2}}{i^{1/2}|E(H)|}(\log \log d)^{-2},$ the inequality becomes
\begin{equation*}
    \begin{split}
        &\log d(i + |V(H)|/2) + |E(H)|i\\
        & \le (\log d - O(\log \log d))
\frac{|V(H)|^{3/2}(\log d)^{3/2}}{i^{1/2}|E(H)|}(\log \log d)^{-2} + 
i(\log d - O(\log \log d))\Longleftrightarrow\\
&O(i\log\log d) + (\log d)|V(H)|/2 + |E(H)|i
         \le (\log d - O(\log \log d))
\frac{|V(H)|^{3/2}(\log d)^{3/2}|}{2i^{1/2}|E(H)|}(\log \log d)^{-2}
    \end{split}
\end{equation*}
We now handle the terms separately.
\begin{enumerate}
    \item 
    $O(i\log\log d)  \le \frac{1}{3}
    (\log d - O(\log \log d))
\frac{|V(H)|^{3/2}(\log d)^{3/2}}{2i^{1/2}|E(H)|}(\log \log d)^{-2}.
    $
    For large enough $d,$ it is enough to show that
    $$
    8i^{3/2}|E(H)|(\log \log d)^3\le (\log d)^{5/2}|V(H)|^{3/2}.
    $$
    This clearly holds as $|E(H)|\le (\log d)^{5/4}, i \le 11|V(H)|/41.$
    \item $(\log d)|V(H)|/2 \le \frac{1}{3}
    (\log d - O(\log \log d))
\frac{|V(H)|^{3/2}(\log d)^{3/2}}{2i^{1/2}|E(H)|}(\log \log d)^{-2}.$ Again, for large enough $d,$ it is enough to show that 
$$
4(\log d)|V(H)|i^{1/2}|E(H)|(\log \log d)^{2}\le |V(H)|^{3/2}(\log d)^{5/2}.
$$
Again, this holds as $i^{1/2}\le |V(H)|^{1/2}$ and 
$|E(H)|\le (\log d)^{5/4}.$
\item $|E(H)|i \le \frac{1}{3}
    (\log d - O(\log \log d))
\frac{|V(H)|^{3/2}(\log d)^{3/2}|}{2i^{1/2}|E(H)|}(\log \log d)^{-2}.$ For large enough $d,$ it is sufficient to show that 
$$
8|E(H)|^2i^{3/2}(\log \log d)^2\le 
(\log d)^{5/2}|V(H)|^{3/2}.
$$
Again, this holds as $|V(H)|\ge i,|E(H)|\le (\log d)^{5/4}/(\log \log d).$\footnote{This is the only place in the proof where we need $|E(H)|\ll (\log d)^{5/4}$ rather than $|E(H)|\ll (\log d)^2.$ Improving \cref{claim:lownumberofsmallsum}/\cref{claim:lowprobabilityofsmallsum} would, potentially, improve the result for polynomials of degree up to $(\log d)^{2-\epsilon}$ for any constant $\epsilon>0.$}
\end{enumerate}

\subsubsection{Proof of \texorpdfstring{\cref{prop:innersumbound}}{ithordererror} for Large Values of \texorpdfstring{$i.$}{i}} 
\label{sec:largeibound}
Suppose that $d\ge i\ge 11|V(H)|/41.$
The main idea behind proving \cref{prop:innersumbound} in that case is to bound each term $\error(A,\lambda)$ and then sum over the $2^{|E(H)|}$ subgraphs of $H.$
\begin{claim}
\label{claim:errorbound}
If $|E(A)|\le  (\log d)^{5/4}/(\log \log d),$
then 
$|\error(A,\lambda)| \le d^{-3+o_d(1)}.$
\end{claim}
\begin{proof}
We first prove the statement in the case when $|V(A)|\le |E(A)|.$ 
Note that $|\psi(K)|\le \lambda^3$ when $|V(K)|\le  3$ and 
$|\psi(K)|\le 2\times \lambda^{|V(K)|/2 + 1}$ otherwise
by  \cref{claim:1dtoruscounts}.
\begin{equation}
        \begin{split}
            & |\error(A,\lambda)| = \sum_{E(K)\subseteq E(A)}(-1)^{|E(K)|}\psi(K)
            \le \sum_{E(K)\subseteq E(A)}
            |\psi(K)|\\
            & =
            \sum_{K:\; |V(K)| \le 3} \lambda^3 + 
            2\lambda \sum_{(\log d)^{5/7} \ge |V(K)|> 3} \lambda^{|V(K)|/2} + 
            2\lambda\sum_{|V(A)|\ge |V(K)|> (\log d)^{5/7}} \lambda^{|V(K)|/2}
            \\
        \end{split}
\end{equation}
We analyse the three sums separately. The constant $5/7$ in $(\log d)^{5/7}$ is chosen arbitrarily in $(5/8,1).$

\paragraph{Case 1)} $|V(K)| \le 3.$ There are $O(|V(A)|^3) = O((\log d)^{15/4})$ subgraphs of $A$ on at most three vertices. Thus, 
$$
\sum_{K:\; |V(K)| \le 3} \lambda^3 
= O\Big((\log d)^{15/4}(\log d/d)^3\Big) = d^{-3 + o_d(1)}.
$$
 
\paragraph{Case 2) $(\log d)^{5/7}\ge |V(K)| > 3.$} When $|V(K)| = t,$ one can choose $V(K)$ in $\binom{|V(A)|}{t}$ ways and, once $V(K)$ is chosen, choose $E(A)$ in $2^{\binom{t}{2}}$ ways at most. This leads to  
\begin{equation}
    \begin{split}
        & 2\lambda\sum_{(\log d)^{5/7}\ge |V(K)|> 3} \lambda^{|V(A)|/2}\\
        &\le 2\lambda\sum_{(\log d)^{5/7}\ge t> 3} \binom{|V(A)|}{t}2^{\binom{t}{2}}\lambda^{t/2}\\
        &\le 
            2\lambda\sum_{(\log d)^{5/7}\ge t> 3}
            \Big(\frac{e^2|V(A)|^2 e^t\lambda}{t^2}\Big)^{t/2}
    \end{split}
\end{equation}
Each value
${e^2|V(A)|^2 2^t\lambda}/{t^2}$ is bounded by  
$$
e^2((\log d)^{5/4})^2e^{(\log d)^{5/7}}\lambda = O\Big((\log d)^{5/2}\times d^{o(1)}\times (\log d)\times d^{-1}) = d^{-1 + o_d(1)},
$$
where we used the fact that $5/7< 1.$
As each exponent $t/2$ is at least 2, the sum is bounded by 
$$
2\lambda(\log d)^{5/7} \times
(d^{-1 +o_d(1)})^{2} = 
\lambda d^{-2 + o_d(1)} = 
d^{-3 + o_d(1)}.
$$

\paragraph{Case 3) $|V(A)|\ge |V(K)| \ge (\log d)^{5/7}.$} Note that when $|V(K)| = t,$ one can choose $V(K)$ in $\binom{|V(A)|}{t}\le \binom{(\log d)^{5/4}}{t}$ ways and, once $V(K)$ is chosen, choose $E(A)$ in $\sum_{j = 0}^{(\log d)^{5/4}}\binom{\binom{t}{2}}{j}\le (\log d)^{5/4}\binom{\binom{t}{2}}{(\log d)^{5/4}}$ ways at most (as $A$ and, thus, $K$ has at most $(\log d)^{5/4}$ edges). This leads to  
\begin{equation}
    \begin{split}
        & \sum_{(\log d)^{5/4}\ge |V(K)|> (\log d)^{5/7}} \lambda^{|V(K)|/2}\\
        &\le 
        (\log d)^{5/4}\sum_{(\log d)^{5/4} \ge t > (\log d)^{5/7}}
        \binom{(\log d)^{5/4}}{t}\binom{\binom{t}{2}}{(\log d)^{5/4}}\lambda^{t/2}\\
        & \le
            (\log d)^{5/4}\sum_{(\log d)^{5/4}\ge t> (\log d)^{5/7}}
            \Big(\frac{e(\log d)^{5/4}}{t}\Big)^t
            \Big(\frac{et^2}{(\log d)^{5/4}}\Big)^{(\log d)^{5/4}}
            \lambda^{t/2}\\
        & = (\log d)^{5/4}
        e^{(\log d)^{5/4}}\sum_{(\log d)^{5/4}\ge t> (\log d)^{5/7}}
        \Big(\frac{t^2}{(\log d)^{5/4}}\Big)^{(\log d)^{5/4}-t}(\lambda e^2t^2)^{t/2}\\
        & \le 
        (\log d)^{5/4}
        e^{(\log d)^{5/4}}\sum_{(\log d)^{5/4}\ge t > (\log d)^{5/7}}
        ((\log d)^{5/4})^{(\log d)^{5/4}}(\lambda e^2t^2)^{t/2}.\\
    \end{split}
\end{equation}
Now, consider the expression $(\log d)^{5/4} e^{(\log d)^{5/4}}((\log d)^{5/4})^{(\log d)^{5/4}}(\lambda e^2t^2)^{t/2}.$ It can be rewritten as 
\begin{equation*}
    \begin{split}
        & \exp\Big(O(\log \log d) + ((\log d)^{5/4}+1) \log ((\log d)^{5/4})  + 
        t (\log t + 1) + 
        (t/2)\log \lambda
        \Big)\\
        & = 
        \exp\Big(O((\log d)^{5/4} \log \log d)   - 
        \Omega((\log d \log d^{5/7})
        \Big)\\
        & = 
        \exp \Big(- \Omega(\log d^{1 + 5/7})\Big),
    \end{split}
\end{equation*}
since $5/4< 1 + 5/7.$
Since each of the $O((\log d)^{5/4})$ summands is of order 
$\exp (- \Omega(\log d^{12/7})),$ the sum is clearly of order $\exp (- \Omega(\log d^{12/7})) = O(d^{-3}).$

Combining the two cases, we obtain that for graphs $A$
satisfying $|V(A)|\le |E(A)|,$
$\displaystyle
|\error(A,\lambda)|\le   d^{-3 + o_d(1)}
$ as desired.

Now, suppose that $|V(A)|>|E(A)|.$ If $A$ is acyclic, then we know that $|\error(A,\lambda)| = 0$ as $\psi(K) = 0$ for all $K\subseteq H$ as subgraphs are also acyclic, \cref{claim:1dtoruscounts}. If $A$ is not acyclic, then, it can be partitioned into two vertex-disjoint graphs $A = A_1\cup A_2,$ where $A_1$ satisfies $|V(A_1)|\le |E(A_1)|$ and $A_2$ is acyclic. As in the proof of \cref{claim:1dtoruscounts}, this implies that
\begin{align}
\nonumber
        & |\error(A,\lambda)| 
        = \expect_{\bfG\sim \RGG(n,\onetorus,\unif, \sigma^\infty_{1-\lambda}, 1-\lambda)}\Big[\unsignedweight_A(\bfG)\Big] - (1-\lambda)^{|E(A)|}\\
        \nonumber
        & = \expect_{\bfG\sim \RGG(n,\onetorus,\unif, \sigma^\infty_{1-\lambda}, 1-\lambda)}\Big[\unsignedweight_{A_1}(\bfG)\Big]
        \expect_{\bfG\sim \RGG(n,\onetorus,\unif, \sigma^\infty_{1-\lambda}, 1-\lambda)}\Big[\unsignedweight_{A_2}(\bfG)\Big]- 
        (1-\lambda)^{|E(A_1)|}(1-\lambda)^{|E(A_2)|}\\
        \nonumber
        & = 
        \Big(\expect_{\bfG\sim \RGG(n,\onetorus,\unif, \sigma^\infty_{1-\lambda}, 1-\lambda)}\Big[\unsignedweight_{A_2}(\bfG)\Big] - (1-\lambda)^{|E(A_2)|}\Big)\expect_{\bfG\sim \RGG(n,\onetorus,\unif, \sigma^\infty_{1-\lambda}, 1-\lambda)}\Big[\unsignedweight_{A_1}(\bfG)\Big]\\
        \nonumber
        & \quad\quad\quad\quad+ \Big(
        \expect_{\bfG\sim \RGG(n,\onetorus,\unif, \sigma^\infty_{1-\lambda}, 1-\lambda)}\Big[\unsignedweight_{A_1}(\bfG)\Big] - (1-\lambda)^{|E(A_1)|}\Big)(1-\lambda)^{|E(A_2)|}\\
        & = 
        \error(A_2,\lambda) \expect_{\bfG\sim \RGG(n,\onetorus,\unif, \sigma^\infty_{1-\lambda}, 1-\lambda)}\Big[\unsignedweight_{A_2}(\bfG)\Big]  + 
        \error(A_1, \lambda) (1-\lambda)^{|E(A_2)|}.
\end{align}
Since $A_2$ is acyclic, $\error(A_2,\lambda) = 0,$ so thee first term vanishes. Thus, 
\[
|\error(A,\lambda)| = 
|\error(A_1, \lambda) (1-\lambda)^{|E(A_2)|}|\le 
|\error(A_1, \lambda)| = 
d^{-3 + o_d(1)}.
\qedhere
\]
\end{proof}

Since the graph $H$ has at most $2^{|E(H)|}$ subgraphs, the LHS in Prop.~\ref{prop:innersumbound} can be bounded as
\begin{equation}
    \begin{split}
        & \Big|\sum_{A\subseteq H}
        (-1)^{|E(H)| - |E(A)|}
        (1-\lambda)^{i(|E(H)| - |E(A)|)}
        \error(A,\lambda)^i\Big|\\
        & \le 
         \sum_{A\subseteq H}\Big|
        (-1)^{|E(H)| - |E(A)|}
        (1-\lambda)^{i(|E(H)| - |E(A)|)}
        \error(A,\lambda)^i\Big|\\
        & \le 
        2^{|E(H)|} d^{-i (3 +o_d(1))}.\\
    \end{split}
\end{equation}
To prove \cref{prop:innersumbound}, it is enough to show that $\displaystyle e^{|E(H)|} d^{-i (3 +o_d(1))}\le \frac{1}{(4d)^i}\times\Big(\frac{(\log d)^{C}}{d}\Big)^{|V(H)|/2}.$ This would follow from 
$$
|E(H)|- 3(1-o_d(1))i\log d \le -i(\log d+2) - (\log d)|V(H)|/2
$$
or, equivalently,
$$
|E(H)| + 2i + (\log d)|V(H)|/2\le 2i\log d.
$$
We analyse each of the terms separately:
\begin{enumerate}
    \item $|E(H)|\le 2i(\log d)/(\log d)^{3/8}$ for large enough $i.$ Indeed, this follows since 
    $$
    |E(H)|\le \sqrt{|V(H)|^2}\times \sqrt{(\log d)^{5/4}/(\log \log d)} = 
    |V(H)|(\log d)^{5/8}\times (\log \log d)^{-1/2}\le 
    2i (\log d)/(\log d)^{3/8}.
    $$
    The last inequality holds for all large enough $d$ since $i \ge 11|V(H)|/41.$
    \item $i \le 2i(\log d)/(2\log d).$
    \item $(\log d)|V(H)|/2\le 2i(\log d)\times \frac{41}{44}$ since $i \ge 11|V(H)|/41.$
\end{enumerate}
Altogether, this gives 
\begin{equation*}
    |E(H)| + 2i + (\log d)|V(H)|/2\le 
2i\log d\Big(
\frac{1}{(\log d)^{3/8}} +
\frac{1}{2\log d} +
\frac{41}{44}
\Big)\le 
2i\log d
\end{equation*}
for large enough $d.$
\hfill \qed

\subsection{Performance of Low-Degree Polynomials in the \texorpdfstring{$L_\infty$}{Sup-Norm} Model}
\label{sec:linftyperformance}

We now finish the proofs of 
\cref{thm:linftylowdegreeindist,thm:linfinitydetection,prop:linftydimensionestmation}. The arguments are standard applications of \cref{thm:linftysignedcounts} and \cref{cor:linftysignedcycles}.

\subsubsection{Proof of \cref{thm:linfinitydetection}}
Observe that $K_n$ has $\Theta(n^3)$ subgraphs isomorphic to $C_3$ and 
$\Theta(n^4)$ subgraphs isomorphic to $C_4.$  
From \cref{cor:linftysignedcycles}, we conclude that
\begin{equation}
\label{eq:3and4signedlinfty}
    \begin{split}
        &\expect_{\bfG\sim \RGG(n,\dtorus,\unif,\sigma_p^\infty,p)}\Big[\signedcount_{C_3}(\bfG)\Big] = \tilde{\Theta}(n^3p^3/d^2),\\
        & \expect_{\bfG\sim \RGG(n,\dtorus,\unif,\sigma_p^\infty,p)}\Big[\signedcount_{C_4}(\bfG)\Big] = \tilde{\Theta}(n^4p^4/d^2).
    \end{split}
\end{equation}

Clearly, $\expect_{\bfK\sim \ergraph}\Big[\signedcount_{C_3}(\bfK)\Big] = \expect_{\bfK\sim \ergraph}\Big[\signedcount_{C_4}(\bfK)\Big] = 0.$
We now need to compute the respective variances as in \cref{def:signedcountsperformance}.

\paragraph{Triangles.}
With respect to both the $\ergraph$ and $\RGG(n,\dtorus,\unif,\sigma_p^\infty,p)$ distributions, one can expand the variance as follows (e.g. \cite{Liu2021APV}). Denote by $\triangle(i,j,k)$ the labelled triangle on vertices $i,j,k.$ Then, taking into account the different possible overlap patterns of two triangles,\footnote{Abusing notation, we write $\signedweight_{\triangle(1,2,3)}(\bfH)$ for the signed weight of the triangle on labelled vertices $1,2,3.$ Similarly, 
$\signedweight_{\square(1,2,3,4)}(\bfH)$ stands for the signed weight of a 4-cycle on labelled vertices $1,2,3,4.$
}
\begin{equation}
\label{eq:linfty3varexpansion}
\begin{split}
    &\Var[\signedcount_{C_3}(\bfH)] = 
    \Theta(n^3)\times \Var[\signedweight_{\triangle(1,2,3)}(\bfH)] + 
    \Theta(n^4)\times \Cov[\signedweight_{\triangle(1,2,3)}(\bfH), 
    \signedweight_{\triangle(1,2,4)}(\bfH)]\\
    &+ 
    \Theta(n^5) \times  \Cov[\signedweight_{\triangle(1,2,3)}(\bfH), 
    \signedweight_{\triangle(1,4,5)}(\bfH)] + 
    \Theta(n^6) \times  \Cov[\signedweight_{\triangle(1,2,3)}(\bfH), 
    \signedweight_{\triangle(4,5,6)}(\bfH)]].
\end{split}
\end{equation}
It turns out that the product of any two signed weights of subgraphs can be naturally decomposed as a (weighted) sum of 
signed weights of subgraphs. Thus, we can bound the above expression via \cref{thm:linftysignedcounts,cor:linftysignedcycles}. We take this approach in \cref{appendix:linftyperformanceomitted} to show the following.
\begin{equation}
\begin{split}
    \label{eq:threecyclevars}
    &\Var_{\bfK\sim \ergraph}[\signedcount_{C_3}(\bfK)] = \Theta(n^3p^3)\text{, and}\\
    &\Var_{\bfG\sim \RGG(n,\dtorus,\unif,\sigma_p^\infty,p)}[\signedcount_{C_3}(\bfG)] = 
    {\Theta}(n^3p^3)+ \tilde{O}(n^4p^5/d^2).
\end{split}
\end{equation}
This is enough to complete part 2 of \cref{thm:linfinitydetection}. According to \cref{def:signedcountsperformance}, one can distinguish between $\ergraph$ and $\RGG(n,\dtorus,\unif,\sigma_p^\infty,p)$ with high probability using the signed triangle test if and only if 
$$
\Big|\expect_{\bfG\sim \RGG}\signedcount_{C_3}(\bfG)\Big|= 
\omega\Big(
\sqrt{\Var_{\bfK\sim \ergraph}[\signedcount_{C_3}(\bfK)] + 
\Var_{\bfG\sim \RGG}[\signedcount_{C_3}(\bfG)]
}
\Big).
$$
Using \cref{eq:3and4signedlinfty,eq:threecyclevars}, this holds if and only if $d = \tilde{o}((np)^{3/4}).$

\paragraph{4-Cycles.} Similarly, in the case of 4-cycles, one obtains
\begin{equation}
\label{eq:linfty4varexpansion}
\begin{split}
    &\Var[\signedcount_{C_4}(\bfH)] = 
    \Theta(n^4)\times \Var[\signedweight_{\square(1,2,3,4)}(\bfH)]\\
    & + 
    \Theta(n^5)\times \Big( \Cov[\signedweight_{\square(1,2,3,4)}(\bfH), 
    \signedweight_{\square(1,2,3,5)}(\bfH)] + 
    \Cov[\signedweight_{\square(1,2,3,4)}(\bfH), 
    \signedweight_{\square(1,2,4,5)}(\bfH)]\Big)
    \\
    &+ 
    \Theta(n^6) \times \Big( \Cov[\signedweight_{\square(1,2,3,4)}(\bfH), 
    \signedweight_{\square(1,2,5,6)}(\bfH)] + 
    \Cov[\signedweight_{\square(1,2,3,4)}(\bfH), 
    \signedweight_{\square(1,5,2,6)}(\bfH)] \\
    & \quad \quad\quad +  
    \Cov[\signedweight_{\square(1,2,3,4)}(\bfH), 
    \signedweight_{\square(1,5,3,6)}(\bfH)]
    \Big)\\
    & + 
    \Theta(n^7)\times \Cov[\signedweight_{\square(1,2,3,4)}(\bfH), 
    \signedweight_{\square(1,5,6,7)}(\bfH)]\\
    & + \Theta(n^8)\times \Cov[\signedweight_{\square(1,2,3,4)}(\bfH), 
    \signedweight_{\square(6,6,7,8)}(\bfH)].
\end{split}
\end{equation}
Similarly, we show in \cref{appendix:linftyperformanceomitted}, that
\begin{equation}
        \label{eq:fourcyclevars}
    \begin{split}
    & \Var_{\bfK\sim \ergraph}[\signedcount_{C_4}(\bfK)] = \Theta(n^4p^4)\text{ and,}\\
    &\Var_{\bfG\sim \RGG(n,\dtorus,\unif,\sigma_p^\infty,p)}[\signedcount_{C_4}(\bfG)] = 
    {\Theta}(n^4p^4)+ \tilde{O}(n^5p^6/d^2 + n^6p^7/d^3).
    \end{split}
\end{equation}
Again,
$$
\Big|\expect_{\bfG\sim \RGG}\signedcount_{C_4}(\bfG)\Big|= 
\omega\Big(
\sqrt{\Var_{\bfK\sim \ergraph}[\signedcount_{C_4}(\bfK)] + 
\Var_{\bfG\sim \RGG}[\signedcount_{C_4}(\bfG)]
}
\Big)
$$
holds if and only if $d = \tilde{o}(np).$ \hfill \qed

\subsubsection{Proof of \cref{prop:linftydimensionestmation}}
\label{sec:applicationstodimensionestimation}
Low degree polynomial statistics are used in the literature not only for testing, but also for estimation (see, for example, \cite{Schramm_2022}). We illustrate with the concrete example of using signed cycles for estimating the dimension of $\RGG(n,\dtorus,\unif,\sigma_p^\infty,p)$ as in \cref{problem:dimensionestimation}. 

Suppose that $m$ is a small odd number. The expected signed count of $m$-cycles is\footnote{The factor $\frac{(m-1)!}{2}\binom{n}{m}$ is the number of undirected $m$-cycle subgraphs of $K_n.$} is
$$\frac{(m-1)!}{2}\binom{n}{m}\times \Big(p^m(d\lambda^m/(1-\lambda)^m) + {O}(p^md^{2}\lambda^{2m})\Big)$$ by \cref{cor:linftysignedcycles}. 
Therefore, one can estimate ${\lambda}$ from the number of signed $m$-cycles. Under a sufficiently strong concentration of the number of signed $m$-cycles, this could allow one to estimate $d$ as $\lambda\approx (\log 1/p)/d.$ Similarly, one can perform this for small even numbers. We define the success of a low-degree polynomial test for estimating a parameter (in our case, the dimension) in analogy to \cref{def:signedcountsperformance}.

\begin{definition}[Success of Polynomial Statistics for Exact Estimation] 
\label{def:estimationfrompoly}
Given is a family of random graph distributions $(\mathcal{D}_{\theta})_{\theta\in \mathcal{A}}$ over $n$ vertices indexed by a parameter $\theta$ taking values in $\mathcal{A}.$
Let $f(\cdot)$ be a polynomial in the edges of an $n$-vertex graph.
For each $\theta \in \mathcal{A},$ let $\mathcal{M}_\theta\coloneqq \expect_{\bfG\sim \mathcal{D}_\theta}[f(\bfG)].$  
We say that polynomial $f(\cdot)$ succeeds with high probability on exactly recovering $\theta$ if the following property holds. There exists some collection of values $\big\{\mathcal{V}_\theta\big\}_{\theta \in \mathcal{A}}$ such that the intervals
$\displaystyle
    \big\{
    \big[\mathcal{M}_\theta - \mathcal{V}_\theta,\mathcal{M}_\theta + \mathcal{V}_\theta\big], \theta\in \mathcal{A}
    \big\}
$
are disjoint and $\mathcal{V}_\theta =\omega(\Var_{\bfG\sim\distribution_\theta}[f(\bfG)]^{1/2})$ for each $\theta.$
If, on the other hand, no such intervals exist, we say that the 
polynomial $f$ fails in the task of exact estimation.
\end{definition}
The interpretation of this definition is simple. Suppose that the true parameter is $\theta'.$ Then, by Chebyshev's, inequality
with high probability over $\bfG\sim \mathcal{D}_{\theta'},$ it is the case that $f(\bfG)\in  \big[\mathcal{M}_{\theta'} - \mathcal{V}_{\theta'},\mathcal{M}_{\theta'} + \mathcal{V}_{\theta'}\big].$ If the intervals are disjoint, this is the unique interval of the form $\big[\mathcal{M}_\theta - \mathcal{V}_\theta,\mathcal{M}_\theta + \mathcal{V}_\theta\big]$ with this property and, thus, one can find $\theta'.$ It must be noted that this is the implicit definition used in \cite{Bubeck14RGG,friedrich2023dimension} for estimating the dimension of random geometric graph models.

To apply this definition to \cref{problem:dimensionestimation}, we use the variance bounds \cref{eq:threecyclevars,eq:fourcyclevars} and the following simple estimate of $\lambda,$ deferred to \cref{appendix:linftyperformanceomitted}.

\begin{proposition}
\label{prop:exactlambda}
Suppose that $d = \omega(\log 1/p).$ Then, 
\begin{align*}
&\lambda^\infty_p = 
1 - p^{1/d} = \frac{\log 1/p}{d} - \frac{1}{2}\bigg(\frac{\log 1/p}{d}\bigg)^2 + O\Bigg(\bigg(\frac{\log 1/p}{d}\bigg)^3\Bigg),\\
& \lambda^\infty_p/(1 - \lambda^\infty_p) = 
p^{-1/d} - 1 = 
\frac{\log 1/p}{d} + \frac{1}{2}\bigg(\frac{\log 1/p}{d}\bigg)^2 + O\Bigg(\bigg(\frac{\log 1/p}{d}\bigg)^3\Bigg).
\end{align*}
\end{proposition}

We are now ready to evaluate the intervals in which the signed triangle and 4-cycle statistics succeed with high probability in the exact dimension recovery tasks. 

\paragraph{Triangles.} Using \cref{cor:linftysignedcycles} and \cref{prop:exactlambda}, the expected signed count of three cycles in dimension $d$ is 
\begin{align*}
& \mathcal{M}^{C_3}_d \coloneqq  \expect_{\bfG\sim \RGG(n,\dtorus, \unif, \sigma^\infty_p,p)}\big[\signedcount_{C_3}(\bfG)\big] = 
\binom{n}{3}p^3\times \Bigg(d\bigg(\frac{\lambda}{1-\lambda}\bigg)^3 + 
O(d^2\lambda^{6})
\Bigg)\\
& = \binom{n}{3}p^3\times \Bigg(
\frac{(\log 1/p)^3}{d^2} + 
\frac{3}{2} \times \frac{(\log 1/p)^4}{d^3} + 
O\bigg(\frac{\log(1/p)^5}{d^4}\bigg)
\Bigg).
\end{align*}
In particular, this means that 
\begin{equation}
\begin{split}
\label{eq:meandifference3cycle}
&\mathcal{M}^{C_3}_d - \mathcal{M}^{C_3}_{d+1} = 
\expect_{\bfG\sim \RGG(n,\dtorus, \unif, \sigma^\infty_p,p)}\big[\signedcount_{C_3}(\bfG)\big] - 
\expect_{\bfG\sim \RGG(n,\mathbb{T}^{d+1}, \unif, \sigma^\infty_p,p)}\big[\signedcount_{C_3}(\bfG)\big]\\
& = 
\binom{n^3}{p^3}\Bigg(\Theta\bigg(
\frac{(\log(1/p)^3)}{d^3}\bigg) + 
O\bigg(
\frac{(\log(1/p)^4)}{d^4}\bigg)
\Bigg)\\
& = 
\Theta(n^3p^3(\log 1/p)^3/d^3).
\end{split}
\end{equation}
In particular, $\mathcal{M}^{C_3}_{d+1}\le \mathcal{M}^{C_3}_d$ when $d = \omega(\log 1/p).$ Therefore, numbers $\mathcal{V}_d$ with the desired property from \cref{def:estimationfrompoly} exist if and only if for all $d\in [\omega(\log 1/p), M],$
$$
    \mathcal{M}^{C_3}_d - \mathcal{M}^{C_3}_{d+1} = \omega\bigg(\sqrt{
\Var_{\bfG\sim \RGG(n,\dtorus, \unif, \sigma^\infty_p,p)}\big[\signedcount_{C_3}(\bfG)\big] + 
\Var_{\bfG\sim \RGG(n,\mathbf{T}^{d+1}, \unif, \sigma^\infty_p,p)}\big[\signedcount_{C_3}(\bfG)\big]
    }\bigg).
$$
Using \cref{eq:threecyclevars} and \cref{eq:meandifference3cycle}, this is equivalent to 
$$
n^3p^3(\log 1/p)^3/d^3 = 
\tilde{\omega}\Big(
\sqrt{
n^3p^3 + n^4p^5/d^2}
\Big).
$$
One can easily check that this is satisfied if and only if 
$d = \tilde{o}((np)^{1/2}).$

\paragraph{4-Cycles.} In the exact same way we conclude from \cref{cor:linftysignedcycles} and \cref{prop:exactlambda}\footnote{Also, from \cref{claim:1dtoruscounts} we recall $\phi(3) = 2/3,$ even though the exact value of  $\phi(3)$ is irrelevant as long as it is non-zero.}
\begin{equation*}
    \begin{split}
        & \mathcal{M}_d^{C_4} = 
        \expect_{\bfG\sim \RGG(n,\dtorus, \unif, \sigma^\infty_p,p)}\big[\signedcount_{C_4}(\bfG)\big]\\
        & = 
        3\binom{n}{4}p^4\Bigg(
        \phi(3)\frac{(\log 1/p)^3}{d^2} 
        +\frac{3}{2}\phi(3)\frac{(\log 1/p)^4}{d^3} - 
        \frac{(\log 1/p)^4}{d^3} + 
         O\bigg(\frac{(\log 1/p)^5}{d^4}\bigg)
        \Bigg)\\
        & = 
        3\binom{n}{4}p^4\Bigg(
        \frac{2}{3}\frac{(\log 1/p)^3}{d^2} 
        + 
         O\bigg(\frac{(\log 1/p)^5}{d^4}\bigg)
        \Bigg).
    \end{split}
\end{equation*}
Thus, $0\le \mathcal{M}_d^{C_4} - \mathcal{M}_{d+1}^{C_4} = \Theta(n^4p^4(\log 1/p)^3/d^3).$ Finally, by \cref{eq:fourcyclevars}, the condition 
$$
    \mathcal{M}^{C_4}_d - \mathcal{M}^{C_4}_{d+1} = \omega\bigg(\sqrt{
\Var_{\bfG\sim \RGG(n,\dtorus, \unif, \sigma^\infty_p,p)}\big[\signedcount_{C_4}(\bfG)\big] + 
\Var_{\bfG\sim \RGG(n,\mathbf{T}^{d+1}, \unif, \sigma^\infty_p,p)}\big[\signedcount_{C_4}(\bfG)\big]
    }\bigg)
$$ is equivalent to 
$$
n^4p^4(\log 1/p)^3/d^3 = 
\tilde{\omega}\Big(
\sqrt{
n^4p^4 + n^5p^6/d^2 + 
n^6p^7/d^3}
\Big).
$$
One can easily check that this is satisfied if and only if 
$d = \tilde{o}((np)^{2/3}).$ \hfill \qed

\subsubsection{Proof of \cref{thm:linftylowdegreeindist}}
\label{sec:advantagebound}
The proof follows a standard procedure for bounding $\advantage_{\le D}^2 ,$ e.g in \cite{hopkins18}.

Suppose that $n^{-1+\epsilon}\le p \le 1/2$ and 
$d\ge np$ for some absolute constant $\epsilon.$
In particular, this means that $d\ge n^\epsilon$ and $d\ge p^{-\delta}$ for some absolute constant $\delta>0.$ 

Let $D = (\log d)^{5/4}/(\log \log d) = \Theta(\log n/\log \log n).$ Consider the orthonormal basis of $\ergraph$ given by the polynomials
$\sfp_H(\cdot ) \coloneqq  \signedweight_H(\cdot)/(p(1-p))^{|E(H)|/2}$ for all 
subgraphs $H$ of $K_n.$ 
From \cref{sec:prelimcompindist}, we know that
to show statistical indistinguishability with respect to degree $D$ polynomials, we simply need to prove the inequality
$$
\advantage^2_{\le D} - 1\coloneqq 
\sum_{H \; : \; 1\le |E(H)|\le D}\expect_{\bfG\sim \RGG(n,\dtorus, \unif, \sigma^\infty_p,p)}[\sfp_H(\bfG)] ^2 = o(1).
$$
We prove this as follows. First, note that if $H$ has a vertex of degree $1,$ then $\expect_{\bfG\sim \RGG}[\sfp_H(\bfG)] = 0$ as in \cref{rmk:onragfactorization}. Thus, we can assume that there is no such vertex and, so, $3\le |V(H)|\le |E(H)|.$ Using \cref{thm:linftysignedcounts}, we have the following inequality.

\begin{equation}
    \begin{split}
        &\sum_{H \; : \; 3\le |E(H)|\le D}\expect_{\bfG\sim \RGG}[\sfp_H(\bfG)] ^2\\
        & =
        \sum_{H \; : \; 3\le |E(H)|\le D}\frac{1}{(p(1-p))^{|E(H)|}}\expect[\signedweight_H(\bfG)]^2\\
        & \le 
        \sum_{H \; : \; 3\le |E(H)|\le D}\frac{1}{(p(1-p))^{|E(H)|}} p^{2|E(H)|}((\log d)^C/d)^{|V(H)|}\\
        & \le 
        \sum_{H \; : \; 3\le |E(H)|\le D}
        (2p)^{|E(H)|}((\log d)^C/d)^{|V(H)|}\\
        & = 
        \sum_{H \; : \; 3\le |E(H)|, \; |V(H)|\le D^{2/3}}
        (2p)^{|V(H)|}((\log d)^C/d)^{|V(H)|}\\
        &\quad\quad\quad + 
        \sum_{H \; : \; 3\le |E(H)|<D, \; |V(H)|>D^{2/3}}
        (2p)^{|V(H)|}((\log d)^C/d)^{|V(H)|}.
    \end{split}
\end{equation}
We used the fact that $1-p\ge 1/2$ and  $|V(H)|\le |E(H)|.$
Now, we consider the two sums separately.

\paragraph{Case 1) $|V(H)|\le D^{2/3}.$} When $|V(H)| = t,$ there are $\binom{n}{t}$ ways to choose 
$V(H)$ and then, once $V(H)$ is chosen, at most 
$2^{\binom{t}{2}}$ ways to choose 
$E(H).$ This gives
\begin{align*}
    &\sum_{H \; : \; 0<|E(H)|, \; |V(H)|\le D^{2/3}}
        (2p)^{|V(H)|}((\log d)^C/d)^{|V(H)|}\\
    &\le \sum_{t = 3}^{D^{2/3}}
    \binom{n}{t}(2p)^t2^{\binom{t}{2}}((\log d)^C/d)^{t}\\
    & \le \sum_{t = 3}^{D^{2/3}}\Bigg(\frac{np2^t(\log d)^C}{d}\Bigg)^t\\
    & \le 
    \sum_{t = 3}^{D^{2/3}}\Bigg(\frac{np2^{(\log d)^{5/6}}(\log d)^C}{d}\Bigg)^t.
\end{align*}
Clearly, if $d= \max\Big((np)^{1 + o_n(1)}), \omega_n(1)\Big),$ one has 
$$
\Bigg(\frac{np2^{(\log d)^{5/6}}(\log d)^C}{d}\Bigg) = o(1).
$$
Thus, there is exponential decay in the sum and it is of order $o(1).$
\paragraph{Case 2) $|V(H)|\ge D^{2/3}.$}
 When $|V(H)| = t,$ there are $\binom{n}{t}$ ways to choose 
$V(H)$ and then, once $V(H)$ is chosen, at most 
$$\sum_{j = 0}^D\binom{\binom{t}{2}}{j}\le 2\binom{\binom{t}{2}}{|D|}$$ ways to choose 
$E(H).$ This gives
\begin{align*}
    &\sum_{H \; : \; |E(H)|\le D, \; |V(H)|>D^{2/3}}
        (2p)^{|V(H)|}((\log d)^C/d)^{|V(H)|}\\
    &\le 
    \sum_{t = D^{2/3}}^D
    (2p)^t\binom{n}{t}2\binom{\binom{t}{2}}{D}((\log d)^C/d)^{t}\\
    & \le 
    2\sum_{t = D^{2/3}}^D
    (2p)^t\Big(\frac{ne}{t}\Big)^t
    \Big(
    \frac{et^2}{D}
    \Big)^D
    \Big(
    \frac{(\log d)^C}{d}
    \Big)^t\\
    & \le 
    2\sum_{t= D^{2/3}}^D
    (2p)^t
    n^t ((t^2)^{D/t})^t
    \Big(
    \frac{(\log d)^C}{d}
    \Big)^t\\
    & \le 
    2\sum_{t= D^{2/3}}^D
    \Big(
    \frac{2pnt^{2D/t}(\log d)^C}{d}
    \Big)^t\\
    & \le 
    2\sum_{t= D^{2/3}}^D
    \Big(
    \frac{2pnD^{2D^{1/3}}(\log d)^C}{d}
    \Big)^t\\
    & \le 
    2\sum_{t= D^{2/3}}^D
    \Big(
    \frac{2pne^{O\big((\log \log d)(\log d)^{5/12}\big)}(\log d)^C}{d}
    \Big)^t.
\end{align*}
Again, under the same conditions $d= \max\big\{(np)^{1 + o_n(1)}), \omega_n(1)\big\},$ the expression is of order $o(1).$  \hfill \qed

\section{Statistical Indistinguishability in the 
\texorpdfstring{$L_\infty$}{Sup-Norm} Model}
\label{sec:statslinfty}
In this section, we prove \cref{thm:linftyinformationtheory}. Recall condition \eqref{eq:assumption}. Suppose further that $d = \omega(n(\log n)^2).$ As in \cref{sec:statslinfty}, $\tau^\infty_p = 1 -\lambda_p^\infty,$ where $\lambda_p^\infty = \frac{\log (1/p)}{d}(1 + o(1)).$ We will write $\sigma,\lambda,\tau$ instead of $\sigma_p^\infty,\lambda_p^\infty,\tau_p^\infty $ for brevity. We can view $\sigma(\bfx, \bfy)$ as a single argument function of $\bfx - \bfy.$ 

Expanding one of the terms in \cref{eq:Kltensorization}, we obtain
\begin{equation}
    \begin{split}
&
\expect\Big[\Big(1+ 
\frac{\gamma(\bfg)}{p(1-p)}
\Big)^k\Big] =
\expect\Big[\Big(\frac{1-2p}{1-p} + 
\frac{\sigma*\sigma(\bfg)}{p(1-p)}
\Big)^k\Big] = 
\sum_{t = 0}^k
\binom{k}{t}
\Bigg(\frac{1-2p}{1-p}\Bigg)^{k-t}
\frac{\expect[(\sigma*\sigma)^{t}]}{p^t(1-p)^t}\,.
    \end{split}
\end{equation}

We will prove the following bound on the moments of $\sigma*\sigma.$

\begin{claim}
\label{claim:explicitconvolutionmoments}
For all $t\ge 1,t = o(1/\lambda) = o(d/(\log d)^2),$ it holds that
    $\expect[(\sigma*\sigma)^t] = p^{2t}(1 + \Theta(d\lambda^3t^2)).$ Also, $\expect[\sigma*\sigma] = p^{2}.$ 
\end{claim}

We will first show how \cref{claim:explicitconvolutionmoments} yields \cref{thm:linftyinformationtheory} before proving \cref{claim:explicitconvolutionmoments}.

\begin{equation}
    \begin{split}
&\expect\Big[\Big(\frac{1-2p}{1-p} + 
\frac{\sigma*\sigma(\bfg)}{p(1-p)}
\Big)^k\Big]\\
& = 
\sum_{t = 0}^k
\binom{k}{t}\Big(\frac{1-2p}{1-p}\Big)^{k-t}
\frac{p^{2t}(1 + \Theta(d\lambda^3t^2))}{p^t(1-p)^t}\\
& = 
\sum_{t = 0}^k
\binom{k}{t}\Big(\frac{1-2p}{1-p}\Big)^{k-t}\Big(\frac{p}{1-p}\Big)^t + 
d\lambda^3
\Theta\Bigg(
\sum_{t = 0}^k
\binom{k}{t}\Big(\frac{1-2p}{1-p}\Big)^{k-t}\Big(\frac{p}{1-p}\Big)^tt^2
\Bigg)\\
& = 
1  + d\lambda^3 \Theta\Bigg(\frac{k(1-2p)^{k-1}p}{(1-p)^k} + 
k(k-1)\frac{p^2}{(1-p)^2}\sum_{t = 2}^k \binom{k-2}{t-2}
\Big(\frac{1-2p}{1-p}\Big)^{k-t}\Big(\frac{p}{1-p}\Big)^{t-2}
\Bigg)\\
& =  1 + d\lambda^3 \Theta(kp + k^2p^2) = 1 + \tilde{\Theta}(d^{-2}kp + d^{-2}k^2p^2).
    \end{split}
\end{equation}
Going back to \cref{eq:Kltensorization},
\begin{align*}
    & \sum_{k = 0}^{n-1}\log  \expect\bigg[\Big(1 + \frac{\gamma(\bfx)}{p(1-p)}\Big)^k\bigg] = 
    \sum_{k = 0}^{n-1}
    \log  \expect\bigg[1 + \tilde{\Theta}(d^{-2}kp + d^{-2}k^2p^2)\bigg]\\
    & \le 
    \tilde{\Theta}\Big(
    d^{-2}p\sum_{k = 0}^{n-1}k + 
    d^{-2}p^2\sum_{k = 0}^{n-1}k
    \Big) = 
    \tilde{\Theta}\Big(
    d^{-2}pn^2 + 
    d^{-2}p^2n^3
    \Big),
\end{align*}
where we used the fact that $t\le k \le n = o(d/\log d).$
The last expression is of order $o(1)$
whenever $p \ge 1/n, d \ge n^{3/2}p$ with which the poof follows.
\begin{proof}[Proof of \cref{claim:explicitconvolutionmoments}] First, note that 
\begin{equation}
    \begin{split}
\sigma*\sigma(\bfx) 
& = 
\expect_\bfg[\sigma(\bfg)\sigma(\bfx-\bfg)] = 
\expect\Bigg[
\prod_{i = 1}^d
\indicator[|g_i|_C\le 1- \lambda]
\indicator[|x_i - g_i|_C\le 1 - \lambda]
\Bigg]\\
& = 
\prod_{i = 1}^d 
\expect\Big[
\indicator[|g_i|_C\le 1- \lambda]
\indicator[|x_i - g_i|_C\le 1 - \lambda]
\Big].
\end{split}
\end{equation}
Now, as is easy to see from \cref{Fig:xfarfromorigindiag,Fig:xclosetoorigindiag},
\begin{equation}
        f(x)\coloneqq \prob_{g\sim \unif(\mathbb{T}^1)}[|g|_C\le 1-\lambda, |x - g|_C\le 1 - \lambda] = 
    \begin{cases}
        1 - 2\lambda\quad \text{when} |x|_C\ge 2\lambda,\\[2pt]
        1 - \lambda - |x|_C/2\quad \text{when}  |x|_C\le 2\lambda.
    \end{cases}
\end{equation}

\begin{figure}[!htb]
   \begin{minipage}{0.47\textwidth}
     \centering
     \includegraphics[width=0.4\linewidth]{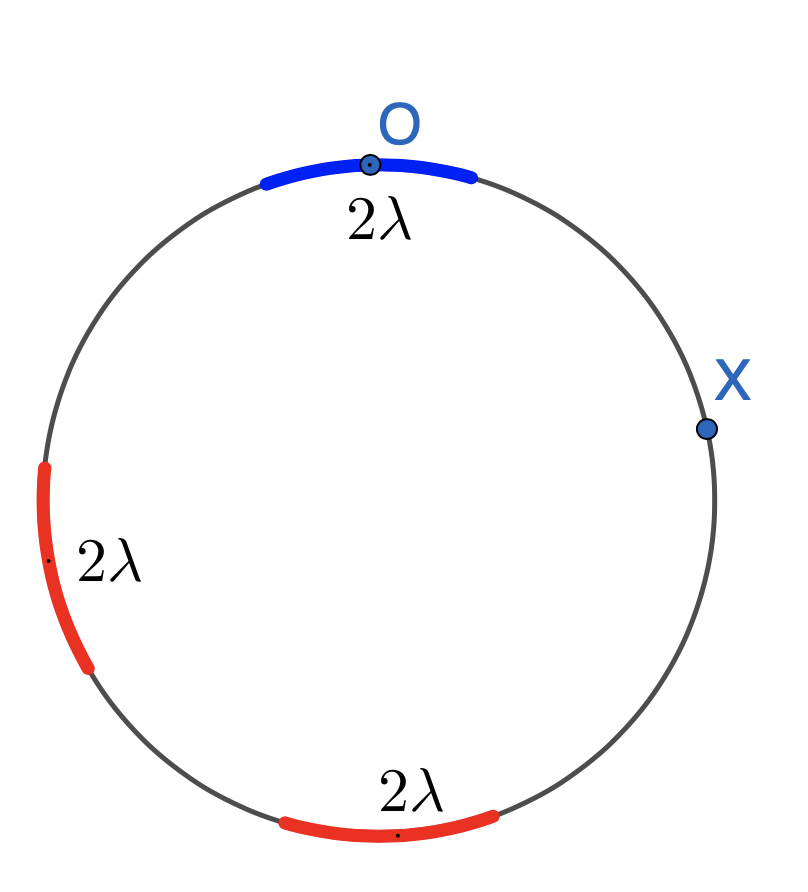}
     \caption{\footnotesize{In the case when $x$ is far from the origin (i.e., not within distance $2\lambda$), the antipodal arcs of length $2\lambda$ of $0$ and $x$ (colored in red) do not intersect. Thus, $g$ should be anywhere outside of the two segments of total length $4\lambda$ to satisfy 
     $|g|_C\le 1-\lambda, |x - g|_C\le 1 - \lambda.$
     }}\label{Fig:xfarfromorigindiag}
   \end{minipage}\hfill
   \begin{minipage}{0.47\textwidth}
     \centering
     \includegraphics[width=0.4\linewidth]{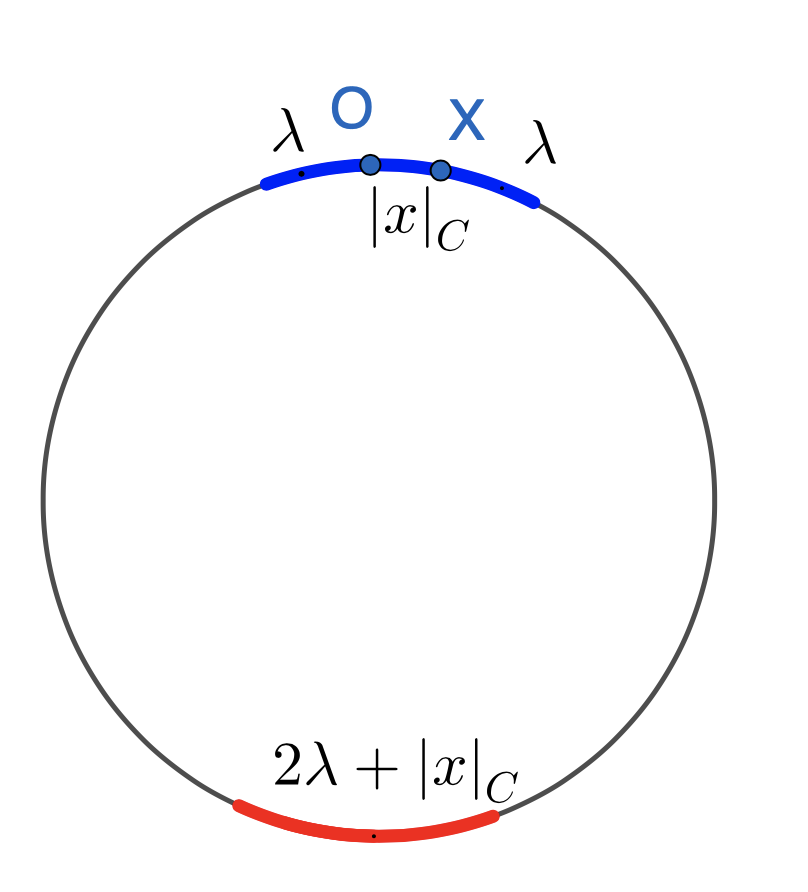}\caption{\footnotesize{In the case when $x$ is close to the origin (i.e., within distance $2\lambda$), the antipodal arcs of length $2\lambda$ of $0$ and $x$ (colored in red) intersect. Thus, $g$ should be anywhere outside of the intersection of the two segments of total length $2\lambda + |x|_C$ to satisfy 
     $|g|_C\le 1-\lambda, |x - g|_C\le 1 - \lambda.$}}\label{Fig:xclosetoorigindiag}
   \end{minipage}\hfill
\end{figure}

Thus, 
$\displaystyle
\sigma*\sigma(\bfx) = 
\prod_{i = 1}^d f(x_i).
$ 
It follows that \begin{equation}
    \begin{split}
    & \expect[(\sigma*\sigma(\bfx))^t] = 
    \expect\Big[(\prod_{i=1}^df(x_i))^t\Big] = 
    \expect[f(x_i)^t]^d\\
    & = \Bigg(
\int_0^1f(x)^t
    \Bigg)^d =\Big( (1-2\lambda)\times (1 - 2\lambda)^t + 
    \int_{0}^{2\lambda}
    (1-\lambda-s/2)^tds
    \Big)^d \\
    & 
    = \Big( (1 - 2\lambda)^{t+1} + 
    2\frac{(1-\lambda)^{t+1} - (1-2\lambda)^{t+1}}{(t+1)}
    \Big)^d \\
    & = \Big(
    \frac{2}{t+1}(1-\lambda)^{t+1} + 
    \frac{t-1}{t+1}(1-2\lambda)^{t+1} 
    \Big)^d.
    \end{split}
\end{equation}
A simple calculation, deferred to \cref{appendix:linftystatscalculation}, shows that the last expression is $p^{2t}(1 + \Theta(d\lambda^3t^2))$. \end{proof}

\section{Statistical Indistinguishability in the \texorpdfstring{$L_q$}{Lq} Model}
\label{sec:lqproof}
Here, we prove \cref{thm:torussmallq}. We will give in full detail the proof in the case $q = o(d/\log d)$ and explain the necessary changes in the (much simpler) case $q = \Omega(d/\log d).$ The latter is technically much simpler and does not use any ideas which do not appear in the case $q = o(d/\log d).$ 
\subsection{The Proof for Small \texorpdfstring{$q$}{q}}
\paragraph{Further Notation.}
Throughout, we fix $q\ge 1, q  = o(d\log^{-1}d)$ and consider $\RGG(n, \dtorus,\unif,\sigma^q_{1/2},1/2).$ For simplicity of notation, we denote $\tau^q_{1/2}$ simply by $\tau$ and $\sigma^q_{1/2}$ by $\sigma.$ Note that $\sigma,$ when viewed as a single argument function, can be equivalently defined as the indicator of 
$B_{L_q(\dtorus),\tau}(\bfnull),$ where $B_{L_q(\dtorus),\tau}(\bfx)$ is the $L_q$ ball of radius $\tau$ on $\dtorus$ centered at $\bfx.$ Under this notation,
\begin{equation}
\gamma(\bfg) = 
\expect_{\bfz}\Bigg[\bigg(\sigma(\bfg-\bfz)-\frac{1}{2}\bigg)\bigg(\sigma(\bfz)-\frac{1}{2}\bigg)\Bigg]
=\expect_\bfz[\sigma(\bfg - \bfz)\sigma(\bfz)]
- \frac{1}{4} = 
\big|B_{L_q(\dtorus),\tau}(\bfg)\cap 
     B_{L_q(\dtorus),\tau}(\bfnull)\big| - \frac{1}{4}.
\end{equation}

\paragraph{Proof Strategy.}
Our main goal will be to prove that 
\begin{equation}
\label{eq:momentnorms}
\expect_{\bfg\sim\unif(\dtorus)}\big[|\gamma(\bfg)|^k\big]^{1/k}\le C\frac{k}{\sqrt{dq}}.
\end{equation}
for an absolute constant $C.$
This is sufficient to conclude \cref{thm:torussmallq} for the following reason. Using \cref{eq:Kltensorization} and the fact $\expect_{\bfg\sim\unif(\dtorus)}[\gamma(\bfg)] = 0,$ 

\begin{equation}
\label{eq:smallqboundingKLdivergence}
\begin{split}\KL\Big(
\RGG(n, \dtorus,\unif,\sigma^q_{1/2},1/2)\|
\ergraphhalf
\Big) 
&\le \sum_{k = 0}^{n-1}
\log \Big(\expect_{\bfg\sim \unif{(\dtorus)}}
[(1 + 4\gamma(\bfg))^k]\Big)\\
&= \sum_{k = 0}^{n-1}
\log \Big(1 +\sum_{t\ge 2}\binom{k}{t}4^t\expect[\gamma(\bfg)^t]
\Big)\\
& \le 
\sum_{k = 0}^{n-1}
\sum_{t = 2}^k
\binom{k}{t}4^t\expect[\gamma(\bfg)^t]\\
& \le 
\sum_{k = 0}^{n-1}
\sum_{t = 2}^k
\binom{k}{t}4^t\expect[|\gamma(\bfg)|^t]\\
&\le 
n\sum_{k = 0}^{n}
\binom{n}{k}4^k\expect[|\gamma(\bfg)|^k]\\
& \le  
n
\sum_{k = 2}^n
\left(\frac{ne}{k}\right)^k4^k C^k \frac{k^k}{(dq)^{k/2}}\\
&\le  
n 
\sum_{k = 2}^n
\left(\frac{4eCn}{\sqrt{dq}}\right)^k\\
& =  
n \times  O\bigg(\frac{n^2}{dq}\bigg) = 
O\bigg(\frac{n^3}{dq}\bigg) = o(1).
\end{split}
\end{equation}
We used the fact that $dq = \omega(n^3)$ to conclude that there is exponential decay in $\displaystyle \sum_{k = 2}^n
\left(\frac{4eCn}{\sqrt{dq}}\right)^k.$

In light of \cref{claim:bernsteinmcdiarmidrag}, to prove 
\cref{eq:momentnorms}, it is enough to show the following two statements:
\begin{enumerate}
    \item \emph{Small Marginal Increments:} $\|D_i\gamma\|_\infty = O(\frac{1}{\sqrt{dq}})$ for all $i.$
    \item \emph{Small Marginal Variances:} $\|\Var_i[\gamma]\|_\infty = O(\frac{1}{d^2q})$ for all $i.$
\end{enumerate}
Due to symmetry, it is enough to prove the statements for $d = i.$
In deriving those two quantities, we will need the following anticoncentration result.

\paragraph{Anticoncentration of random $L_q$-distances.}

\begin{claim}
\label{claim:lqanticoncentration}
Suppose that 
$U_1, U_2, \ldots, U_{d-1}$ are iid $\unif([0,1])$ random variables and $q = o(d/\log d).$ Let $F$ be the CDF of $\sum_{i = 1}^{d-1}U_i^q.$ Then, for $\psi(\ell)\coloneqq F(\tau^q) - F(\tau^q - \ell^q) = F([\tau^q - \ell^q, \tau^q]),$ we have 
\begin{equation}
\int_0^1
\psi(\ell)d\ell = 
O\Big(
\frac{1}{\sqrt{dq}}
\Big)\quad \text{ and }\quad
 \int_0^1
\psi(\ell)^2d\ell = O\Big(\frac{1}{d}\Big).
\end{equation}
\end{claim}
The proof is delayed to 
\cref{appendix:anticoncentration}. We derive it applying (an approximate version of) a bound on the sup-norm of the density of a convolution of several random variables with given sup-norms of their densities due to Bobkov and Chistyakov \cite{Bobkov14}. This nearly captures the setting of the above claim since the result would follow from a small enough density of $\sum_{i = 1}^{d-1}U_i^q.$ We cannot directly apply the result from \cite{Bobkov14}, however, as the density of $U_i^q$ diverges around $0.$ A simple argument based on representing $U_i^q$ as the mixture of two random variables, one of which having uniformly small density, suffices. We derive \cref{claim:lqanticoncentration} via the following bound, which can be viewed as a strengthening of the Berry-Esseen theorem for arbitrarily small intervals. \

\begin{proposition}
\label{prop:inervalprobsmallq}
For any interval $[a,b],$ 
$\displaystyle
\prob\bigg[\sum_{i = 1}^{d-1}U_i^q\in [a,b]\bigg]\le
\exp(- \Omega(d/q)) + 
(b-a)\times \sqrt{q/d}.
$  
\end{proposition}
\cref{cor:smallball} generalizes this result to other random variables besides $U^q.$

\subsubsection{Bounding Marginal Increments}
For any fixed $\bfg_{-d},$
\begin{equation}
\begin{split}
    D_d\gamma(\bfg_{-d})\; & =  
    \sup_{\bfg^M_d}
    \gamma(\bfg_{-d},\bfg^M_d) - 
    \inf_{\bfg^m_d}
    \gamma(\bfg_{-d},\bfg^m_d)\\
    & =  
    \sup_{\bfg^M_d}\big|B_{L_q(\dtorus),\tau}(\bfg_{-d},\bfg^M_d)\cap
     B_{L_q(\dtorus),\tau}(\bfnull)\big| - 
     \inf_{\bfg^m_d}\big|B_{L_q(\dtorus),\tau}(\bfg_{-d},\bfg^m_d)\cap
     B_{L_q(\dtorus),\tau}(\bfnull)\big|\\
     & = 
     |B_{L_q(\dtorus),\tau}(\bfg_{-d},0)\cap
     B_{L_q(\dtorus),\tau}(\bfnull)\big| - 
     \big|B_{L_q(\dtorus),\tau}(\bfg_{-d},1)\cap
     B_{L_q(\dtorus),\tau}(\bfnull)\big|\\
     & \le 
     |
     B_{L_q(\dtorus),\tau}(\bfg_{-d},0)\backslash
     B_{L_q(\dtorus),\tau}(\bfg_{-d},1)
     |\\
     & = 
     |B_{L_q(\dtorus),\tau}(\bfnull)\backslash
     B_{L_q(\dtorus),\tau}(\bfnull_{-d},1)|.
\end{split}
\end{equation}
Now, observe that a point $(h_1, h_2, \ldots, h_{d-1}, h_{d})$ is in $B_{L_q(\dtorus),\tau}(\bfnull)\backslash
     B_{L_q(\dtorus),\tau}(\bfnull_{-d},1)$ if and only if 

\begin{equation*}
    \begin{split} \sum_{i = 1}^{d-1} |h_i|_C^{q} \le  
\tau^q - |h_d|_C^q\quad \text{ and }\quad
 \sum_{i = 1}^{d-1} |h_i|_C^{q} \ge  
\tau^q - |1 - h_d|_C^q = 
\tau^q - (1-|h_d|_C)^q.
    \end{split}
\end{equation*}
Clearly, one needs to have $|h_d|_C \in [0,1/2]$ for this event to occur.  
Since each $|h_i|_C$ is uniformly distributed on $[0,1],$ we conclude that the probability of this event is 
\begin{equation}
\label{eq:supnormclaimuse}
    \begin{split}
& \int_0^{1/2}\prob\Bigg[\tau^q -(1-\ell)^q\le \sum_{i = 1}^{d-1} |h_i|_C^{q}\le \tau^q - \ell^q\Bigg]d\ell\\
& \le  
\int_{0}^1\prob\Bigg[\tau^q -\ell^q\le \sum_{i = 1}^{d-1} |h_i|_C^{q}\le \tau^q \Bigg]d\ell\\
 & = 
 \int_0^1 F([\tau^q - \ell^q, \tau^q])d\ell = 
 \int_0^1 \psi(\ell)d\ell = 
 O\left(\frac{1}{\sqrt{qd}}\right),
    \end{split}
\end{equation}
as desired. 
\subsubsection{Bounding Marginal Variances}
For the second moment, we will first rewrite $\gamma.$ By definition,
\begin{equation}
    \begin{split}
\gamma(\bfg) + \frac{1}{4} = &
\int_{B_{L_q(\dtorus), \tau}(\bfnull)}
\indicator\bigg[\bfz \in B_{L_q(\dtorus), \tau}(\bfg)\bigg]d\bfz\\
= & 
\int_{B_{L_q(\dtorus), \tau}(\bfnull)}
\indicator\Bigg[\sum_{i = 1}^d |g_i -z_i|_C^q \le \tau^q\Bigg]d\bfz.
    \end{split}
\end{equation}
Now, fix $\bfg_{-d}$ and denote 
\begin{equation*}
    \kappa(u) = 
\int_{B_{L_q(\dtorus), \tau}(\bfnull)}
\indicator\bigg[\sum_{i = 1}^{d-1} |g_i - z_i|_C^q + 
|u-z_d|_C^q \le \tau^q\bigg]d\bfz.
\end{equation*}
$\Var_{U\sim \unif(\unitcircle)
}[\kappa(U)]$ is exactly $\Var_d[\gamma(\bfg_{-d})].$ By definition, 
\begin{equation*}
\begin{split}
    \expect_{U \sim \unif(\unitcircle)}[\kappa(U)] &= 
    \expect_U\Bigg[
    \int_{B_{L_q(\dtorus), \tau}(\bfnull)}
    \indicator[\sum_{i = 1}^{d-1} |g_i - z_i|_C^q + 
    |U-z_i|_C^q \le \tau^q]d\bfz
    \Bigg]\\
    & = 
    \expect_V\Bigg[
    \int_{B_{L_q(\dtorus), \tau}(\bfnull)}
    \indicator[\sum_{i = 1}^{d-1} |g_i - z_i|_C^q + 
    |V|_C^q \le \tau^q]d\bfz
    \Bigg]: = \expect[\rho(V)],
\end{split}
\end{equation*}
where $\rho$ is defined by the last equation, i.e., 
$$
\rho(v) \coloneqq  
\int_{B_{L_q(\dtorus), \tau}(\bfnull)}
    \indicator\bigg[\sum_{i = 1}^{d-1} |g_i - z_i|_C^q + 
    |v|_C^q \le \tau^q\bigg]d\bfz.
$$
On the other hand, 
\begin{equation*}
\begin{split}
    &\expect_{U \sim \unif(\unitcircle)}[\kappa^2(U)]\\
    & =
    \expect_U\Bigg[
    \int_{B_{L_q(\dtorus), \tau}(\bfnull)}
    \indicator\bigg[\sum_{i = 1}^{d-1} |g_i - z^1_i|_C^q + 
    |U-z^1_i|_C^q \le \tau^q\bigg]d\bfz^{1}\\
    &\quad \quad \times 
    \int_{B_{L_q(\dtorus), \tau}(\bfnull)}
    \indicator\bigg[\sum_{i = 1}^{d-1} |g_i - z^2_i|_C^q + 
    |U-z^2_i|_C^q \le \tau^q\bigg]d\bfz^{2}
    \Bigg]\\
    & = \expect_U\Bigg[
    \int_{B_{L_q(\dtorus), \tau}(\bfnull)}
    \indicator\bigg[\sum_{i = 1}^{d-1} |g_i - z^1_i|_C^q + 
    |V|_C^q \le \tau^q\bigg]d\bfz^{1}\\
    &\quad \quad \times 
    \int_{B_{L_q(\dtorus), \tau}(\bfnull)}
    \indicator\bigg[\sum_{i = 1}^{d-1} |g_i - z^2_i|_C^q + 
    |V + z^1_i-z^2_i|_C^q \le \tau^q\bigg]d\bfz^{2}
    \Bigg]\\
    & = \expect_{V, R}[\rho(V)\rho(V+R)],
\end{split}
\end{equation*}
where $V = U-z^1_d\sim \unif(\unitcircle)$ and 
$R \sim z^1_d -z^2_d$ and $V, R$ are independent. It follows that 
\begin{equation}
    \Var[\kappa] = \expect_{V, R}[\rho(V)\rho(V + R)] - 
    \expect[\rho(V)]^2.
\end{equation}
Since $\rho:\unitcircle \longrightarrow [0,1]$ is clearly $L_2$-integrable, we can write its Fourier series. Furthermore, as $\rho(v) = \rho(-v)$ and $\rho $ is real, we can write 
$$
\rho(v) = 
\widehat{\rho}(0) + 
\sum_{k \ge  1}
2\widehat{\rho}(k) \cos(\pi k v).
$$
In particular, we have $\expect_{V\sim \unif(\unitcircle)}[\rho(V)] = \widehat{\rho}(0)$ and, using the convolution formula, 

\begin{equation}
    \begin{split}
        \expect_{V, R}[\rho(V)\rho(V+R)] = 
        \widehat{\rho}(0)^2 + \sum_{k \ge 1}2{\widehat{\rho}(k)^2}\expect_R[\cos(\pi k R)].
    \end{split}
\end{equation}
Putting all of this together, we have 
\begin{equation}
\begin{split}
    &\Var_d[\gamma(\bfg_{-d})] = 
    \sum_{k \ge 1}2{\widehat{\rho}(k)^2}\expect_R[\cos(\pi k R)]\le 
    \bigg(\sum_{k \ge 1}2{\widehat{\rho}(k)^2}\bigg)\times 
    \sup_{k\ge 1}
    \expect_R[\cos(\pi k R)]\\
    & =  
    \Var_{V\sim \unif(\unitcircle)}[\rho(V)]
    \times 
    \sup_{k\ge 1}
    \expect_R[\cos(\pi k R)].
\end{split}
\end{equation}
We will now bound the variance of $\rho$ and the cosine expectations separately. That is, we will show that 
$$
\Var_{V\sim \unif(\unitcircle)}[\rho(V)] = O\bigg(\frac{1}{d}\bigg)\quad\text{ and }\quad
\sup_{k\ge 1}
    \expect_R[\cos(\pi k R)] = 
    O\bigg(\frac{1}{qd}\bigg),
$$
which is enough.

\paragraph{1) Variance of $\rho.$}
\begin{equation*}
    \begin{split}
\Var[\rho]\; &\le 
\expect_V\bigg[
\big(
\rho(0)
- \rho(V)\big)^2\bigg]\\
&  =  
\expect_V\Bigg[
\Bigg(
\int_{B_{L_q(\dtorus), \tau}(\bfnull)}
    \indicator\bigg[\sum_{i = 1}^{d-1} |g_i - z_i|_C^q \le \tau^q\bigg]d\bfz - 
    \indicator\bigg[\sum_{i = 1}^{d-1} |g_i - z_i|_C^q  + |V|_C^q\le \tau^q\bigg]d\bfz
    \Bigg)^2\Bigg]\\
 & = 
\expect_V\Bigg[
\Bigg(
\int_{B_{L_q(\dtorus), \tau}(\bfnull)}
    \indicator\bigg[\sum_{i = 1}^{d-1} |g_i - z_i|_C^q \in [\tau^q, \tau^q - |V|^q_C]\bigg]d\bfz
    \Bigg)^2\Bigg]\\ 
& \le 
\expect_V\Bigg[
\Bigg(
\int_{\dtorus}
    \indicator\bigg[\sum_{i = 1}^{d-1} |g_i - z_i|_C^q \in [\tau^q, \tau^q - |V|^q_C]\bigg]d\bfz
    \Bigg)^2\Bigg].
    \end{split}
\end{equation*}
Since the integral is over the entire torus, the variables $\{|g_i - z_i|_C\}_{i = 1}^{d-1}$ are iid uniformly distributed over $[0,1],$ just like $V.$ Therefore, the last expression equals 
\begin{equation}
   \int_0^1 F([\tau^q - \ell^q, \tau^q])^2d\ell  = O\left(\frac{1}{d}\right),
\end{equation}
where we used \cref{claim:lqanticoncentration}.

\paragraph{2) Cosine Expectation.}
We need to find
\begin{equation}
    \sup_{k\ge 1}
    \expect_R[\cos(\pi k R)], \quad R\sim z^1_d - z^2_d,
\end{equation}
where $z^1_d, z^2_d$ are independent copies of the last coordinate of uniformly random point in $B_{L_q(\dtorus), \tau}(\bfnull).$
First, we will make a few simple observations abound the density of $R.$ Let the density of $z^1_d$ be $\nu(x).$ Note that 
$$
\nu(x)\propto 
F(\tau^q - |x|_C^q)
$$
since $(z_1, z_2, \ldots, z_{d-1},x)\in B_{L_q(\dtorus), \tau}(\bfnull)$ if and only if 
$$
\sum_{i=1}^{d-1}|z_i|_C^q\le \tau^q - x^q, 
$$
but $|z_1|_C, |z_2|_C, \ldots, |z_{d-1}|_C$ are iid $\unif([0,1])$ variables. In particular, this has the following implications:
\begin{enumerate}
    \item $|x|_C\longrightarrow \nu(x)$ is positive and decreasing.
    \item $\nu$ is even, i.e. $\nu(x) = \nu(-x).$
\end{enumerate}
Now, $R \sim z^1_d-z^2_d\sim z^1_d + z^2_d.$ Thus, if $\mu$ is the distribution of $R,$ clearly $\mu = \nu *\nu.$ In particular:
\begin{equation}
\label{eq:mudensityproperties}
    \begin{split}
        & \text{1. $\mu$ is even, i.e. $\mu(y) = \mu(-y),$}\\
        &\text{2. $|y|_C\longrightarrow \mu(y)$ is decreasing.}
    \end{split}
\end{equation}
The first fact is trivial. The second fact for $y \in [0,1]$ can be shown as follows. First, note that $\nu'(x) \le 0 $ for $x\in [0,1]$ as $|x|_C\longrightarrow \nu(x)$ is decreasing
    and $\nu'(x) = -\nu'(-x)$ since $x$ is even. 
    Now, 
    \begin{equation*}
        \begin{split}
            \mu'(y)\; & = (\nu * \nu)'(y) = 
            (\nu'*\nu)(y) = 
            \int_{-1}^1
            \nu'(x)\nu(y-x)dx\\
            & =  
            \int_0^1
            \nu'(x)\nu(y-x)dx + 
            \int_{-1}^0
            \nu'(x)\nu(y-x)dx\\
            & =  
            \int_0^1
            \nu'(x)\nu(y-x)dx + 
            \int_{0}^1
            \nu'(-x)\nu(y+x)dx\\
            & =  
            \int_0^1
            \nu'(x)\big(\nu(y-x) - \nu(y+x)\big)dx.
        \end{split}
    \end{equation*}
We know that $\nu'(x)\le 0$ for $x\in [0,1].$ On the other hand $\nu(y-x) \ge \nu(y+x)$ holds because  $|z|_C\longrightarrow \nu(z)$ is decreasing and 
$|y-x|_C\le |y+x|_C$ whenever $x,y\in [0,1].$ To show the last part, note that $|y-x|_C\in \{y-x, x-y\},$ and 
$|y+x|_C\in \{y+x, 2- y - x\}.$ However, 
\begin{enumerate}
    \item $y- x, x-y \le y + x$ whenever $x,y \ge 0.$
    \item $y - x, x - y \le 2 - y - x$ whenever $x,y \le 1.$ 
\end{enumerate}
We split the rest of the proof into two claims.

\begin{claim}
    $\sup_{k\ge 1}
    \expect_R[\cos(\pi k R)] \le 
    2\TV\Big(R, \unif([-1,1])\Big).$
\end{claim}
\begin{proof}
Let $(R,U)$ be an optimal coupling of a $\unif([-1,1])$ random variable $U$ with $R.$ Then, 
\begin{equation}
    \begin{split}
        \expect [\cos (\pi k R)] =& 
\expect[\cos(\pi k  R) - \cos(\pi k U)] = 
\expect\bigg[\indicator[U \neq R]\big(\cos(\pi k  R) - \cos(\pi k U)\big)\bigg]\\
\le & 
\|\cos(\pi k R) - \cos(\pi k U)\|_\infty\times 
\expect\bigg[\indicator[U \neq R]\bigg] \le 
2\TV\Big(R, \unif([-1,1])\Big).
    \end{split}
\end{equation}
\end{proof}

\begin{claim}
    $\TV\Big(R, \unif([-1,1])\Big)
     = O(1/dq).$
\end{claim}
\begin{proof} We use properties of the aforementioned density $\mu.$ Let $U\sim \unif([-1,1]).$
\begin{equation}
    \begin{split}
        \TV(R, U) &= \sup_A\bigg(\prob[U \in A] - \prob[R \in A] \bigg)\\
        &= 
\sup_A\int_A \bigg(\frac{1}{2} - \mu(u)\bigg)dy\le 
|A|\times \bigg(\frac{1}{2} - \inf_y\mu(y)\bigg)\le 
\bigg(\frac{1}{2} - \inf_y\mu(y)\bigg).
    \end{split}
\end{equation}

Our last step will be to show that $\inf_y\mu(y) = \frac{1}{2}- O\bigg(\frac{1}{dq}\bigg).$ As we know, $\nu(x) \propto F(\tau^q - |x|_C^q) = F(\tau^q) - \psi(|x|_C) = C - \psi(|x|_C),$ where 
$$C = F(\tau^q) = 
\prob\bigg[\sum_{i = 1}^{d-1}U_i^q\le \tau^q\bigg]\ge 
\prob\bigg[\sum_{i = 1}^{d}U_i^q\le \tau^q\bigg]
\ge \frac{1}{2}.$$ 
Thus,
$$
\mu(y) = (\nu*\nu)(y)\propto 
\int_{-1}^1
F(\tau^q - |x|_C^q)
F(\tau^q - |y - x|_C^q)dx = 
\int^1_{-1} (C - \psi(|x|_C))(C - \psi(|y - x|_C))dx.
$$
Now, we will find the normalizing constant in $\propto.$ 
\begin{equation}
\label{eq:secondmomentclaimusecosine}
    \begin{split}
K\; &= 2\int^1_{-1}\int^1_{-1} (C - \psi(|x|_C))(C - \psi(|y - x|_C))dxdy = 
2\bigg(
\int^1_{-1} (C - \psi(|x|_C))dx
\bigg)^2\\
&=  
2\bigg(
2C - \int_{-1}^1\psi(|x|_C)dx 
\bigg)^2 = 
2\Bigg(4C^2 - 4C\int_{-1}^1\psi(|x|_C)dx + \bigg(\int_{-1}^1\psi(|x|_C)dx \bigg)^2\Bigg)\\
&=8C^2 - 8C\int_{-1}^1\psi(|x|_C)dx + O\bigg(\frac{1}{dq}\bigg),
    \end{split}
\end{equation}
where we used \cref{claim:lqanticoncentration}. Note that $K = \Theta(1)$ since $C \ge 1/2$ and $\int_{-1}^1\psi(|x|_C)dx = O(\sqrt{\frac{1}{qd}}) = o(1)$ by \cref{claim:lqanticoncentration}.
However, we know that $\inf_y \mu(y) = \mu(1)$ by \cref{eq:mudensityproperties}, so
\begin{equation*}
    \begin{split}
        \mu(1)\; &= 
        \frac{1}{K}\int_{-1}^1
        (C - \psi(|x|_C))(C - \psi(|1 - x|_C))dx\\
        & =  \frac{1}{K}\Bigg(4C^2 - 2C\bigg(
        \int_{-1}^1\psi(|x|_C)dx + 
        \int_{-1}^1\psi(|1-x|_C)dx
        \bigg) + 
        \int_{-1}^1\psi(|x|_C)dx \times 
        \int_{-1}^1\psi(|1-x|_C)dx\bigg)\\
        & \ge \frac{1}{K}\Bigg(4C^2 - 2C\bigg(
        \int_{-1}^1\psi(|x|_C)dx + 
        \int_{-1}^1\psi(|1-x|_C)dx
        \bigg)\Bigg)\\
        & = 
        \frac{1}{K}\Bigg(4C^2 - 4C
        \int_{-1}^1\psi(|x|_C)dx\Bigg)\\ 
        & =  \frac{K/2 - O\bigg(\frac{1}{dq}\bigg)}{K} = 
        \frac{1}{2} - O\bigg(\frac{1}{dq}\bigg),
    \end{split}
\end{equation*}
with which the desired bound on the cosine expectation follows. 
\end{proof}

\subsection{The Proof for Large \texorpdfstring{$q$}{q}}
When $q = \Omega(d/\log d),$ we will follow a similar strategy as in the proof for the case of $q = o(d/\log d).$ Namely, our goal will be to prove that for any integer $k,$
\begin{equation}
\label{eq:momentnormslargeq}
    \expect_{\bfg\sim\unif(\dtorus)}[|\gamma(\bfg)|^k]^{1/k} \le  C(\log d)^C\frac{k}{d}.
\end{equation}
for some absolute constant $C.$ Following the same steps as in \cref{eq:smallqboundingKLdivergence}, this will be enough to conclude the second part of \cref{thm:torussmallq}. Again, we will use the Bernstein-McDiarmid approach to bounding the moments of 
of $\gamma.$ Our goal, this time, is to show the following. 
\begin{enumerate}
    \item \emph{Small Marginal Increments:} $\|D_i\gamma\|_\infty = \tilde{O}(\frac{1}{d})$ for all $i.$
    \item \emph{Small Marginal Variances:} $\|\Var_i[\gamma]\|_\infty = \tilde{O}(\frac{1}{d^3})$ for all $i.$
\end{enumerate}
We use the following anticoncentration results instead of \cref{claim:lqanticoncentration}. The rest of the proof is exactly the same.

\begin{claim}
\label{claim:lqanticoncentrationlarge}
Suppose that 
$U_1, U_2, \ldots, U_{d-1}$ are iid $\unif([0,1])$ random variables and $q = \Omega(d/\log d).$ Let $F$ be the CDF of $\sum_{i = 1}^{d-1}U_i^q.$ Then, for $\psi(\ell)\coloneqq F(\tau^q) - F(\tau^q - \ell^q) = F([\tau^q - \ell^q, \tau^q]),$ we have 
\begin{equation}
\int_0^1
\psi(\ell)d\ell = 
\tilde{O}\left(
\frac{1}{d}
\right)\quad \text{ and }\quad
 \int_0^1
\psi(\ell)^2d\ell = \tilde{O}\left(\frac{1}{d}\right).
\end{equation}
\end{claim}
The proof of \cref{claim:lqanticoncentration} is substantially different (and much simpler) than the proof of \cref{claim:lqanticoncentration}. As we will need one of the ingredients in the next section as well, we present it in full detail here.

\begin{lemma}
\label{lem:smallballq}
Suppose that 
$U_1, U_2, \ldots, U_{d-1}$ are iid $\unif([0,1])$ random variables and $q \ge 1.$ Then, for any interval $[a,b],$
$$
\prob\bigg[\sum_{i  =1}^{d-1}U_i^q\in [a,b]\bigg]\le 
b^{(d-1)/q} - a^{(d-1)/q}.
$$
\end{lemma}
\begin{proof}
The main idea is to reduce the computation to a computation for $q = \infty.$
Let $W = \max(U_1, \ldots, U_{d-1})$ and
$V_1, V_2, \ldots, V_{d-2}\iidsim \unif([0,1])$ be 
independent of $W.$ Then, $\sum_{i  =1}^{d-1}U_i^q$ has the same distribution as $W^q\times \big(1 + \sum_{j  =1}^{d-2}V_j^q\big).$ 
This follows simply by conditioning on the maximal value $w$ of $U_1, U_2, \ldots, U_{d-1}.$ Denote $T = 1 + \sum_{j  =1}^{d-2}V_j^q$ and observe that $T\ge 1$ a.s. This implies that 
\begin{equation*}
    \begin{split}
        & \prob\bigg[\sum_{i  =1}^{d-1}U_i^q\in [a,b]\bigg] = 
        \prob\bigg[W^q\times T\in [a,b]\bigg]\le 
        \sup_{t \ge 1}
        \prob\bigg[W^q\times t\in [a,b]\bigg] =  
        \sup_{t\ge 1}
        \prob\bigg[\big(\frac{a}{t}\big)^{1/q}\le W\le 
        \big(\frac{b}{t}\big)^{1/q}
        \bigg].
    \end{split}
\end{equation*}
Now, since $W$ is the maximum of $d-1$ iid $\unif([0,1])$ random variables, $\prob[W\le x] = x^{d-1}$ for any $x\in [0,1].$ Thus, 
$\displaystyle \prob\bigg[\big(\frac{a}{t}\big)^{1/q}\le W\le 
        \big(\frac{b}{t}\big)^{1/q}
        \bigg] = \big(\frac{b}{t}\big)^{\frac{d-1}{q}} - 
    \big(\frac{a}{t}\big)^{\frac{d-1}{q}}\le  b^{(d-1)/q} - a^{(d-1)/q},$
where we used $t\ge 1.$
\end{proof}

Now, we are ready to prove \cref{claim:lqanticoncentrationlarge}.

\begin{proof}[Proof of \cref{claim:lqanticoncentrationlarge}] 
Suppose that $q\ge d/(C'\log d)$ for some absolute constant $C'.$ We begin by proving the following two simple statements:
\begin{enumerate}
    \item $\tau \ge 1 - 1/d.$ Recall that $\tau$ is defined as the radius  of a $1/2$ volume ball in $(\dtorus, L_q).$ Let $U_1,U_2, \ldots, U_{d}$ be iid $\unif([0,1])$ random variables. Thus,
    $1/2 = \prob\big[\tau\ge \|(U_1, U_2, \ldots, U_d)\|_q\big].$ However,
    $$ \prob\big[1 - 1/d\ge \|(U_1, U_2, \ldots, U_d)\|_q\big]\le 
    \prob\big[1 - 1/d\ge \|(U_1, U_2, \ldots, U_d)\|_\infty\big] = (1-1/d)^d \le 1/e<1/2
    ,$$  
    which means that $\tau\ge 1-1/d.$
    \item $\tau^q\le C''(\log d)$ for some constant $C''$ depending solely on $C'.$ Observe that each variable $U^q_i$ has expectation $1/(q+1),$ variance lass than 
    $\expect[U_i^{2q}] = 1/(2q+1)$ and is bounded between 0 and 1. Thus, by \cref{claim:bernsteinmcdiarmidrag}, 
    $$
\prob\bigg[
\sum_{j = 1}^d U_j^q\ge t + d/(q+1)
\bigg]\le 
\exp\Big(- \min\big\{\Theta(t^2/(d/q)), \Theta(t)\big\}\Big).
    $$
\end{enumerate} 
In particular, this means that setting $t = C''\times
\max(1, d/q) \le C''\times C'\times (\log d)$ for large enough $C'',$ we obtain a tail bound less than $1/2.$ Thus, $\tau^q\le C''\log d$ for some $C''.$

Now, we go back to proving \cref{claim:lqanticoncentrationlarge}. We begin with the first inequality. 

\begin{align*}
& \int_0^1
\psi(\ell)d\ell = 
\int_0^1
\prob\bigg[\tau^q-\ell^q\le \sum_{i = 1}^{d-1}U_i^{d-1}\le \tau^q\bigg]d\ell\\
& = \int_0^{1 - \frac{(\log d)^3}{d}}
\prob\bigg[\tau^q-\ell^q\le \sum_{i = 1}^{d-1}U_i^{d-1}\le \tau^q\bigg]d\ell + 
\int_{1 - \frac{(\log d)^3}{d}}^1
\prob\bigg[\tau^q-\ell^q\le \sum_{i = 1}^{d-1}U_i^{d-1}\le \tau^q\bigg]d\ell\\
& \le 
\prob\bigg[\tau^q-\bigg(1 - \frac{(\log d)^3}{d}\bigg)^q\le \sum_{i = 1}^{d-1}U_i^{d-1}\le \tau^q\bigg] +\frac{(\log d)^3}{d}.
\end{align*}
All that is left to do is bound $\prob\bigg[\tau^q-\bigg(1 - \frac{(\log d)^3}{d}\bigg)^q\le \sum_{i = 1}^{d-1}U_i^{d-1}\le \tau^q\bigg].$ Using \cref{lem:smallballq},
\begin{equation}
    \begin{split}
        & \prob\bigg[\tau^q-\bigg(1 - \frac{(\log d)^3}{d}\bigg)^q\le \sum_{i = 1}^{d-1}U_i^{d-1}\le \tau^q\bigg]\\
        & \le 
        (\tau^q)^{(d-1)/q} - \Bigg(\tau^q-\Bigg(1 - \frac{(\log d)^3}{d}\Bigg)^q\Bigg)^{(d-1)/q}\\
        & = 
        (\tau^{q})^{(d-1)/q}\times \Bigg[
        1 - \Bigg(1 - \bigg(\frac{1 - (\log d)^3/d}{\tau}\bigg)^{q}\Bigg)^{(d-1)/q}
        \Bigg].
    \end{split}
\end{equation}
Since $\tau \ge 1-1/d,$ it is the case that 
$\frac{1 - (\log d)^3/d}{\tau}\le 1 - (\log d)^3/(2d).$ Thus, $$\bigg(\frac{1 - (\log d)^3/d}{\tau}\bigg)^q\le \bigg(1 - (\log d)^3/(2d)\bigg)^q\le \exp\bigg(-(\log d)^3q/(2d)\bigg)
\le \exp\bigg( - \Theta((\log d)^2)\bigg).
$$
It follows that 
\begin{align*}
&\Bigg(1 - \bigg(\frac{1 - (\log d)^3/d}{\tau}\bigg)^{q}\Bigg)^{(d-1)/q}\ge 
\Bigg( 1- \exp\bigg( - \Theta((\log d)^2)\bigg)\Bigg)^{(d-1)/q}\\
& \ge 
\Bigg( 1- \exp\bigg( - \Theta((\log d)^2)\bigg)\Bigg)^{C'(\log d)} = 
1- \exp\bigg( - \Theta((\log d)^2)\bigg).
\end{align*}
Therefore, 
\begin{align*}
    & (\tau^{q})^{(d-1)/q}\times \Bigg[
        1 - \Bigg(1 - \bigg(\frac{1 - (\log d)^3/d}{\tau}\bigg)^{q}\Bigg)^{(d-1)/q}
        \Bigg]\\
    & \le
    (C''(\log d))^{C'\log d}\times \exp\bigg( - \Theta((\log d)^2)\bigg) = 
    \exp\bigg( - \Theta((\log d)^2)\bigg) = 
    o(1/d).
\end{align*}
With this, the proof of the first inequality is completed. 
The second inequality follows directly as
\[
\int_0^1
\psi(\ell)^2d\ell = 
\int_0^1
\prob\bigg[\tau^q-\ell^q\le \sum_{i = 1}^{d-1}U_i^{d-1}\le \tau^q\bigg]^2d\ell\le  
\int_0^1
\prob\bigg[\tau^q-\ell^q\le \sum_{i = 1}^{d-1}U_i^{d-1}\le \tau^q\bigg]d\ell=\int_0^1
\psi(\ell)d\ell.\qedhere
\]
\end{proof}
\section{Entropic Upper Bound in the \texorpdfstring{$L_q$}{Lq} Model}

\begin{theorem}[$\varepsilon$-Net Argument for Hard Threshold Random Geometric Graphs]
There exists some constant $C$ with the following property.
Consider a random geometric graph $\RGG(n,\Omega,\mathcal{D}, \sigma,p)$ over the metric space $(\Omega, \mu),$ where $1/n \le p\le 1/2$ and
$\sigma(\bfx, \bfy) = \indicator[\mu(\bfx,\bfy)\le \tau]$ for some $\tau.$ Suppose, further, that 
$(\Omega, \mu)$ has a finite $\varepsilon$-net $\mathcal{N}(\varepsilon)$ which satisfies the following property. $\prob_{\bfx,\bfy\iidsim\mathcal{D}}[\mu(\bfx,\bfy)\in [\tau-2\varepsilon, \tau+2\varepsilon]] = o(n^{-2}).$ If $|\mathcal{N}(\varepsilon)|\le \exp\Big(Cnp\log 1/p\Big),$ then 
$$
\TV\Big(\RGG(n,\Omega,\mathcal{D}, \sigma,p), \ergraph\Big) = 1- o(1).
$$
\end{theorem}
\begin{proof} First, we will show that there exists a graph distribution $\mathcal{Q}$
on support of size at most 
$|\mathcal{N}(\varepsilon)|^n$ such that 
$$
\TV\Big(\RGG(n,\Omega,\mathcal{D}, \sigma,p), \mathcal{Q}\Big) = o(1).
$$
Let $\pi$ be the projection map form $\Omega$ to $\mathcal{N}(\varepsilon).$ Let $\mathcal{D}'$ be the distribution over $\mathcal{N}(\varepsilon)$ defined by $\pi\circ\mathcal{D}.$ Let $p' = \prob_{\bfx,\bfy\iidsim 
\mathcal{D}'}\big[\mu(\bfx,\bfy)\le \tau\big].$ We will show that $\mathcal{Q} = \RGG(n,\mathcal{N}(\varepsilon),\mathcal{D}', \sigma,p')$ satisfies the desired property. Here, we think of $\mathcal{N}(\varepsilon)$ as a metric space with the induced metric $\mu.$

First, $\mathcal{Q}$ has support of size at most $|\mathcal{N}(\varepsilon)|^n$ as the $n$ latent vectors in $\mathcal{N}(\varepsilon)$ uniquely determine the corresponding geometric graph.

Second, we will form a coupling between $\RGG(n,\Omega,\mathcal{D}, \sigma,p)$ and 
$\RGG(n,\mathcal{N}(\varepsilon),\mathcal{D}', \sigma,p')$ as follows. For latent vectors $\bfx^1, \bfx^2, \ldots \bfx^n\in \Omega,$ let $\bfgg_\Omega(\bfx^1, \bfx^2, \ldots, \bfx^n)$ be the corresponding graph according to $\RGG(n,\Omega,\mathcal{D}, \sigma,p)$ and 
$\bfgg_{\mathcal{N}(\varepsilon)}(\pi(\bfx^1), \pi(\bfx^2), \ldots, \pi(\bfx^n))$ be the corresponding graph according to\linebreak $\RGG(n,\mathcal{N}(\varepsilon),\mathcal{D}', \sigma,p').$ By definition, when we take 
$\bfx^1, \bfx^2, \ldots, \bfx^n\iidsim\mathcal{D},$ it is the case that 
\begin{align*}
    & \bfgg_\Omega(\pi(\bfx^1), \pi(\bfx^2), \ldots, \pi(\bfx^n))\sim \RGG(n,\Omega,\mathcal{D}, \sigma,p)\text{ and,}\\
    & \bfgg_{\mathcal{N}(\varepsilon)}(\pi(\bfx^1), \pi(\bfx^2), \ldots, \pi(\bfx^n))
    \sim \RGG(n,\mathcal{N}(\varepsilon),\mathcal{D}', \sigma,p').
\end{align*}
All that is left to show is that with probability $1-o(1)$ over $\bfx^1, \bfx^2, \ldots, \bfx^n\iidsim\mathcal{D},$ it is the case that 
$\bfgg_\Omega(\pi(\bfx^1), \pi(\bfx^2), \ldots, \pi(\bfx^n)) = 
\bfgg_{\mathcal{N}(\varepsilon)}(\pi(\bfx^1), \pi(\bfx^2), \ldots, \pi(\bfx^n)).
$

Observe that whenever $\bfgg_\Omega(\pi(\bfx^1), \pi(\bfx^2), \ldots, \pi(\bfx^n)) \neq 
\bfgg_{\mathcal{N}(\varepsilon)}(\pi(\bfx^1), \pi(\bfx^2), \ldots, \pi(\bfx^n)),$ there exist some $i,j$ such that 
$$
\indicator\bigg[\mu(\bfx^i,\bfx^j)\le \tau\bigg]\neq 
\indicator\bigg[\mu(\pi(\bfx^i),\pi(\bfx^j))\le \tau\bigg].
$$
However, by triangle inequality,
$$
\bigg|\mu(\bfx^i,\bfx^j) - 
\mu(\pi(\bfx^i),\pi(\bfx^j))
\bigg|\le 
\mu(\bfx^i, \pi(\bfx^i)) + 
\mu(\bfx^j, \pi(\bfx^j))\le 2\varepsilon.
$$
In particular, this means that $\mu(\bfx^i,\bfx^j)\in [\tau-2\varepsilon,\tau+2\varepsilon].$ As this happens with probability $o(n^{-2})$ for a fixed pair $i,j,$  the union bound implies that this happens with probability 
$o(1)$ for some $i,j,$ which finishes the proof that 
$\TV\Big(\RGG(n,\Omega,\mathcal{D}, \sigma,p), \mathcal{Q}\Big) = o(1).$ Thus, it is enough to show that $\TV\Big(\mathcal{Q},\ergraph\Big) = 1- o(1)$ under the given conditions.
This follows immediately from $|\mathcal{N}(\varepsilon)|\le \exp\Big(Cnp\log 1/p\Big)$ as shown in \cite[Theorem 7.5]{bangachev2023random}.
\end{proof}

\cref{thm:entropybound} now immediately follows from the following proposition.

\begin{theorem} Consider any $q\in [1,+\infty)\cup\{\infty\}, d \ge n^\delta, p\ge n^{-1+\epsilon}.$ For $\varepsilon = \exp(-(\log nd)^4),$
there exists an $\varepsilon$-net of $(\dtorus, L_q)$ of size $\exp(\tilde{\Theta}(d)).$ Furthermore,  $\prob_{\bfx,\bfy\iidsim \dtorus}[\|\bfx-\bfy\|_q\in [\tau^q_p - 2\varepsilon, \tau^q_p+2\varepsilon] \le n^{-3}.$
\end{theorem}
\begin{proof} First, we will show the existence of a small $\varepsilon$ net. Let $k  = \lceil d/\varepsilon\rceil$ be an integer and consider the set $\mathcal{N} = \{i/k\in \mathbb{T}^1\; : \; 0\le i \le 2k-1\}^d\subseteq \dtorus.$ This is a set of size $(2k)^d = \exp(\tilde{\Theta}(d)).$ Furthermore, it is a $\varepsilon$-net for any $L_q$ geometry for the following reason. Take $\bfx\in \dtorus$ and let $\bfu = (u_1/k,u_2/k,\ldots, u_d/k)$ be the projection of $\bfx$ to $\mathcal{N}.$ Then, for any $q\in [1,+\infty)\cup\{\infty\},$
$$
\|\bfx - \bfu\|_q\le 
\|\bfx - \bfu\|_1= 
\sum_{j = 1}^d |x_u - u_i/k|\le d/k \le \varepsilon.
$$

Now, we need to show that for each $q,$ it is the case that $\prob_{\bfx,\bfy\iidsim \dtorus}[\|\bfx-\bfy\|_q\in [\tau^q_p - 2\varepsilon, \tau^q_p+2\varepsilon] \le n^{-3}.$ This is equivalent to showing that for $U_1, U_2, \ldots, U_d\iidsim\unif([0,1]),$ it is the case that 
\begin{align*}
& \prob\bigg[
\|
(U_1, U_2, \ldots, U_d)
\|_q
\in [\tau^q_p - 2\varepsilon, \tau^q_p + 2\varepsilon]\bigg] \le n^{-3} \Longleftrightarrow\\
& \prob\bigg[
\sum_{j = 1}^d U_j^d
\in [(\tau^q_p - 2\varepsilon)^{q}, (\tau^q_p + 2\varepsilon)^q]\bigg] \le n^{-3}.
\end{align*}
As in the proof of \cref{claim:lqanticoncentrationlarge}, clearly 
$(\tau^q_p)\ge 1 - (\log 1/p)/d\ge 1/2.$ Furthermore, note that
$(\tau^q_p)^q\le d$ as $\|
(U_1, U_2, \ldots, U_d)
\|^q_q\le d$ a.s. Now, we consider two cases:
\paragraph{Case 1)  When $q = o(d/(\log d)).$} Note that $(\tau^q_p +2\varepsilon)^{q}- (\tau^q_p + 2\varepsilon)^q = (\tau^q_p)^q \Bigg((1 + 2\varepsilon/\tau_q^p)^q - (1 -2\varepsilon/\tau_q^p)^q\Bigg).$ Using that $q = o(d/\log d) = o(1/\varepsilon),(\tau^q_p)^q\le d,\tau_p^q\ge 1/2,$ the last expression is of order ${O}(dq\varepsilon) = o(n^{-3}).$ By \cref{prop:inervalprobsmallq}, 
$\prob\bigg[
\sum_{j = 1}^d U_j^d
\in [(\tau^q_p - 2\varepsilon)^{q}, (\tau^q_p + 2\varepsilon)^q]\bigg]  = o (n^{-3}),$ as desired.

\paragraph{Case 2) When $q= \Omega(d/\log d).$} Using
\cref{lem:smallballq},
\begin{align*}
    & \prob\bigg[
\sum_{j = 1}^d U_j^d
\in [(\tau^q_p - 2\varepsilon)^{q}, (\tau^q_p + 2\varepsilon)^q]\bigg]\le 
\Bigg((\tau^q_p + 2\varepsilon)^q\Bigg)^{d/q} - 
\Bigg((\tau^q_p - 2\varepsilon)^q\Bigg)^{d/q}\\
& = 
((\tau^q_p)^q)^{d/q}\Bigg(
\bigg(1 +  2\varepsilon/\tau^q_p
\bigg)^d - 
\bigg(1 -  2\varepsilon/\tau^q_p
\bigg)^d\Bigg)\\
& \le d^{d/q}\times O(d\epsilon) = 
\exp(O((\log d)^2))\times 
\exp( - (\log (nd))^4)\le n^{-3}.\qedhere
\end{align*}

\end{proof}

\section{Random Algebraic Graphs Over the Hypercube}
\subsection{Preliminaries}
We begin with some preliminaries on Boolean Fourier analysis. 
Any function $f:\hypercube\longrightarrow \mathbb{R}$ can be written uniquely as $f(\bfx) = \sum_{S\subseteq [d]}\widehat{f}(S)\omega_S(\bfx),$ where $\omega_S(\bfx)\coloneqq \prod_{i \in S}x_i$ is the Walsh polynomial \cite{ODonellBoolean}. The influence $\infl_i[f]$ of variable $i$ is defined as
\begin{equation}
    \label{eq:influence}
    \infl_i[f] = \sum_{i \in S}\widehat{f}(S)^2 = 
    \expect_{\bfx\sim \unif(\hypercube)}\Big[
    (f(\bfx)  - f(\bfx^{\oplus i}))^2/4\Big],
\end{equation}
where $\bfx^{\oplus i}$ is $\bfx$ with the $i$-th coordinate flipped. We denote $\inflvector[f]$ as the vector in $\mathbb{R}_{\ge 0}^d$ with $i'$th coordinate equal to $\infl_i[f].$ In particular, 
$\|\inflvector[f]\|_1 = \sum_{i = 1}^d\infl_i[f] = \infl[f],$ which is the total influence, and 
$\|\inflvector[f]\|_\infty = \maxinfl[f],$ which is the max influence. Also, $\|\inflvector[f]\|^2_2 = \sum_i \infl^2_i[f],$ which is the quantity of interest in \cref{thm:hypercube}.
\subsection{The Proof of \texorpdfstring{\cref{thm:hypercube}}{Hypercube Theorem}}

 Throughout, we make the following assumption, without which the statement of \cref{thm:hypercube} is trivial (as it gives an upper bound of a total variation by a number larger than 1).

\begin{equation}
\label{assumption:l2norm}
\frac{n\|\inflvector[\sigma]\|_2}{p(1-p)} = o(1)
\end{equation} 

Write $\sigma$ in the standard Fourier basis as 
$
\sigma(\bfg) = p + \sum_{
\emptyset \subsetneq S\subsetneq [d]}
\widehat{\sigma}(S)
\omega_S(\bfg),
$
where  $\widehat{\sigma}(\emptyset) = \expect[\sigma] = p.$
Then, 
$
\gamma(\bfg) = \sigma*\sigma(\bfg) - p^2 = 
\sum_{
\emptyset \subsetneq S\subseteq [d]}
\widehat{\sigma}(S)^2
\omega_S(\bfg).
$
In particular, this means that for any $i \in [d],$ and 
$\bfh_{-i}\in \{\pm1\}^{d-1},$ we have
$$
\gamma(\bfg)|_{\bfg_{-i} = \bfh_{-i}} = 
\sum_{i \not \in S} 
\widehat{\sigma}(S)^2
\omega_S(\bfh_{-i})
  + 
\bfg_i
\sum_{i \in S} 
\widehat{\sigma}(S)^2
\omega_{S\backslash \{i\}}(\bfh_{-i}).
$$
It follows that $$D_i\gamma(\bfh_{-i}) = 2\sum_{i \in S} 
\widehat{\sigma}(S)^2
\omega_{S\backslash \{i\}}(\bfh_{-i}) \le 
2\sum_{i \in S} 
\widehat{\sigma}(S)^2 = 
2\infl_i[\sigma]$$ and 
$\Var_i[\gamma(\bfh_{-i})] = \Bigg(\sum_{i \in S} 
\widehat{\sigma}(S)^2
\omega_{S\backslash \{i\}}(\bfh_{-i})\Bigg)^2 \le 
\Bigg(\sum_{i \in S} 
\widehat{\sigma}(S)^2\Bigg)^2
= \infl^2_i[\sigma].$ 
Therefore, by \cref{claim:bernsteinmcdiarmidrag},
$$
\|\gamma\|_k \le 
C\Bigg(
\sqrt{k}\sqrt{\sum_{i = 1}^d\infl^2_i[\sigma]} + 
k\times \maxinfl[\sigma]
\Bigg) = 
C(\sqrt{k}\times \|\inflvector[\sigma]\|_2 + k \times 
\|\inflvector[\sigma]\|_\infty.
)
$$
This implies 
$$
\|\gamma\|_k^k \le 
(2C)^{k}\sqrt{k}^k\|\inflvector[\sigma]\|^k_2 + 
(2C)^kk^k\|\inflvector[\sigma]\|^k_\infty.
$$

Plugging this into \cref{eq:Kltensorization} and using $\expect_\bfg[\gamma(\bfg)] = 0,$ we obtain the following bound. The computation is analogous to \cref{eq:smallqboundingKLdivergence}.

\begin{align*}
 \KL\Big(\RAG(n,\hypercube,\sigma,p)\|\ergraph\Big)\; &\le 
\sum_{k =0}^{n-1}\log\Bigg( \expect_\bfg\bigg[\bigg(1 + \frac{\gamma(\bfg)}{p(1-p)}\bigg)^k\bigg]\Bigg)\\
& \le \sum_{k = 0}^{n-1}\log \Bigg(
1 + \sum_{t =2}^k\binom{k}{t}\frac{\expect[|\gamma|^t]}{p^{t}(1-p)^t}
\Bigg)\\
& \le n\sum_{k = 2}^{n}
\binom{n}{k}\frac{\expect[|\gamma|^k]}{p^k(1-p)^k}\\
& \le 
n
\sum_{k \ge 2}\binom{n}{k}(2C)^k\sqrt{k}^k\frac{\|\inflvector[\sigma]\|_2^k}{p^k(1-p)^k} +
n
\sum_{k \ge 2}\binom{n}{k}(2C)^k{k}^k\frac{\|\inflvector[\sigma]\|_\infty^k}{p^k(1-p)^k}.
\end{align*}
We now handle the two sums separately. We will use the inequality $\binom{n}{k}\le (ne/k)^k.$

\paragraph{Sum depending on $L_2$ norm.} 
\begin{equation}
	\begin{split}
		&n
\sum_{k \ge 2}\binom{n}{k}(2C)^k\sqrt{k}^k\frac{\|\inflvector[\sigma]\|_2^k}{p^k(1-p)^k}
\le n
\sum_{k \ge 2}
\Bigg(\frac{2eCn\|\inflvector[\sigma]\|_2}{\sqrt{k}p(1-p)}\Bigg)^k.
	\end{split}
\end{equation}
We will show exponential decay in the summands. That is, for all $k\ge 2,$
$$
\Bigg(\frac{2eCn\|\inflvector[\sigma]\|_2}{\sqrt{k}p(1-p)}\Bigg)^k\ge 
2 \Bigg(\frac{2eCn\|\inflvector[\sigma]\|_2}{\sqrt{k+1}p(1-p)}\Bigg)^{k+1}.
$$
This is equivalent to 
$
\sqrt{k+1}\ge C'
\frac{n\|\inflvector[\sigma]\|_2}{p(1-p)}
$ for some absolute constant $C'.$
The latter inequality clearly holds for all $k\ge 2$ by \cref{assumption:l2norm}. Since there is exponential decay, the term for $k = 2$ is dominant and, thus, the entire expression is of order
$
O\Bigg(\frac{n^3\|\inflvector[\sigma]\|_2^2}{p^2(1-p)^2}\Bigg).$

\paragraph{Sum depending on $L_\infty$ norm.}
Using the same reasoning, the expression can be bounded by 
\begin{equation}
n\sum_{k\ge 2}
\Bigg(
\frac{2eCn\|\inflvector[\sigma]\|_\infty}{p(1-p)}
\Bigg)^k.
\end{equation}
Again, whenever $\frac{n\|\inflvector[\sigma]\|_\infty}{p(1-p)} = o(1),$ we have exponential decay. This, however, clearly is the case by \cref{assumption:l2norm} as $\|\inflvector[\sigma]\|_\infty\le 
\|\inflvector[\sigma]\|_2.$ Thus, the term for $k = 2$ is dominant,
so the $L_\infty$ contribution is bounded by
$
O\Bigg(\frac{n^3\|\inflvector[\sigma]\|_\infty^2}{p^2(1-p)^2}\Bigg).
$
Combining with the $L_2$ contribution, the statement follows as 
$\|\inflvector[\sigma]\|^2_\infty\le 
\|\inflvector[\sigma]\|^2_2.$
\section{Discussion}
\label{sec:discussion}
We studied the question of detecting $L_q$ geometry in random geometric graphs. Our work shows that for different values of $q,$ not only the limits of computational and statistical detection vary, but also the optimal algorithms are different. In particular, contrary to previous work, we show that the signed triangle count is not always optimal as the signed 4-cycle test might succeed in a polynomially larger range. This, however, opens more questions than it answers. What other tests besides counting signed 3- and 4- cycles can be optimal for detecting latent geometry? Are there instances in which a statistical-computational gap for detecting geometry is present? A positive answer to this question might even be hidden in the $\RGG(n,\dtorus, \unif, \sigma^q_p,p)$ models considered in the current paper as our statistical lower-bounds and computationally efficient algorithmic upper bounds  are essentially nowhere matching. 

Similarly, one can study other statistical tasks related to random geometric graphs with $L_q$ geometry. Especially intriguing seems the task of efficiently embedding a sample from $\RGG(n,\dtorus, \unif, \sigma^q_p,p)$ into $(\dtorus, \|\cdot\|_q)$ so that marginal distances are non-trivially approximated. This question will most likely require new ideas, different from previous work on embedding random geometric graphs. The spectral approach of \cite{li2023spectral} heavily relies on an inner product structure, which is only present in  $\RGG(n,\dtorus, \unif, \sigma^q_p,p)$ when $q = 2.$ The optimization framework of \cite{Ma2020UniversalLS} works in settings of $L_q$ geometry for general $q,$ but only gives strong poly-time guarantees for connection functions bounded away from $0$ and $1,$ i.e. $c\le\sigma(\bfx, \bfy)\le 1-c $ for some $c>0.$ This however, is not the case in $\RGG(n,\dtorus, \unif, \sigma^q_p,p)$ as $\sigma^q_p$ only takes values 0 and 1.

Finally, we compared the $L_\infty$ geometry giving rise to an \textsf{AND} structure in random geometric graphs and the $L_2$ geometry giving rise to a (weighted) \textsf{MAJORITY} structure (as would any $L_C$ when $C = O(1)$). It could be interesting to consider an extension of these constructions for general $f:\{0,1\}^d\longrightarrow\{0,1\}$ beyond \textsf{AND} an \textsf{MAJORITY}.
One way to formalize is the following. One first samples $d$ graphs on $n$ vertices $\bfG^1, \bfG^2, \ldots, \bfG^d$ from some fixed ``1-dimensional'' distribution $\mathcal{G}$ and then forms the $d$ dimensional graph $\bfG$ in which $\bfG_{ij} = f(\{\bfG^u_{ij}\}_{u = 1}^d).$ When do such graphs converge to \ER? When can they be represented as random geometric graphs? We note that in the special case $f= \textsf{AND}$ (respectively, $f= \textsf{OR}$ under taking a complement) one can apply our cluster-expansion based approach for \textit{any} 1-dimensional distribution $\mathcal{G}.$ An example of the $\textsf{AND}$ structure beyond $L_\infty$ toric random geometric graphs are random intersection graphs (e.g. \cite{Brennan19PhaseTransition}).

\section*{Acknowledgements}
We want to thank Will Perkins for insightful conversations on the cluster expansion formula and for suggesting that it might be useful in understanding random geometric graphs. We also want to thank Tselil Schramm for stimulating conversations on random geometric graphs.
\printbibliography

\appendix
\section{Signed Counts in \texorpdfstring{$L_q$}{Lq} Geometries}
\label{appendix:signedcountsinLq}
Our only rigorous progress towards signed  subgraph tests in $L_q$ geometries is the following. 
\begin{theorem}
\label{thm:signedfourcyclesinlq}
The signed 4-cycle test cannot distinguish between 
$H_0: \ergraphhalf$ and\linebreak
$H_1:\RGG(n,\dtorus,\unif, \sigma^q_{1/2},1/2)$ in the following regimes:
\begin{enumerate}
    \item When $q = o(d/\log d)$ and $dq = \omega(n^2).$ 
    \item When $q = \Omega(d/\log d)$ and $d = \tilde{\omega}(n).$
\end{enumerate}
\end{theorem}
\begin{proof} The main observation is that the signed 4-cycle count corresponds to the second moment of $\gamma.$  
\begin{align*}
    & \expect_{\bfG\sim \RGG(n,\dtorus, \unif, \sigma^q_{1/2}, 1/2)}[\signedweight_{C_4}(\bfG)]
     = 
    \expect_{\bfG\sim \RGG}[
    (\bfG_{12} - 1/2)(\bfG_{23}  - 1/2)
    (\bfG_{34} - 1/2)(\bfG_{41}  - 1/2)
    ]\\
        & = 
    \expect_{\bfg_1, \bfg_2, \bfg_3, \bfg_4\iidsim\unif(\dtorus)}
    \Big[\big(\sigma(\bfg_1 - \bfg_2) - 1/2\big)
        \big(\sigma(\bfg_2 - \bfg_3) - 1/2\big)
        \big(\sigma(\bfg_3 - \bfg_4) - 1/2\big)
        \big(\sigma(\bfg_4 - \bfg_1) - 1/2\big)\Big]\\
    & = 
    \expect_{\bfh, \bfz_1, \bfz_2\iidsim\unif(\dtorus)}
    \Big[\big(\sigma(\bfz_1) - 1/2\big)
        \big(\sigma(\bfh - \bfz_1) - 1/2\big)
        \big(\sigma(\bfz_2 ) - 1/2\big)
        \big(\sigma(\bfh - \bfz_2) - 1/2\big)\Big]\\
    & = \expect\Big[\big(\sigma*\sigma(\bfh)-1/4\big)^2\Big],
\end{align*}
as desired. We used the substitution $\bfz_1 =\bfg_1 - \bfg_2, \bfz_2 = \bfg_1 - \bfg_4, \bfh = \bfg_1 - \bfg_3.$ Recalling \cref{eq:momentnorms,eq:momentnormslargeq}, we conclude that the signed count is of order $O(1/dq)$ in the regime $q = o(d/\log d))$ and of order $\tilde{O}(1/d)$ in the regime $q = \Omega(d/\log d).$ However, $K_n$ has $\Theta(n^4)$ subgraphs isomorphic to $C_4$ and
$\Var_{\bfH\sim \ergraphhalf}[\signedcount_{C_4}(\bfH)] = \Theta(n^4)$ by \cref{eq:fourcyclevars}. Therefore, a necessary condition for detection via the signed 4-cycle test is 
$n^4\expect_{\bfG\sim \RGG}[\signedweight_{C_4}(\bfG)] = \omega(\sqrt{n^4}).$ This leads to the desired conclusion.
\end{proof}

We believe that $1/dq$ and $1/d^2$ are the correct (up to $\log$ factors) orders of the signed 4-cycle count in the two regimes. Note that when $q = \infty,$ the signed 4-cycle count is indeed ${\Theta}(d^{-2})$ by \cref{cor:linftysignedcycles}. Similarly, in $L_2$ geometry (admittedly over a different latent space such as $\hypercube,$ but again with a hard threshold connection with density $1/2$), the signed 4-cycle count is 
$\tilde{\Theta}(1/d)$ (follows directly from  \cite[Observation 2.12]{bangachev2023random}). As this is the correct behaviour at both ends, we believe that it is also correct for all $q,$ which leads to the following conjecture.

\begin{conjecture}
\label{conj:sigendfourcycleinlqp}
The signed four-cycle test distinguishes w. h. p. between 
$H_0: \ergraphhalf$ and
$H_1:\RGG(n,\dtorus,\unif,\sigma^q_{1/2},1/2)$ in the following regimes:
\begin{enumerate}
    \item When $q = o(d/\log d)$ and $dq = \omega(n^2).$ 
    \item When $q = \Omega(d/\log d)$ and $d = \tilde{\omega}(n).$
\end{enumerate}
\end{conjecture}

Similarly, we conjecture the performance of the signed-triangle statistic by extrapolating from behaviour at $q = 2$ and $q = \infty.$

\begin{conjecture}
\label{conj:sigendtriangleinlq}
When testing between
$H_0: \ergraphhalf$ and
$H_1:\RGG(n,\dtorus,\unif,\sigma^q_{1/2},1/2),$ when $dq^3 = \omega(n^3).$ 
the signed triangle test:
\begin{enumerate}
    \item Succeeds with high probability when $dq^3 = o(n^3)$ and fails with high probability when 
    $dq^3 = \omega(n^3)$ for $q = o(d/\log d).$
    \item Succeeds with high probability when $d = \tilde{o}(n^{3/4})$ and fails with high probability when 
    $d^3 = \tilde{\omega}(n^{3/4})$ for $q = \Omega(d/\log d).$
\end{enumerate}
\end{conjecture}
These conjectures can be summarized with the following diagram.

\begin{figure}[!htb]
\centering
     \includegraphics[width=0.4\linewidth]{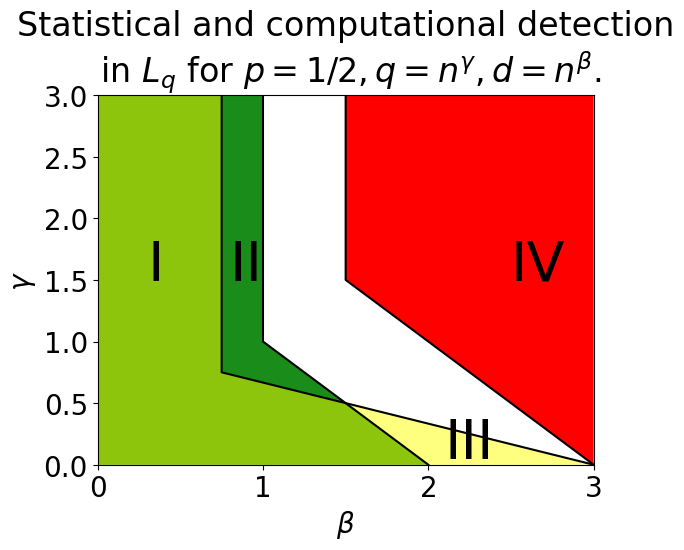}
     \caption{\footnotesize{Visualizing \cref{thm:torussmallq,thm:signedfourcyclesinlq,conj:sigendfourcycleinlqp} and \cref{,conj:sigendtriangleinlq}. $I + III$ is the conjectured region in which the signed triangle test solves \cref{eq:hypoithesistesting} for $\RGG(n,\dtorus, \unif, \sigma^q_{1/2},{1/2})$ with high probability. Region $I + II$
     is the conjectured region in which the signed four-cycle test succeeds with high probability. In $IV,$ it is information theoretically impossible to solve \cref{eq:hypoithesistesting} with high probability. 
     The last region is potentially suboptimal.}
     Interestingly, if these conjectures are correct, the signed 4-cycle statistic is always at least as good as the entropic upper bound \cref{thm:entropybound} unlike the signed 3-cycle statistic.
     }\label{Fig:lqdiagappendix}
\end{figure}

\paragraph{A Fourier-based Approach to Signed Cycle Counts.}
We end with a Fourier-based approach to computing the signed cycle counts for $\RGG(n,\dtorus, \unif, \sigma^q_p,p)$ (which extends to any random algebraic graph over $\dtorus$
).

We begin with some brief refresher on Fourier analysis over $\dtorus.$ Recall that we defined $\dtorus$ as a product of $d$ circles of circumference 2, or, equivalently, $\dtorus = \mathbf{R}^d/\sim,$ where $\bfx\sim\bfy$ if and only if $\bfx-\bfy\in 2\mathbb{Z}^d.$
Similarly to the Boolean case, we will use the fact that any $L_2$-integrable function $f:\dtorus\longrightarrow \mathbf{R}$ can be uniquely written as 
$f(\bfx) = \sum_{\bfv \in \mathbb{Z}^d}\widehat{f}(\bfv)\exp(i\pi \langle \bfv, \bfx\rangle).$ We make the following simple well-known observation. If $f$ satisfies $f(\bfx) = f(-\bfx)$ for all $\bfx,$ then each coefficient $\widehat{f}(\bfv)$ is real and, furthermore, $\widehat{f}(\bfv) =\widehat{f}(-\bfv).$ Indeed, this follows by uniqueness as 
$$
\sum_{\bfv}
\widehat{f}(\bfv)\exp(-i\pi \langle \bfv, \bfx\rangle)  = 
f(-\bfx) = 
f(\bfx) = 
\overline{f(\bfx)} = 
\sum_\bfv
\overline{\widehat{f}(\bfv)\exp(i\pi \langle \bfv, \bfx\rangle)} = 
\sum_\bfv
\overline{\widehat{f}(\bfv)}
\exp(-i\pi \langle \bfv, \bfx\rangle).
$$
Finally, recall that $\widehat{f}(\bfnull) = \int_\dtorus f(\bfx)d\bfx.$
Now, $\sigma^q_p$ is clearly $L_2$-integrable. Thus, for any  signed $k$-cycle weight, 
\begin{equation}
    \begin{split}
        & 
        \expect_{\bfG\sim\RGG(n,\dtorus,\unif, \sigma^q_p,p)}\Big[\signedweight_{C_k}(\bfG)\Big]\\
        & = \expect_{\bfg_1, \bfg_2, \cdots, \bfg_k\iidsim\unif(\dtorus)}
        \Bigg[\prod_{i = 1}^k (\sigma(\bfg_i - \bfg_{i+1}) - p)\Bigg]\\
        & = \int_{(\dtorus)^k}
        \prod_{i = 1}^k \sum_{\bfv\in \mathbb{Z}^d\backslash\bfnull}\widehat{\sigma}(\bfv)\exp(i \pi \langle \bfv, \bfg_i - \bfg_{i+1}\rangle)d\bfg_1d\bfg_2\cdots\bfg_k, \\
        & = \sum_{\bfv_1, \bfv_2,\cdots,\bfv_k\in \mathbb{Z}^d\backslash\bfnull}\int_{(\dtorus)^k}
        \widehat{\sigma}(\bfv_1)
        \widehat{\sigma}(\bfv_1)
        \cdots\widehat{\sigma}(\bfv_k)
        \exp(- i\pi \sum_{i = 1}^k \langle \bfv_i, \bfg_i - \bfg_{i+1}\rangle)\\
        & = \sum_{\bfv_1, \bfv_2,\cdots,\bfv_k\in \mathbb{Z}^d\backslash\bfnull}\int_{(\dtorus)^k}
        \widehat{\sigma}(\bfv_1)
        \widehat{\sigma}(\bfv_1)
        \cdots\widehat{\sigma}(\bfv_k)
        \exp(- i\pi \sum_{i = 1}^k \langle \bfg_i, \bfv_i - \bfv_{i-1}\rangle)\\
        & = 
        \sum_{\bfv\in \mathbb{Z}^d\backslash\bfnull}
        \widehat{\sigma}(\bfv)^k,
    \end{split}
\end{equation}
where the last line follows from the simple observation that if $\bfv_i \neq \bfv_{i-1}$ for some $i,$ the integral vanishes. It must be noted, however, that even if one manages to compute a signed cycle count, there still remains the obstacle of computing its variance.

\section{Comparison of \texorpdfstring{\cref{thm:hypercube}}{Hypercube Theorem} with Prior Work}
\label{appendix:twohypercuberesults} 
In \cite{bangachev2023random}, the authors prove the following theorem in the same setup as \cref{thm:hypercube}. 

\begin{theorem}[\cite{bangachev2023random}]
\label{thm:maintheoremindistingishability}
Consider a dimension $d\in \mathbb{N},$ connection $\sigma:\{\pm 1 \}^d\longrightarrow [0,1]$ with expectation $p,$ and absolute constant 
$m\in \mathbb{N}.$ There exists a constant $K_m$ depending only on $m,$ but not on $\sigma, d, n,p,$ with the following property.
Suppose that $n \in \mathbb{N}$ is such that $nK_m < d.$
For $1\le i \le d,$ let\linebreak $ B_i = \max \B\{ |\widehat{\sigma}(S)|\binom{d}{i}^{1/2} : 
|S| = i\B\}
.$ Denote also
\begin{equation*}
C_m  = \sum_{i = m+1}^{ \frac{d}{2en}} B_i^2 + \sum_{i = d-\frac{d}{2en}}^{d-m-1} B_i^2\quadand
D   = \sum_{\frac{d}{2en}\le 
j
\le d - \frac{d}{2en}
} B_i^2.
\end{equation*}
If the following conditions additionally hold
\begin{itemize}
    \item $d \ge K_m\times n\times  \B(\frac{C_m}{p(1-p)}\B)^{\frac{2}{m+1}},$
    \item $d\ge K_m \times n \times \B(\frac{B^2_u}{p(1-p)}\B)^{\frac{2}{u}}$ for all $2\le u \le m,$
    \item $d\ge K_m \times n \times \B(\frac{B^2_{d-u}}{p(1-p)}\B)^{\frac{2}{u}}$ for all $2\le u \le m,$
\end{itemize}
then
\begin{align*}
        & \TV\Big( \hypercubegraph\| \ergraph\Big)^2\\ & \le  K_m\times \frac{n^3}{p^2(1-p)^2}\times 
\left(
\sum_{i = 1}^m
\frac{B_i^4}{d^i} + 
\sum_{i = d-m}^d
\frac{B_i^4}{d^i} 
+ \frac{C^2_m}{{d}^{m+1}} + 
{D^2}\times \exp\left( - \frac{d}{2en}\right)
\right).
\end{align*}
\end{theorem}

We make several remarks on the comparison between those two theorems. 

First, \cref{thm:hypercube} is much easier to apply than \cref{thm:maintheoremindistingishability} and its proof is substantially simpler. 

Furthermore, it can be applied in setting when $d = o(n).$ Thus, for example in \cite{bangachev2023random} prove the first part of \cref{cor:applicationshypercube} only when $d = \Omega(n),$ that is $r  = O(\sqrt{n}).$ 

Still, in many cases \cref{thm:maintheoremindistingishability} is much stronger. For example, consider the double threshold connection $\sigma(\bfg) = \indicator\big[|\sum_{i = 1}^d g_i| \ge \chi_d\big],$ where $\chi_d$ is chosen so that $\expect[\sigma] = 1/2.$ Then, \cref{thm:maintheoremindistingishability} implies that $\TV\Big(\RAG(n,\hypercube, \sigma,1/2), \ergraphhalf\Big) = o(1)$ whenever $d = \omega(n^{3/2})$ \cite[Corollary 4.10]{bangachev2023random}. However, 
\cref{thm:hypercube} only implies this for $d = \omega(n^3).$ The reason \cref{thm:hypercube} is much weaker in this setting is that the expression $\sum_{i=1}^d\infl^2_i[\sigma]$ puts a much larger weight on levels close to $d.$ Indeed, note that for $S\subset[d],$ the Fourier coefficient $\widehat{\sigma}(S)$ contributes to $|S|$ of the terms $\infl^2_i[\sigma],$ but it only contributes once to the expression $\sum_{i = d-m}^d
\frac{B_i^4}{d^i} $ from \cref{thm:maintheoremindistingishability}.

\section{On the Bound of Racz and Liu}
\label{appendix:onthekltensobound}

\paragraph{3-Term Arithmetic Progressions.}
Expanding the left-hand side of \cref{eq:Kltensorization}, we conclude that small (centered) moments of the self-convolution imply a certain randomness of $\sigma,$ respectively of 
$A\subseteq \Group$ when $\sigma(\bfg)\coloneqq \indicator[\bfg\in A].$ We note that the same notion of pseudorandomness was recently used by Kelley and Meka in their breakthrough paper \cite{kelley2023strong} on 3-term arithmetic progressions, in the case $\Group= \mathbf{F}_q^n$ (see also the exposition \cite{bloom2023kelleymeka}). One simplification in our setup is that $\sigma(\bfg) = \sigma(-\bfg)$ in the context of random algebraic graphs, so $\sigma*\sigma(\bfg)\coloneqq 
\expect_\bfh\sigma(\bfg-\bfh)\sigma(\bfh) = 
\expect_\bfh\sigma(\bfg+\bfh)\sigma(\bfh) =: 
\sigma\star\sigma(\bfg).
$

\paragraph{Quasi-Randomness.} The left-hand side of \cref{eq:Kltensorization} can be expanded either in terms of the moments of $\sigma*\sigma$ or in terms of the moments of $(\sigma-p)*(\sigma-p).$ 
However, one can easily observe that the $k$-th moment of $\expect[(\sigma*\sigma)^k]$ is exactly the probability that each edge of a fixed copy of $K_{2,t}$ appears in $\RAG(n,\Group, \sigma,p).$ In other words, one interpretation of \cref{eq:Kltensorization} is that if all subgraphs of the form $K_{2,t}$ appear with probability sufficiently close to $p^{2t}$ in $\RAG(n,\Group, \sigma,p),$ then 
$\RAG(n,\Group, \sigma,p)$ is (up to $o(1)$ total variation) the same as $\ergraph.$ This can be viewed as a certain analogue of the celebrated theorem due to Chung-Graham-Wilson \cite{chung87}. It  (very informally) states that if a graph simultaneously has a number of edges and 4-cycles close that of $\ergraph,$ then every other subgraph count is close to that of $\ergraph.$
Similarly, the $t$-th moment of 
$(\sigma-p)*(\sigma-p)/(p(1-p))$ is exactly the 
Fourier coefficient corresponding to $K_{2,t}$ and one can make an equivalent interpretation for signed copies of $K_{2,t}.$

\section{Anticoncentration of Convolutions and the Proof of \texorpdfstring{\cref{claim:lqanticoncentration}}{Lq Anticoncentration}}
\label{appendix:anticoncentration}

Suppose that $X$ is a real-valued random variable with density which is absolutely continuous with respect to the Lebesgue density on $\mathbb{R}.$ Denote by $M(X)\in \mathbb{R}_+\cup \{+\infty\}$ the maximum value of the density of $X.$ We will use the following fact from \cite{Bobkov14}.\footnote{The result in \cite{Bobkov14} is more general and holds for random variables taking values in any $\mathbb{R}^a.$}

\begin{theorem}
\label{thm:bobkovdensitybound}
 Suppose that $Y_1, Y_2, \ldots, Y_d$ are independent real random variables with densities absolutely continuous with respect to the Lebesgue measure. Then, 
 $$
 M^{-2}(Y_1 + Y_2 + \cdots + Y_d)\ge \frac{1}{e}\sum_{i = 1}^d M^{-2}(Y_i).
$$
\end{theorem}
In particular, when $Y_1, Y_2, \ldots, Y_d$ are iid, this implies that 
$M(Y_1 + Y_2 + \cdots + Y_d)\le \sqrt{\frac{e}{d}}M(Y_1).$ As already mentioned in \cref{sec:lqproof}, in the setup of \cref{claim:lqanticoncentration} $M(U^q) = +\infty$ when $q>1$ and, thus, we need to generalize \cref{thm:bobkovdensitybound}. We do so as follows.

\begin{lemma}
\label{lem:TVdensitylemma}
Suppose that $X$ is a real-valued random variable with the following property. There exists another random variable $Y$ such that
\begin{enumerate}
    \item $\TV(X,Y)= 1-p\in [0,1),$ and
    \item The density of $Y$ is absolutely continuous with respect to the Lebesgue measure on $\mathbb{R}$ and\linebreak $M(Y) =  m <\infty.$
\end{enumerate}
Let $d$ be an integer and let $X_1, X_2, \ldots, X_d$ be independent copies of $X.$ Then, there exists a random variable $Z_d$ on $\mathbb{R}$ such that 
\begin{enumerate}
    \item $\TV(X_1+X_2 + \cdots + X_d, Z_d)\le \exp( - dp/8),$ and
    \item The density of $Z_d$ is absolutely continuous with respect to the Lebesgue measure on $\mathbb{R}$ and\linebreak $M(Z_d)\le \sqrt{2e} \frac{m}{\sqrt{p^3d}}.$
\end{enumerate}
\end{lemma}
\begin{proof} We first introduce two notational conventions. 

If $(\mathcal{D}_i)_{i = 1}^n$ are probability 
distributions and $\vec{p}\in \mathbb{R}^n_{\ge 0}$ is a vector with weights with sum to $1,$ we define the mixture
$\sum_{i = 1}^n p_i\mathcal{D}_i$ as follows.
First, one takes $B\in [n]$ such that $\prob[B = i] = p_i.$ Then, one draws $Z\sim \mathcal{D}_B$ independently from $B.$

If $\mathcal{D},\mathcal{F}$ are real-valued probability distributions, denote by $\mathcal{D}*\mathcal{F}$ the distribution of $Y_\mathcal{D}+Y_\mathcal{F},$ where $Y_\mathcal{D},Y_\mathcal{F}$ are independent and $Y_\mathcal{D}\sim\mathcal{D}, Y_\mathcal{F}\sim\mathcal{F}.$

We will use the following trivial identity.
$$
\Bigg(\sum_{i = 1}^{n}p_i\mathcal{D}_i\Bigg) *
\Bigg(\sum_{j = 1}^{m}q_j\mathcal{F}_j\Bigg) = 
\sum_{1\le i \le n,1\le j \le m}
p_iq_j 
\mathcal{D}_i * 
\mathcal{F}_j.
$$

Now, we go back to \cref{lem:TVdensitylemma}. Consider such a random variable $X$ and let $Y$ be its corresponding random variable from the statement of the lemma. Consider an optimal coupling $(X',Y')$ of $X$ and $Y$ such that $X' = Y'$ with probability $p.$ Denote by $\distribution_{\neq}$ the distribution of $X'|X'\neq Y'$  and by $\distribution_=$ the distribution of $X'|X'= Y',$ which is the same as the distribution of $Y'|X' = Y'.$  Since $Y$ is absolutely continuous with respect to  the Lebesgue measure, so is $Y'|X'=Y'.$ Furthermore, the maximum value of the density of $\mathcal{D}_=$ is at most $mp^{-1}$ as $m$ is the maximum value of the density of $Y$ and $\prob[X' = Y'] = p.$ 

In particular, note that the distribution $\distribution$ of $X$ is the mixture  $(1-p)\distribution_{\neq} + p\distribution_=,$ where  $\distribution_=$ is absolutely continuous with respect to the Lebesgue measure and its density is bounded by $mp^{-1}.$ Therefore, the distribution of $X_1 + X_2 + \cdots + X_d$ is the mixture 
\begin{equation}
    \begin{split}
        &\sum_{k = 0}^d \binom{d}{k}p^{k}(1-p)^{d-k}
        (\distribution_=)^{* k}*
        (\distribution_{\neq})^{*(d-k)}\\ 
        & = \sum_{k <dp/2} \binom{d}{k}p^{k}(1-p)^{d-k}
        (\distribution_=)^{* k}*
        (\distribution_{\neq})^{*(d-k)} + 
        \sum_{k \ge dp/2}\binom{d}{k}p^{k}(1-p)^{d-k}
        (\distribution_=)^{* k}*
        (\distribution_{\neq})^{*(d-k)}.\\
    \end{split}
\end{equation}
We now show the following two facts. First, the weight on summands $k<dp/2$ is at most $\exp(- dp/8),$ which means that 
$X_1 + X_2 + \cdots + X_d$ is $\exp(- dp/8)$-close to the mixture\linebreak $\sum_{k \ge dp/2}\binom{d}{k}p^{k}(1-p)^{d-k}(\distribution_=)^{* k}*(\distribution_{\neq})^{*(d-k)}.$ On the other hand, the latter mixture is absolutely continuous with respect to the Lebesgue measure and has density bounded by $\sqrt{2e} \frac{m}{\sqrt{p^3d}}.$ We begin with the first part.

\begin{claim}
\label{claim:mixturenumberconcnetration}
$\displaystyle \sum_{k < dp/2}\binom{d}{k}p^{k}(1-p)^{d-k}\le \exp(- dp/8).$
\end{claim}
\begin{proof}
This is a trivial application of Chernoff bounds. Let $V_1, V_2, \ldots,V_d$ be iid $Bernoulli(p)$ random variables. Then, 
\begin{equation*}
\begin{split}
\sum_{k < dp/2}\binom{d}{k}p^{k}(1-p)^{d-k}
= 
\prob\left[
\sum_{i = 1}^d V_i<dp (1 - 1/2)
\right]\le 
\exp(-{dp}/{8}).
\end{split}
\end{equation*}
\end{proof}

\begin{claim}
\label{claim:convolutionbounds}
For each $k\ge dp,$ $M\bigg((\distribution_=)^{* k}*
        (\distribution_{\neq})^{*(d-k)}\bigg)\le \sqrt{2e}\frac{m}{\sqrt{dp^3}}.$
\end{claim}
\begin{proof} Note that the density of $\distribution_=$ is at most $mp^{-1}$ as discussed. Therefore, by \cref{thm:bobkovdensitybound}, we immediately obtain
$$
M((\distribution_=)^{* k})\le 
\sqrt{\frac{e}{k}}mp^{-1}\le 
\sqrt{\frac{2e}{dp}}mp^{-1}.
$$
This is enough since 
$$
M\bigg((\distribution_=)^{* k}*
        (\distribution_{\neq})^{*(d-k)}\bigg)\le
M((\distribution_=)^{* k})\le 
\sqrt{2e}\frac{m}{\sqrt{dp^3}}.
$$
\end{proof}
Now let $\mathcal{D}_{<}$ be an arbitrary random variable on $\mathbb{R}$ with maximal density at most $\sqrt{2e}\frac{m}{p\sqrt{dp}}.$ Consider $Z$ distributed according to 
$$
Z\sim 
\left(
\sum_{k \ge dp/2}\binom{d}{k}p^{k}(1-p)^{d-k}
\right)\distribution_< + 
\sum_{k < dp/2}\binom{d}{k}p^{k}(1-p)^{d-k}
        (\distribution_=)^{* k}*
        (\distribution_{\neq})^{*(d-k)}.
$$
\cref{claim:mixturenumberconcnetration} implies $\TV(X_1 + X_2 +\cdots + X_d, Z) \le \exp(- dp/8).$ 
\cref{claim:convolutionbounds} implies that 
$M(Z)\le \sqrt{2e}\frac{m}{\sqrt{dp^3}}.$
\end{proof}

An immediate corollary of \cref{lem:TVdensitylemma} is the following small-ball probability bound which we will use to prove \cref{claim:lqanticoncentration}.

\begin{corollary}
\label{cor:smallball}
Suppose that $X$  is a non-negative real-valued random variable that is absolutely continuous with respect to the Lebesgue density with pdf $f.$ Let $d\in \mathbb{N}$ and $p \in (0,1]$ be such that $d>p^{-1}.$ Let $m$ be such that 
$\int_{\{f(x)> m\}}f(y)dy = 1-p.$ Then, for any interval $[a,b]\subseteq\mathbb{R},$ if $X_1, X_2, \ldots, X_d$ are independent copies of $X,$
$$
\prob\big[X_1 + X_2 + \cdots + X_d \in [a,b]\big]\le 
\exp(- dp/8) + \sqrt{2e} \frac{m}{\sqrt{p^3d}}(b-a).
$$
\end{corollary}
\begin{proof}
Let $\Omega  = \{x\ge 0 \; :\;f(x)\le  m\}.$ 
Clearly, $\int_\Omega f(x)dx = p.$
Let $Y$ be the real-valued random variable with density $f(x)$ for $x\in \Omega$ and density equal to $m$ on $[-(1-p)/m,0].$ Then, the density of $Y$ is bounded by $m$ and $\TV(Y, X) = 1-p$ (as the two densities agree on $\Omega$ which has measure $p$). Now, we simply find the random variable $Z_d$ given by \cref{lem:TVdensitylemma} and observe that for an optimal coupling of $Z_d, X_1 +X_2 + \cdots + X_d,$ we have 
$$
\prob\big[X_1 + X_2 + \cdots + X_d \in [a,b]\big]\le 
\prob[X_1 + X_2 + \cdots +X_d\neq Z_d] + 
\prob[Z_d\in [a,b]]\le
\exp(- dp/8) + M(Z_d)(b-a),
$$
from which the claim follows.
\end{proof}

\begin{proof}[Proof of \cref{prop:inervalprobsmallq}]
We apply \cref{cor:smallball} as follows. Consider the random variable $U^q,$ where $U \sim \unif([0,1]).$ The CDF $\phi(x)$ of $U^q$ for $x\in [0,1]$ is 
$$
\phi(x) = 
\prob[U^q\le x] = 
\prob[U\le x^{1/q}] = 
x^{1/q}.
$$
Thus, the density $h(x)$ of $U^q$ is $h(x) = (x^{1/q})' = \frac{1}{q}x^{1/q-1}\indicator\big[x\in (0,1]\big].$ Now, observe that $h(x)\le \frac{1}{q}(1/2)^{1/q-1}\le \frac{2}{q}$ for $x\in [1/2,1]$ and also 
$$
p\coloneqq \prob\big[U^q\in [1/2, 1]\big] = 
1 - \prob\big[U^q\le 1/2\big] = 
1- (1/2)^{1/q}\ge 
\frac{1}{2q}.
$$
Thus, applying \cref{cor:smallball} with $m = \frac{1}{q}(1/2)^{1/q-1} \le \frac{2}{q}, p \ge \frac{1}{2q}$ gives the result.
\end{proof}

\begin{proof}[Proof of \cref{claim:lqanticoncentration}]
Using \cref{prop:inervalprobsmallq},
$$
\psi(\ell) =\prob\big[U^q_1 + U^q_2 + \cdots + U^q_{d-1} \in [\tau^q,\tau^q - \ell^q]\big]\le 
\exp(- \Omega(d/q)) + O\left(\sqrt{\frac{q}{d}}\ell^q\right).
$$
Using that $d/q = \omega(\log d)$ and integrating over $[0,1],$ we conclude 
\begin{equation}
    \begin{split}
        \int_0^1 \psi(\ell)d\ell  = 
        \exp(- \Omega(d/q)) + 
        O\left(\sqrt{\frac{q}{d}}\int_0^1\ell^qd\ell\right) = 
        \exp(- \Omega(d/q)) + 
        O\left(\sqrt{\frac{1}{qd}}\right) = 
        O\left(\sqrt{\frac{1}{qd}}\right).
    \end{split}
\end{equation}
Similarly, 
\begin{equation}
    \begin{split}
        &\int_0^1 \psi^2(\ell)d\ell  = 
        \int_0^1
        \left(\exp(- \Omega(d/q)) + O\left(\sqrt{\frac{q}{d}}\ell^q\right)\right)^2d\ell\\
        & = 
        O(\exp(- \Omega(d/q))) + 
        O\left({\frac{q}{d}}\int_0^1\ell^{2q}d\ell\right) = 
        O\left(\frac{1}{d}\right).
    \end{split}
\end{equation}
\end{proof}
\section{Finishing the Proof of 
\texorpdfstring{\cref{claim:explicitconvolutionmoments}}{Explicit Convolution Moments}}
\label{appendix:linftystatscalculation}
\label{sec:lq}
All that was left to show is that 
whenever $t = o(d/(\log d)^2),$
$$
\Big(
    \frac{2}{t+1}(1-\lambda)^{t+1} + 
    \frac{t-1}{t+1}(1-2\lambda)^{t+1} 
    \Big)^d = 
p^{2t}(1 + \Theta(d\lambda^3t^2)).   
$$

We expand the brackets on the left-hand side as follows.
\begin{equation}
    \begin{split}
    &\frac{2}{t+1}(1-\lambda)^{t+1} + 
    \frac{t-1}{t+1}(1-2\lambda)^{t+1}  = \\
    & 
    \sum_{i = 0}^{t+1}\binom{t+1}{i}\frac{1}{t+1}(2(-\lambda)^i + (t-1)(-2\lambda)^i)
    =\\
    & 1 - 2t \lambda  + \binom{2t}{2}\lambda^2 - 
    \binom{2t}{3}\lambda^3 + \frac{2t(t-1)}{6}\lambda^3 + 
    \sum_{i = 4}^t \binom{t+1}{i}\Big(\frac{2(-\lambda)^i + (t-1)(-2\lambda)^i}{t+1}\Big).
    \end{split}
\end{equation}

We claim that the last expression equals $(1-\lambda)^{2t} + \frac{2t(t-1)}{6}\lambda^3(1 + o(1)).$ This is equivalent to proving that 
\begin{equation}
    \begin{split}
        \sum_{i = 4}^k\Bigg( \binom{t+1}{i}\Big(\frac{2(-\lambda)^i + (t-1)(-2\lambda)^i}{t+1}\Big) - 
        \binom{2t}{i}(-\lambda)^i\Bigg) = o(t^2\lambda^3).
    \end{split}
\end{equation}
We split the sum into two parts, $i \ge 4\log\frac{1}{\lambda}$ and $i < 4\log\frac{1}{\lambda}.$

\paragraph{Case 1) Large values of $i$.} We have 
\begin{equation}
    \begin{split}
        & \Bigg|\sum_{i \ge 4\log\frac{1}{\lambda}}^t\Bigg( \binom{t+1}{i}\Big(\frac{2(-\lambda)^i + (t-1)(-2\lambda)^i}{t+1}\Big) - 
        \binom{2t}{i}(-\lambda)^i\Bigg)\Bigg|\\
        & \le 
         \sum_{i \ge 4\log\frac{1}{\lambda}}^t\Bigg| \binom{t+1}{i}\Big(\frac{2(-\lambda)^i + (t-1)(-2\lambda)^i}{t+1}\Bigg)\Bigg| + \Bigg| 
        \binom{2t}{i}(-\lambda)^i\Bigg|\\
        & \le 2 \sum_{i \ge 4\log\frac{1}{\lambda}} 
        (3\lambda t)^i = 
        O((3\lambda t)^{4\log\frac{1}{\lambda}}) = 
        O(\lambda^4) = o(\lambda^3t^2),
    \end{split}
\end{equation}
where we used the fact that $\lambda t  = o(1).$

\paragraph{Case 2) Small values of $i$.} 
We bound the coefficient in front of $(-\lambda)^i$ as follows. 
\begin{equation}
    \begin{split}
        &\binom{t+1}{i}\frac{2 + 2^{i}(t-1)}{t+1} - \binom{2t}{i}\\
        & = 
        \frac{t(t-1)(t-2)\cdots (t - i+2)(2^i(t-1) + 2) - 2t(2t-1)(2t-2)\cdots (2t - i+1)}{i!}\\
        & = 
        \frac{2t(t-1)(t-2) \cdots (t-i+2)}{i!} + \\
        & \quad \frac{2t(2t-2)(2t-4)\cdots (2t-2i+4)(2t-2) - 
        2t(2t-1)(2t-2)\cdots (2t - i+1)
        }{i!}\\
        & =
        O(t^{i-1}) - 
        \frac{2t(2t-1)(2t-2)\cdots (2t - i+1)}{i!}\Big(1 - \frac{2t(2t-2)(2t-4)\cdots (2t-2i+4)(2t-2)}{2t(2t-1)(2t-2)\cdots (2t - i+1)}\Big).
    \end{split}
\end{equation}
Now, observe that 
$$
\frac{2t(2t-2)(2t-4)\cdots (2t-2i+4)(2t-2)}{2t(2t-1)(2t-2)\cdots (2t - i+1)}\ge 
\Big(\frac{2t-2i}{2t}\Big)^i = 
(1 - i/t)^i \ge 1 - i^2/t,
$$
where we used that $4\le i \le t$ and Bernoulli's inequality. Furthermore,
$$
\frac{2t(2t-2)(2t-4)\cdots (2t-2i+4)(2t-2)}{2t(2t-1)(2t-2)\cdots (2t - i+1)} = 
\frac{2t-2}{2t-1}
\prod_{j = 2}^{i-2}\frac{2t - 2j}{2t - j - 1}<1. 
$$
Hence, the desired sum is of order 
\begin{equation}
    \begin{split}
        O(t^{i-1})  - 
        O\Big(\frac{(2t)^i}{i!}\frac{i^2}{t}\Big) = O(t^{i-1}).
    \end{split}
\end{equation}
It follows that the sum in the small $i$ case is bounded by 
$$
\sum_{i = 4}^{4\log\frac{1}{\lambda}}O(t^{i-1}\lambda^i) = 
O(t^3\lambda^4) = o(t^2\lambda^3),
$$
where again we used $t\lambda = o(1).$

Altogether, using that $(1-2\lambda)^t \ge 1-2\lambda t = \Theta(1),$
\begin{equation}
\begin{split}
    & \expect[(\sigma*\sigma(\bfg))^t] = 
    \Big(
    (1-\lambda)^{2t} + \Theta(t^2\lambda^3)
    \Big)^d\\
    & =  
    \Big(
    (1-\lambda)^{2t}(1 + \Theta(t^2\lambda^3)
    \Big)^d\\
    & = 
    (1-\lambda)^{2dt}(1 + \Theta(t^2\lambda^3))^d\\
    & = p^{2t}(1 + \Theta(dt^2\lambda^3)), 
    \end{split}
\end{equation}
where again we used that $t = o(\lambda^{-1}),$ $dt^2\lambda^3 = o(1)$ and 
$(1-\lambda)^d = p$ by definition.

\section{Graphs Counts in the \texorpdfstring{$L_\infty$}{Sup-Norm} Model}
\label{appendix:linftygraphcounts}
\begin{proof}[Proof of \cref{prop:linftygraphcounts}]
\textbf{Case 1) The smallest cycle of $H$ is of even size.} From \cref{eq:pietruncated}, we have 
\begin{equation}
\begin{split}
\label{eq:pietruncatedapp}
        \sum_{A\subseteq E(H)\; : \;  |A|\le m+1}
        (-1)^{|E(A)|}
       \chi(A)\le 
        \expect_{\bfG\sim \RGG(n,\mathbb{T}^1, \unif, \sigma^\infty_{1-\lambda}, 1-\lambda)}[\unsignedweight_H(\bfG)] \le  
        \sum_{A\subseteq E(H)\; : \;  |A|\le m}
        (-1)^{|E(A)|}
       \chi(A).
\end{split}
\end{equation}
First, consider the upper bound. Since each subgraph of $H$ on at most $m-1$ edges is acyclic and there are exactly $N(m)$ cycles on $m$ edges, from \cref{claim:1dtoruscounts}, 
\begin{align*}
&\sum_{A\subseteq H\; : \;  |A|\le m}
        (-1)^{|E(A)|}
       \chi(A) = 
\sum_{j = 0}^m (-1)^{j}\binom{|E(H)|}{j}\lambda^{j} + 
N(m)\times \Big(\lambda^{m-1}\phi(m-1) - \lambda^m\Big).\\
\end{align*}
Similarly, we can carry out the calculation for the lower bound in \cref{eq:pietruncatedapp}. Note that all $(m+1)$-edge subgraphs of $H$ have one of three structures: 1) Acyclic, in which case $\chi(A) = \lambda^{m+1},$ 2) An $m$ cycle with an extra edge not creating a cycle, in which case $\chi(A) = \lambda^{m}\phi(m-1),$ 3) An $m+1$ cycle in which case $\chi(A) = 0.$ In all three cases, importantly, 
$|\chi(A)|\le \lambda^{m}.$ Altogether, this means that 
\begin{equation}
\label{eq:linftyevencasesumapprox}
\begin{split}
    & \expect_{\bfG\sim \RGG(n,\mathbb{T}^1, \unif, \sigma^\infty_{1-\lambda}, 1-\lambda)}[\unsignedweight_H(\bfG)]\\
    & = 
    \sum_{j = 0}^m (-1)^{j}\binom{|E(H)|}{j}\lambda^{j} + 
N(m)\times \Big(\lambda^{m-1}\phi(m-1) - \lambda^m\Big)+
O\Bigg(\binom{|E(H)|}{m+1}\lambda^m
\Bigg).
\end{split}
\end{equation}
Again, using the truncated principle of inclusion-exclusion, 
$$
\sum_{j = 0}^m (-1)^{j}\binom{|E(H)|}{j}\lambda^{j}\ge 
(1-\lambda)^{|E(H)|}
\ge 
\sum_{j = 0}^{m+1} (-1)^{j}\binom{|E(H)|}{j}\lambda^{j},
$$
so $\sum_{j = 0}^m (-1)^{j}\binom{|E(H)|}{j}\lambda^{j} = 
(1-\lambda)^{|E(H)|} + O\bigg(\binom{|E(H)|}{m+1}\lambda^m
\bigg)
.$
Using \cref{eq:linftyevencasesumapprox},
\begin{align*}
& \expect_{\bfG\sim \RGG(n,\mathbb{T}^1, \unif, \sigma^\infty_{1-\lambda}, 1-\lambda)}[\unsignedweight_H(\bfG)]\\
& = 
    (1-\lambda)^{|E(H)|}+ 
N(m)\times \Big(\lambda^{m-1}\phi(m-1) - \lambda^m\Big)+
O\Bigg(\binom{|E(H)|}{m+1}\lambda^m
\Bigg)\\
& = 
(1-\lambda)^{|E(H)|}+ 
N(m)\times \lambda^{m-1}\phi(m-1) + 
O\Bigg(\binom{|E(H)|}{m+1}\lambda^m + 
\binom{|E(H)|}{m}\lambda^m
\Bigg),
\end{align*}
where we used the trivial observation that $N(m)\le \binom{|E(H)|}{m+1}.$ Now, using the simple fact that  that $\phi(m - 1) = \Theta(m^{-1/2})$ (see \cref{claim:1dtoruscounts}) and the assumption that $ |E(H)|^{m+1} = o(\lambda^{-1}),$ one can easily see that the last expression is of order
$$
(1-\lambda)^{|E(H)|}+ 
N(m)\times \lambda^{m-1}\phi(m-1)(1 + o(1)).
$$
Finally, 
\begin{align*}
    \begin{split}
    & \expect_{\bfG\sim \RGG(n,\dtorus, \unif, \sigma^\infty_p, p)}[\unsignedweight_H(\bfG)]\\
    & = 
    \Bigg(
    (1-\lambda)^{|E(H)|}+ 
N(m)\times \lambda^{m-1}\phi(m-1)(1 + o(1))
    \Bigg)^d\\
    & = 
    (1-\lambda)^{d\times |E(H)|}\times 
    \Bigg( 1+ 
\frac{N(m)\times \lambda^{m-1}\phi(m-1)(1 + o(1))}{(1-\lambda)^{|E(H)|}}
    \Bigg)^d.
    \end{split}
\end{align*}
Since $|E(H)|\le \lambda^{1/(m+1)},$ clearly $(1-\lambda)^{|E(H)|} = 1 + o(1).$ Furthermore, 
$|N(m)\lambda^{m-1}| = O(\binom{|E(H)|}{m}\lambda^{m-1}) = O(\lambda^{m-2}) = O(1/d)
$  as $\lambda  = O(\log n /d) = o(d^{-1/2}).$ Using also the fact that $(1-\lambda)^{d} = p,$ we conclude that 
$$
\expect_{\bfG\sim \RGG(n,\dtorus, \unif, \sigma^\infty_p, p)}[\unsignedweight_H(\bfG)] = 
p^{|E(H)|}\times 
    \Bigg( 1+ 
{d N(m)\phi(m-1)\lambda^{m-1}}\big(1 + o(1)\big)\Bigg).
$$

\noindent
\textbf{Case 2) The smallest cycle of $H$ is of even size.} We repeat the same steps as in the even case. The only difference is that when considering cycles of length $m+1,$ one needs to take extra care of cycles of length $m+1$ as $\chi(C_{m+1}) =\lambda^{m}\phi(m),$ which is of the same order as $\psi(C_m) = -\lambda^m.$ Namely, we have
\begin{equation}
    \begin{split}
        & \expect_{\bfG\sim \RGG(n,\mathbb{T}^1, \unif, \sigma^\infty_{1-\lambda}, 1-\lambda)}[\unsignedweight_H(\bfG)] \le  
        \sum_{A\subseteq E(H)\; : \;  |A|\le m+1}
        (-1)^{|E(A)|}
       \chi(A)\\
       & = 
       \sum_{j = 0}^{m+1}
       (-\lambda)^{j}\binom{|E(H)|}{m+1}
        + N(m)\lambda^m + 
        N(m+1)(\phi(m)\lambda^m - \lambda^{m+1}) + 
        O\Bigg(
        \lambda^{m+1}\binom{|E(H)|}{m+1}
        \Bigg).
    \end{split}
\end{equation}
Again, we used that all subgraphs of $H$ on at most $m$ edges are acyclic, except for $N(m)$ isomorphic to $C_m.$ The subgraphs on $m+1$ vertices
have one of three structures: 1) Acyclic, in which case $\chi(A) = \lambda^{m+1},$ 2) An $m$ cycle with an extra edge not creating a cycle, in which case $\chi(A) = 0,$ 3) An $m+1$ cycle in which case $\chi(A) = \phi(m)\lambda^m.$ 

Similarly, the subgraphs on $m+2$ vertices have one of four structures: 1) Acyclic, in which case $\chi(A) = \lambda^{m+2},$ 2) An $m$ cycle with two extra edges, in which case $\chi(A) = 0,$ 3) An $m+2$ cycle, in which case $\chi(A) = 0,$ 4) An $m+1$ cycle with an extra edge, in which case $\chi(A) = \phi(m)\lambda^{m+1}$. In all cases, $|\chi(A)|$ is at most $\lambda^{m+1}.$ Thus, 
\begin{align*}
& \expect_{\bfG\sim \RGG(n,\mathbb{T}^1, \unif, \sigma^\infty_{1-\lambda}, 1-\lambda)}[\unsignedweight_H(\bfG)]\\
& = 
\sum_{j = 0}^{m+1}
       (-\lambda)^{j}\binom{|E(H)|}{m+1}
        + N(m)\lambda^m + 
        N(m+1)\phi(m)\lambda^m + 
O\Bigg(
\binom{|E(H)|}{m+1}\lambda^{m+1} + 
\binom{|E(H)|}{m+2}\lambda^{m+1} 
\Bigg)\\
& = 
(1-\lambda)^{|E(H)|} + 
\Bigg(N(m)\lambda^m + N(m+1)\phi(m)\lambda^m\Bigg)\big(1 + o(1)\big).
\end{align*}
As in the even case, we used 
 $\sum_{j = 0}^{m+1} (-1)^{j}\binom{|E(H)|}{j}\lambda^{j} = 
(1-\lambda)^{|E(H)|} + O\bigg(\binom{|E(H)|}{m+2}\lambda^{m+1}
\bigg)
.$ The desired conclusion follows as in the even case.
\end{proof}
\section{Omitted Proofs from 
\texorpdfstring{\cref{sec:linftysignedcounts}}{Sigend Counts}}
\label{appendix:fromsignedsubgraphcounts}
\begin{proof}[Proof of \cref{claim:subadditivity}]
Let $G_1$ have $a$ connected components with vertex sets $D_1, D_2, \ldots, D_a$ and let $G_2$ have $b$ connected components with vertex sets $F_1, F_2, \ldots, F_b.$ Consider the bipartite graph $\mathcal{G}$ on parts $\mathcal{D}, \mathcal{F}$ with vertex sets respectively $D_1, D_2, \ldots, D_a$ and $F_1, F_2, \ldots, F_b.$ Draw an edge between $D_i$ and $F_j$ if and only if they have a common vertex and, if so, label this edge with one of their common vertices. Clearly, each edge is labeled by a different vertex.

Note that $\numconnectedcomponets(G) = \numconnectedcomponets(\mathcal{G}).$ On the other hand, as each edge is labelled by a different repeated vertex, $|V(G)|\le |V(G_1)| + |V(G_2)| - |E(\mathcal{G})|.$ Trivially, $$\numconnectedcomponets(\mathcal{G})\ge |V(\mathcal{G})| - |E(\mathcal{G})| = \numconnectedcomponets(G_1) + 
\numconnectedcomponets(G_2)
-  |E(\mathcal{G})|.
$$ 
Putting these together, we obtain
\[
|V(G_1)| + |V(G_2)| - |V(G)|\ge 
|E(\mathcal{G})|\ge 
\numconnectedcomponets(G_1) + 
\numconnectedcomponets(G_2) - 
\numconnectedcomponets(G).
\qedhere\]
\end{proof}

\begin{proof}[Proof of \cref{claim:numcintwoconnectedsubgraph}]
First, suppose that $\mathcal{K}$ is connected and $V(H) = V(\mathcal{K}).$ Then, the right-hand side of the desired inequality equals 0 and the left-hand side is non-negative.

Otherwise, let the connected components of $\mathcal{K}$ be $F_1, F_2, \ldots ,F_a,$ where $a = \numconnectedcomponets(\mathcal{K}).$  
Consider the multigraph (with multiedges, but no self-loops) $H'$ on $\numconnectedcomponets(K) + |V(H)| - V(\mathcal{K})$ vertices $[a]\cup \big(V(H)\backslash V(\mathcal{K})\big).$ In $H',$ two vertices in $V(H)\backslash V(\mathcal{K})$ are adjacent with multiplicity 1 if and only if they are adjacent in $H.$
The multiplicity of an edge between 
a connected component $F_j$ and a vertex $u\in \big(V(H)\backslash V(\mathcal{K})\big)$ equals the number of neighbours of $u$ in $F_j$ with respect to $H.$ Finally, the multiplicity between $F_j$ and $F_{j'}$ equals the number of edges between them in $H.$

Clearly, the number of edges (with multiplicities) in $H$ is at most $|E(H)| - |E(\mathcal{K})|.$ On the other hand, it must be at least $a + |V(H)\backslash V(\mathcal{K})| = \numconnectedcomponets(\mathcal{K}) + |V(H)| - |V(\mathcal{K})|.$ Indeed, otherwise there is a vertex of degree (counted with multiplicities) $1$ or $0$ in $H'.$ If it is of degree 0, clearly $H$ cannot be connected. If it is of degree 1, suppose that the corresponding edge in $H$ is $(u,v)$ and the vertex of degree 1 is either the vertex $u$ or a connected component $F_j$ containing $u.$ Since $H$ is $2$-connected, $v$ has at least one more neighbour $u_1$ other than $u$ in $H.$ In 
$H|_{V(H)\backslash \{v\}},$ there is no path between $u$ and $u_1.$ This is a contradiction with the 2-connectivity of $H.$ Thus, it must be the case that 
$|E(H)| - |E(\mathcal{K})|\ge \numconnectedcomponets(\mathcal{K}) + |V(H)| - |V(\mathcal{K})|.$
\end{proof}

\begin{proposition}
\label{prop:linftysignedboundcl1}
If $|V(H)|\ge 6$ and $\displaystyle i \le \frac{3(\log d)|V(H)|}{10(|E(H)| + \log d)},$ then
$(2i\lambda )^{|V(H)| - 1}2^{i|E(H)|}\le d^{- i - |V(H)|/2}$ for all large enough values of $d.$
\end{proposition}
\begin{proof}
It is enough to show that 
$(2i\lambda )^{|V(H)| - 1}e^{i|E(H)|}\le d^{- i - |V(H)|/2}.$ Taking a logarithm on both sides, this reduces to showing 
$$
i |E(H)| + i \log d + (|V(H)|/2)\log d\le 
(|V(H)| - 1)(\log d - O(\log \log d)),
$$
which is equivalent to 
$$
i \le \frac{(\log d) (|V(H)|/2 - 1)}{|E(H)| + \log d}(1 + o_d(1))\le
 \frac{3(\log d)|V(H)|}{10(|E(H)| + \log d)},
$$
where in the last line we used $|V(H)|>5.$
\end{proof}
\section{Omitted Proofs from \texorpdfstring{
\cref{sec:linftyperformance}}{Sup-Norm Performance}}
\label{appendix:linftyperformanceomitted}
\subsection*{Proof of \cref{eq:threecyclevars}}
Recall \cref{eq:linfty3varexpansion}.

We begin by calculating the variance for \ER. Note that whenever $(i,j,k)\neq (i',j',k'),$ there exists some edge in only one of the two triangles $\triangle(i,j,k), \triangle(i',j',k').$ Without loss of generality this is $(ij).$ Then,
$\expect_{\bfK\sim \ergraph}[\signedweight_{\triangle(i,j,k)}(\bfK)\signedweight_{\triangle(i',j',k')}(\bfK)] = 0$ as the product 
$\signedweight_{\triangle(i,j,k)}(\bfK)\signedweight_{\triangle(i',j',k')}(\bfK)$ contains the factor $\bfK_{(ij)} - p$ which is independent of everything else. Similarly,\linebreak 
$\expect_{\bfK\sim \ergraph}[\signedweight_{\triangle(i,j,k)}(\bfK)]\expect_{\bfK\sim \ergraph}[\signedweight_{\triangle(i',j',k')}(\bfK)] = 0.$
Thus, 
$$
\Var_{\bfK\sim \ergraph}[\signedweight_{C_3}(\bfK)] = 
\Theta(n^3)\times \Var_{\bfK\sim \ergraph}[\signedcount_{C_3}(\bfK)].
$$
We now use the following fact:
\begin{equation}
\label{eq:signedindicatorsqaure}
    \begin{split}
        \text{For any indicator $I,$ the equality $(I-p)^2 = (I-p)(1-2p) + (p-p^2)$ holds.}
    \end{split}
\end{equation}
We deduce that $\Var_{\bfK\sim \ergraph}[\signedcount_{C_3}(\bfK)] = \expect_{\bfK\sim \ergraph}[\signedcount_{C_3}(\bfK)^2] = 
\Theta(n^3)\times (p-p^2)^3 =\Theta(n^3p^3).
$

Now, we proceed to bounding $\Var_{\bfG\sim \RGG(n,\dtorus, \unif, \sigma^\infty_p,p)}[\signedweight_{C_3}(\bfG)].$ The idea is to split each term\linebreak
$\expect_{\bfG\sim \RGG}[\signedweight_{\triangle(i,j,k)}(\bfG)\signedweight_{\triangle(i',j',k')}(\bfG)]$ into a (weighted) sum of signed counts with respect to\linebreak $\RGG(n,\dtorus, \unif, \sigma^\infty_p,p)$ via \cref{eq:signedindicatorsqaure} and then apply \cref{thm:linftysignedcounts}.

\paragraph{1)} $\Var_{\bfG\sim\RGG}[\signedweight_{\triangle(1,2,3)}(\bfG)] = \expect_{\bfG\sim\RGG}[\signedweight_{\triangle(1,2,3)}(\bfG)^2] - 
\expect_{\bfG\sim\RGG}[\signedweight_{\triangle(1,2,3)}(\bfG)]^2.
$
By \cref{cor:linftysignedcycles}, \linebreak $\expect_{\bfG\sim\RGG}[\signedweight_{\triangle(1,2,3)}(\bfG)]^2 = 
\tilde{\Theta}(p^6/d^4) = o(p^3).$  On the other hand, 
$$\expect_{\bfG\sim\RGG}[\signedweight_{\triangle(1,2,3)}(\bfG)^2] = 
\expect_{\bfG\sim\RGG}[(\bfG_{(12) - p})^2(\bfG_{(23) - p})^2(\bfG_{(13) - p})^2]
.$$ 
Using \cref{eq:signedindicatorsqaure}, this is equal to 
$(p-p^2)^3 + (1-2p)^3\expect_{\bfG\sim\RGG}[\signedweight_{\triangle(1,2,3)}(\bfG)]$ and some terms with only one or two factors of the form $\bfG_{ij} - p.$ By \cref{rmk:onragfactorization}, those terms vanish as they form a graph with a leaf. Thus, the result is of order $\Theta(p^3) + \tilde{\Theta}((1-2p)^2p^3/d^2) = \Theta(p^3).$

\paragraph{2)} $\Cov[\signedweight_{\triangle(1,2,3)}(\bfG), 
    \signedweight_{\triangle(1,2,4)}(\bfG)] = 
    \expect[\signedweight_{\triangle(1,2,3)}(\bfG) 
    \signedweight_{\triangle(1,2,4)}(\bfG)] - \tilde{\Theta}(p^6/d^4)
    $ by \cref{cor:linftysignedcycles}. However, using \cref{eq:signedindicatorsqaure},
    \begin{align*}
        &\expect_{\bfG\sim\RGG}[\signedweight_{\triangle(1,2,3)}(\bfG) 
    \signedweight_{\triangle(1,2,4)}(\bfG)]\\
    & = 
    (p-p^2)\expect_{\bfG\sim\RGG}[(\bfG_{(13)} - p)(\bfG_{(23)} - p)(\bfG_{(14)} - p)(\bfG_{(24)} - p)]\\
    &\quad \quad + 
    (1-2p)
    \expect_{\bfG\sim\RGG}[(\bfG_{(12)}-p)(\bfG_{(13)} - p)(\bfG_{(23)} - p)(\bfG_{(14)} - p)(\bfG_{(24)} - p)].
    \end{align*}
 Both summands correspond to the signed weights of graphs on at most $4$ vertices. By \cref{thm:linftysignedcounts}, the last expression is of order $\tilde{O}(p\times p^4/d^2) = \tilde{O}(p^5/d^2).$
\paragraph{3)} In the cases of $\Cov[\signedweight_{\triangle(1,2,3)}(\bfH), 
    \signedweight_{\triangle(1,4,5)}(\bfH)]$ and $\Cov[\signedweight_{\triangle(1,2,3)}(\bfH), 
    \signedweight_{\triangle(4,5,6)}(\bfH)]],$ the two graphs $\triangle(i,j,k)$ and $\triangle(i',j',k')$ share at most one vertex, so their (signed) weights are independent by \cref{rmk:onragfactorization} and the covariance is zero. 

Combining those estimates via \cref{eq:linfty3varexpansion}, the variance is of order 
$\Theta(n^3p^3) + \tilde{O}(n^4p^5/d^2).$

\subsection*{Proof of \cref{eq:fourcyclevars}}
The estimate for \ER holds in the same way as in the proof of \cref{eq:threecyclevars}. We now estimate each of the terms in \cref{eq:linfty4varexpansion} for $\bfG\sim \RGG(n,\dtorus, \unif, \sigma^\infty_p,p).$

\paragraph{1)} As in the case fro triangles, we estimate
\begin{align*}
    &\Var_{\bfG\sim\RGG}[\signedweight_{\square(1,2,3,4)}(\bfG)]\\
    & = 
    \expect_{\bfG\sim\RGG}[(\bfG_{(12)} - p)^2(\bfG_{(23)} - p)^2(\bfG_{(34)} - p)^2(\bfG_{(24)} - p)^2] - \tilde{\Theta}(p^8/d^4)\\
    & = 
    (p-p^2)^4 + 
    (1-2p)^4
    \expect_{\bfG\sim\RGG}[(\bfG_{(12)} - p)(\bfG_{(23)} - p)(\bfG_{(34)} - p)(\bfG_{(24)} - p)] + \tilde{\Theta}(p^8/d^4)\\
    & = 
    (p-p^2)^4 + 
    (1-2p)^4\tilde{\Theta}(p^4/d^2)+ \tilde{\Theta}(p^8/d^4) = \Theta(p^4).
\end{align*}

\paragraph{2)} Similarly, using \cref{eq:signedindicatorsqaure}
\begin{align*}
    &\Cov_{\bfG\sim\RGG}[\signedweight_{\square(1,2,3,4)}(\bfG), 
    \signedweight_{\square(1,2,3,5)}(\bfG)]\\
    & = 
    \expect_{\bfG\sim\RGG}[
    (\bfG_{(12)} - p)^2
    (\bfG_{(23)} - p)^2
    (\bfG_{(34)} - p)
    (\bfG_{(14)} - p)
    (\bfG_{(35)} - p)
    (\bfG_{(15)} - p)
    ] - \tilde{\Theta}(p^8/d^4)\\
    & = 
    (p-p^2)^2
    \expect_{\bfG\sim\RGG}[
    (\bfG_{(34)} - p)
    (\bfG_{(14)} - p)
    (\bfG_{(35)} - p)
    (\bfG_{(15)} - p)
    ]\\
    & \quad \quad+ 2(p-p^2)(1-2p)
    \expect_{\bfG\sim\RGG}[
    (\bfG_{(23)} - p)
    (\bfG_{(34)} - p)
    (\bfG_{(14)} - p)
    (\bfG_{(35)} - p)
    (\bfG_{(15)} - p)
    ]\\
    & \quad \quad+ (1-2p)^2
    \expect_{\bfG\sim\RGG}[
    (\bfG_{(12)} - p)
    (\bfG_{(23)} - p)
    (\bfG_{(34)} - p)
    (\bfG_{(14)} - p)
    (\bfG_{(35)} - p)
    (\bfG_{(15)} - p)
    ] - 
    \tilde{\Theta}(p^8/d^4)\\
    & = 
    \tilde{O}(p^6/d^2) + 
    \tilde{O}(p^6/d^{5/2}) + 
    \tilde{O}(p^6/d^3) - 
    \tilde{\Theta}(p^8/d^4) = 
    \tilde{O}(p^6/d^2).
\end{align*}
In the second to last line, we used \cref{thm:linftysignedcounts} for each for the corresponding graphs.

\paragraph{3)} 
\begin{align*}
    &\Cov_{\bfG\sim\RGG}[\signedweight_{\square(1,2,3,4)}(\bfG), 
    \signedweight_{\square(1,2,4,5)}(\bfG)]\\
    & = 
    \expect_{\bfG\sim\RGG}[
    (\bfG_{(12)} - p)^2
    (\bfG_{(23)} - p)
    (\bfG_{(34)} - p)
    (\bfG_{(14)} - p)
    (\bfG_{(35)} - p)
    (\bfG_{(15)} - p)
    (\bfG_{(25)} - p)
    ] - \tilde{\Theta}(p^8/d^4)\\
    & = 
    (p-p^2)
    \expect_{\bfG\sim\RGG}[
    (\bfG_{(23)} - p)
    (\bfG_{(34)} - p)
    (\bfG_{(14)} - p)
    (\bfG_{(35)} - p)
    (\bfG_{(15)} - p)
    (\bfG_{(25)} - p)
    ]\\
    & \quad \quad+ (1-2p)
    \expect_{\bfG\sim\RGG}[
     (\bfG_{(12)} - p)
    (\bfG_{(23)} - p)
    (\bfG_{(34)} - p)
    (\bfG_{(14)} - p)
    (\bfG_{(35)} - p)
    (\bfG_{(15)} - p)
    (\bfG_{(25)} - p)
    ]\\
    &\quad\quad- 
    \tilde{\Theta}(p^8/d^4)\\
    & = 
    \tilde{O}(p^7/d^{5/2}) + 
    \tilde{O}(p^7/d^{5/2}) - 
    \tilde{\Theta}(p^8/d^4)=
    \tilde{O}(p^7/d^{5/2}) = 
    \tilde{O}(p^6/d^2).
\end{align*}

\paragraph{4)}
\begin{align*}
    &\Cov_{\bfG\sim\RGG}[\signedweight_{\square(1,2,3,4)}(\bfG), 
    \signedweight_{\square(1,2,5,6)}(\bfG)]\\
    & = 
    \expect_{\bfG\sim\RGG}[
    (\bfG_{(12)} - p)^2
    (\bfG_{(23)} - p)
    (\bfG_{(34)} - p)
    (\bfG_{(14)} - p)
    (\bfG_{(25)} - p)
    (\bfG_{(56)} - p)
    (\bfG_{(61)} - p)
    ] - \tilde{\Theta}(p^8/d^4)\\
    & = 
    (p-p^2)
    \expect_{\bfG\sim\RGG}[
    (\bfG_{(23)} - p)
    (\bfG_{(34)} - p)
    (\bfG_{(14)} - p)
    (\bfG_{(25)} - p)
    (\bfG_{(56)} - p)
    (\bfG_{(61)} - p)
    ]\\
    & \quad \quad+ (1-2p)
    \expect_{\bfG\sim\RGG}[
     (\bfG_{(12)} - p)
    (\bfG_{(23)} - p)
    (\bfG_{(34)} - p)
    (\bfG_{(14)} - p)
    (\bfG_{(25)} - p)
    (\bfG_{(56)} - p)
    (\bfG_{(61)} - p)
    ]\\
    &\quad\quad- 
    \tilde{\Theta}(p^8/d^4)\\
    & = 
    \tilde{O}(p^7/d^{3}) + 
    \tilde{O}(p^7/d^{3}) - 
    \tilde{\Theta}(p^8/d^4)=
    \tilde{O}(p^7/d^{3}) = 
    \tilde{O}(p^7/d^3)
\end{align*}

\paragraph{5)}
\begin{align*}
    &\Cov_{\bfG\sim\RGG}[\signedweight_{\square(1,2,3,4)}(\bfG), 
    \signedweight_{\square(1,5,2,6)}(\bfG)]\\
    & = 
    \expect_{\bfG\sim\RGG}[
    (\bfG_{(12)} - p)
    (\bfG_{(23)} - p)
    (\bfG_{(34)} - p)
    (\bfG_{(14)} - p)
    (\bfG_{(15)} - p)
    (\bfG_{(52)} - p)
    (\bfG_{(26)} - p)
    (\bfG_{(61)} - p)
    ]\\
    &\quad\quad
    - \tilde{\Theta}(p^8/d^4)\\
    & = 
    \tilde{O}(p^8/d^{3}) + 
    \tilde{\Theta}(p^8/d^4)=
    \tilde{O}(p^8/d^3) = \tilde{O}(p^7/d^3)
\end{align*}

\paragraph{6)}
\begin{align*}
    &\Cov_{\bfG\sim\RGG}[\signedweight_{\square(1,2,3,4)}(\bfG), 
    \signedweight_{\square(1,5,3,6)}(\bfG)]\\
    & = 
    \expect_{\bfG\sim\RGG}[
    (\bfG_{(12)} - p)
    (\bfG_{(23)} - p)
    (\bfG_{(34)} - p)
    (\bfG_{(14)} - p)
    (\bfG_{(15)} - p)
    (\bfG_{(53)} - p)
    (\bfG_{(36)} - p)
    (\bfG_{(61)} - p)
    ]\\
    &\quad\quad
    - \tilde{\Theta}(p^8/d^4)\\
    & = 
    \tilde{O}(p^8/d^{3}) + 
    \tilde{\Theta}(p^8/d^4)=
    \tilde{O}(p^8/d^3) = \tilde{O}(p^7/d^3)
\end{align*}

\paragraph{7)} In the cases of $\Cov[\signedweight_{\square(1,2,3,4)}(\bfH), 
    \signedweight_{\square(1,5,6,7)}(\bfH)]$ and $\Cov[\signedweight_{\square(1,2,3,4)}(\bfH), 
    \signedweight_{\square(5,6,7,8)}(\bfH)]],$ the two graphs $\square(i,j,k,\ell)$ and $\square(i',j',k',\ell')$ share at most one vertex, so their (signed) weights are independent by \cref{rmk:onragfactorization} and the covariance is zero. 

Combining those estimates via \cref{eq:linfty4varexpansion}, the variance is of order 
$\Theta(n^4p^4) + \tilde{O}(n^5p^6/d^2 + n^6p^7/d^3).$

\subsection*{Proof of \cref{prop:exactlambda}}
We know that $\lambda^\infty_p$ satisfies 
$(1-\lambda^\infty_p)^d = p.$ Thus, $$\lambda^\infty_p = 1 - p^{1/d} = 1 - \exp(\log p/d) = 1 - \exp(-(\log 1/p)/d).
$$ 
By the usual Taylor series expansion, 
$\exp(-x) = 1 - x + x^2/2 + O(x^3)$ when $x = o(1).$ Similarly, 
$\lambda^\infty_p/(1-\lambda^\infty_p) = p^{-1/d} - 1$ and we argue in the same way.

\end{document}